\documentclass[10pt, reqno]{amsart}
\usepackage{amsmath}
\usepackage{cases}
\usepackage{mathrsfs}
\usepackage{cite}
\usepackage{bbm}
\usepackage{amscd}
\usepackage{amsfonts,latexsym,amsmath,amsthm,amsxtra,mathdots,amssymb,latexsym,mathabx}
\usepackage[all,cmtip]{xy}
\RequirePackage{amsmath} \RequirePackage{amssymb}
\usepackage{color}
\usepackage{colordvi}
\usepackage{multicol}
\usepackage{hyperref}
\usepackage{mathtools}
\usepackage{soul} %for strikethrough
\usepackage[margin=1.1in]{geometry}
\usepackage{xcolor}
\usepackage{tikz}

\hypersetup{
    colorlinks,
    linkcolor={red!50!black},
    citecolor={blue!50!black},
    urlcolor={blue!80!black}
}

\emergencystretch=\maxdimen
    \newcommand{\BC}{{\mathbb {C}}}

     \newcommand{\BR}{{\mathbb {R}}}

     \newcommand{\BZ}{{\mathbb {Z}}}

     \newcommand{\CF}{{\mathcal {F}}}
     
    \newcommand{\CI}{{\mathcal {I}}}

\newcommand{\R}{\mbox{$\mathbb R$}}	% For Integers

\def\-{^{-1}}
\def\sumx{\ \sideset{}{^\star}\sum}

\newtheorem{Theorem}{Theorem}[section]
\newtheorem{Lemma}[Theorem]{Lemma}

\newtheorem{Remark}[Theorem]{Remark}
\newtheorem{Definition}[Theorem]{Definition}
\newcommand{\sumstar}{\ \sideset{}{^\star}\sum}

\let\underbrace\LaTeXunderbrace

\begin{document}

\title{Short second moment bound and Subconvexity for $\rm GL(3)$ $L$-functions}
 \author{Keshav Aggarwal, Wing Hong Leung and Ritabrata Munshi}
\date{}
\address{Department of Mathematics, Indian Institute of Technology Bombay\\ Mumbai 400076, India}
\email{keshav@math.iitb.ac.in}

\address{Department of Mathematics, Rutgers University\\ 110 Frelinghuysen Road\\
		Piscataway, New Jersey 08854, USA}
	\email{joseph.leung@rutgers.edu}
	
\address{Stat-Math Unit, Indian Statistical Institute, 203 B.T. Road, Kolkata 700108, India
}
\email{ritabratamunshi@gmail.com}

\begin{abstract}
Let $\pi$ be a Hecke cusp form for $\rm SL_3(\BZ)$. We bound the second moment average of $L(s,\pi)$ over a short interval to obtain the subconvexity estimate
$$ L(1/2+it, \pi) \ll_{\pi, \varepsilon} (1+|t|)^{3/4-1/8+\varepsilon}. $$
This is the first time a short second moment has been used to obtain a subconvexity bound in a higher rank group.
\end{abstract}
\maketitle

\tableofcontents

\section{Introduction and statement of result}

The problem of bounding automorphic $L$-functions on the central line $\text{Re}(s)=1/2$, has enthralled researchers for over a century. The first subconvex bound was announced by Littlewood in a proceedings of the London Mathematical Society in 1921. Adopting a method of Weyl for bounding exponential sums, Hardy and Littlewood proved that
\begin{align*}
    \zeta(\tfrac{1}{2}+it)\ll_\varepsilon (1+|t|)^{\frac{1}{6}+\varepsilon},
\end{align*}
for any $\varepsilon>0$. In contrast the trivial (convexity) bound, a simple consequence of the functional equation of the zeta function, yields the exponent $1/4$ in place of $1/6$. Complete details of the proof of Hardy-Littlewood's result appeared only later in a paper of Landau \cite{landau}. The exponent $1/6$, which is $2/3$-rd of the convexity bound, is `fundamental', and can be achieved using other methods, e.g. short moment computation of Iwaniec \cite{Iwaniec1979} and high moment computation of Heath-Brown \cite{HB-twelfth-moment}. To go beyond $1/6$, the so called Weyl bound, one needs deeper understanding of short exponential sums, as was done systematically by van der Corput and others \cite{graham_kolesnik_1991}. Accordingly the Weyl bound has been breached, albeit mildly, dozens of times in the last hundred years, and the present record is due to Bourgain \cite{bourgain} which states that
\begin{align*}
    \zeta(\tfrac{1}{2}+it)\ll_\varepsilon (1+|t|)^{\frac{1}{6}-\frac{1}{84}+\varepsilon}.
\end{align*}
The interest in such improvements lies not in the bound itself, but in the method that leads to it. However the strength of the bound documents the strength of the method, and as such is an important marker. \\

The classical result of Hardy-Littlewood-Weyl was first extended to degree two $L$-functions by Good \cite{Good} in 1980's. For any holomorphic form $f$ of full level, Good established the bound
\begin{align*}
    L(\tfrac{1}{2}+it,f)\ll_{f,\varepsilon} (1+|t|)^{\frac{1}{3}+\varepsilon}.
\end{align*}
Since this $L$-function can be compared with the square of the Riemann zeta function, one notices that the above exponent matches with the Weyl exponent. Later the same bound was established for the Maass forms by Meurmann \cite{Meurman}. Around the same time Jutila \cite{Jut87} proved the same result using more elementary means based on his study of exponential sums twisted by $GL(2)$ Fourier coefficients. In a recent work, Booker et al \cite{Bo-Mi-Ng} have established Weyl bound for $L$-functions of holomorphic cusp forms for $\Gamma_1(N)$ for any level $N$. This was extended to degree two $L$-functions, of any level, nebentypus and spectral parameter by the first author \cite{Agg2018published}. However any improvement over this bound is yet to be achieved. The recent work \cite{Ho-Mu-Qi-2021arxiv} of the third author together with Holowinsky and Qi, makes some progress towards this important problem by establishing cancellation in very short sums (beyond the Weyl range). Unfortunately their work does not yield anything new in the generic length (square root of the conductor), and hence one does not get a sub-Weyl bound for the $L$-function.\\

The $t$-aspect subconvexity problem for degree three $L$-functions remained an important open problem for a long time, till the breakthrough work of Li \cite{Li1}, where she adopted the method of Conrey and Iwaniec \cite{Conrey-Iwaniec} to establish subconvex bound for symmetric square (self-dual) $L$-functions
\begin{align*}
    L(\tfrac{1}{2}+it, \mathrm{Sym}^2 f)\ll_{f,\varepsilon} (1+|t|)^{\frac{11}{16}+\varepsilon}.
\end{align*}
Since non-negativity of central $L$-values play a crucial role in the Conrey-Iwaniec approach, it seems that Li's method can not be extended beyond self-dual forms. However the bound has been improved substantially. First McKee, Sun and Ye \cite{McKee-Sun-Ye} established the exponent $2/3$, using the techniques of Li, but applying finer tools to deal with the exponential integrals. This was further improved by Nunes \cite{Nunes2017} to $5/8$, which doubles the initial saving given by Li. Qi \cite{Qi20} extended the result to arbitrary number fields with the same exponent of $5/8$. A recent preprint of Lin, Nunes and Qi \cite{Lin-Nunes-Qi-arxiv} establishes the much stronger exponent $3/5$ over the rationals, which is most likely the limit of the method in this set-up, but still falls short of the Weyl exponent $1/2$. Hybrid bounds have been obtained by Young \cite{Young-second-moment} and Khan-Young \cite{khan_young_2021} by estimating the second moment of $\rm GL_3\times GL_2$ $L$-functions. In the case of general non self-dual degree three $L$-functions $L(s,\pi)$, the first subconvex bound was established by the third author \cite{Mun3} using the delta method -
\begin{align*}
    L(\tfrac{1}{2}+it,\pi)\ll_{\pi,\varepsilon} (1+|t|)^{\frac{11}{16}+\varepsilon}.
\end{align*}
This was improved by the first author \cite{Agg-IJNT}, who established the exponent $27/40$ by applying delta method with more careful analysis of the integral transforms. His exponent is the limit of the delta method recipe in this context. There has been progress in proving $t$-aspect subconvexity bound for Rankin--Selberg convolutions of $\rm GL(2)$ and $\rm GL(3)$ $L$-functions. The third author \cite{Mun-JEMS2022} used a delta method to prove a subconvexity bound for $\rm GL(3)\times GL(2)$ $L$-functions. Recently, Blomer--Jana--Nelson \cite{Blomer-Jana-Nelson} proved a $t$-aspect Weyl-type bound for $\rm GL(2)\times GL(2)$ $L$-functions by modifying the ideas of Bernstein--Reznikov \cite{Bernstein-Reznikov} and using Kuznetsov trace formula. \\

In a recent pre-print, Nelson \cite{Nelson21} has announced a solution of the $t$-aspect subconvexity problem for standard $L$-functions, regardless of its degree. This breakthrough work is based on period approach introduced in earlier landmark papers of Bernstein-Reznikov \cite{Bernstein-Reznikov} and Michel-Venkatesh \cite{Michel-Venkatesh}. For low degree $L$-functions, Nelson's bound is worse than what is already known. It is however understood that the method has the potential to produce strong bounds even in the case of low degrees.\\

\subsection{Statement of Results and Methodology}

We now state our results and discuss the main ideas. Let $\pi$ be a Hecke cusp form of type $(\nu_1,\nu_2)$ for $\rm SL_3(\BZ)$. Let the normalized Fourier coefficients be given by $\lambda(m_1,m_2)$ (so that $\lambda(1,1)=1$). The $L$-series associated with $\pi$ is given by
$$L(s,\pi) = \sum_{n\geq1}\lambda(1,n)n^{-s}, \quad \text{ for } Re(s)>1.$$

Our main result is given as the following theorem.

\begin{Theorem}\label{SecondMomentThm}
    Let $\pi$ be a Hecke-Maass cusp form for $SL_3(\BZ)$. 
    Then for any $\varepsilon>0$ and $t^{1/2} < M < t^{1-\varepsilon}$, we have \begin{align*}
        \int_{t-M}^{t+M} \left|L\left(\frac{1}{2}+iv,\pi\right)\right|^2 dv \ll_{\pi,\varepsilon}t^\varepsilon\left(\frac{t^{9/4}}{M^{3/2}}+\frac{M^3}{t^{21/20}}+M^{7/4}t^{3/40}+M^{15/14}t^{15/28}\right).
    \end{align*}
    In particular, when $M=t^{2/3}$, we have \begin{align*}
        \int_{t-t^{2/3}}^{t+t^{2/3}} \left|L\left(\frac{1}{2}+iv,\pi\right)\right|^2dv \ll_{\pi,\varepsilon} t^{5/4+\varepsilon}.
    \end{align*}
\end{Theorem}

The second moment average can be used to bound the $L$-function by modifying Good's arguments \cite{Good} (Lemma \ref{ShortMomentLemma}). We obtain that for any $\log t<M<t^{1-\varepsilon}$, 
\begin{align}\label{second-moment}
    \left|L\left(\frac{1}{2}+it,\pi\right)\right|^2\ll&\log t\left(1+\int_\BR U\left(\frac{v}{M}\right)\left|L\left(\frac{1}{2}+it+iv,\pi\right)\right|^2dv\right),
\end{align}

\noindent where $U(x)\in C_c^\infty([1,2])$ is an appropriate bump function. This yields the following subconvexity bound.

\begin{Theorem}\label{main theorem gl3} 
Let $\pi$ be a Hecke-Maass cusp form for $SL_3(\BZ)$. Then for any $\varepsilon>0$,
\begin{equation*}
L\left(1/2+it,\pi\right)\ll_{\pi, \varepsilon} t^{3/4-1/8+\varepsilon}.
\end{equation*}
\end{Theorem}

Studying moments of $L$-functions is an important theme in modern analytic number theory. The bound in Theorem \ref{SecondMomentThm} compares with the best known unconditional estimate for the sixth moment of the Riemann zeta function, which follows from the work of Ivic \cite{Ivic2015}

$$ \int_T^{2T} |\zeta(1/2+it)|^6\, dt \ll T^{5/4 + \epsilon}. $$

Various results have been obtained by estimating the first and second moment of $L$-functions in the last twenty years. This is the first time a short second moment has been used to obtain a subconvex bound in higher rank. 

Li \cite{Li1} and subsequent works improving her subconvexity exponent consider the short first moment of a family of $\rm GL_3\times GL_2$ $L$-functions
\begin{equation*}
\sum_{|t_j-T|<M} L(1/2, \pi\times f_j) + \int_{T-M}^{T+M} |L(1/2+it, \pi)|^2 dt.
\end{equation*}
Here the $\rm GL_3$ form $\pi$ is self-dual and fixed, and the average is taken over a family of $\rm GL_2$ forms $\{f_j\}$ with eigenvalues $t_j$. In comparison, we consider a short second moment of the $L$-function of a fixed $\rm GL_3$ Hecke-Maass cusp form. The clear advantage is our ability to take $\pi$ not necessarily self-dual.\\

We briefly discuss the main ideas to highlight the novelty of our proof. We start with using an approximate functional equation in \eqref{second-moment} and execute the $v$-integral to obtain a shifted sum twisted with an additive character
\begin{align*}
    S_{M,H}(N)\sim M\sum_{h\sim H}\sum_{n\sim N}\lambda(1,n)\overline{\lambda(1,n+h)}e\left(\frac{th}{2\pi n}\right),
\end{align*}
with $H$ going up to $N^{1+\varepsilon}/M$. Next, we apply Duke-Friedlander-Iwaniec's delta method (Lemma \ref{DFILemma}) to separate the oscillations of the Fourier coefficients \begin{align*}
    S_{M,H}(N)\sim M\sum_{h\sim H}\sum_{n\sim N}\lambda(1,n)e\left(\frac{th}{2\pi n}\right)\sum_{m\sim N}\overline{\lambda(1,m)}\delta(m=n+h).
\end{align*}
The application of the delta method contributes $Q^2$-many additive frequencies of size $1\leq q\leq Q$, where $Q$ is a parameter to be chosen later. This is followed by applying dual summation formulas to the $n, m$ and $h$ sums. However, at this point, we diverge from the `routine' and apply a combination of Cauchy-Schwarz inequality and the Duality principle. This allows us to change the lengths of summations inside the absolute-value squared obtained after Cauchy-Schwarz inequality and may be considered the first novelty of our proof.

The next step is to apply dual summation formulas to the new $m$ and $n$ sums followed by stationary phase analysis of the various integral transforms. Although not novel, one needs extreme care in this analysis. This brings us close to the convexity barrier.

The previous application of Cauchy-Schwarz and Duality principle had created two copies of the $h$ and $q$-sums. The final steps involve applying Poisson summation formula to the new $h_1, h_2$-sum and the $q_1, q_2$-sum. However, we are faced with a unique issue due to the Duality principle. The $q_1, q_2$-sum has a structure of the form
\begin{align*}
    \sup_{\|\beta\|_2=1}\mathop{\sum\sum}_{q_1,q_2\sim Q}\beta(q_1)\overline{\beta(q_2)}e\left(\frac{aq_2}{q_1}+\frac{bq_1}{q_2}\right)G(q_1,q_2),
\end{align*}
where $a,b$ are some integers, $G$ is some function and $\beta$ is a sequence of complex numbers appearing from the Duality principle. While we would like to extract savings from the $q_1,q_2$-sums, the presence of $\beta$ is an obstacle. To overcome this hurdle, we apply Cauchy-Schwarz inequality to take out the $q_1$-sum. This allows us to get rid of $\beta(q_1)$ and hence perform Poisson summation in the $q_1$-sum to obtain cancellations. If we stop the analysis here, the character sum present at this stage is a Kloosterman sum. Application of the Weil-Deligne bound yields
\begin{align*}
        \int_{t-M}^{t+M}\left|L\left(\frac{1}{2}+it,\pi\right)\right|^2dt \ll_{\pi,\varepsilon}t^\varepsilon\left(\frac{t^{9/4}}{M^{3/2}}+\frac{M^3}{t^{21/20}}+M^{7/4}t^{3/40}+M^{9/8}t^{9/16}\right).
    \end{align*}
The optimal choice of the length of the second moment is $M=t^{27/42}$, giving the subconvexity bound
\begin{align*}
    L\left(\frac{1}{2}+it,\pi\right)\ll_{\pi,\varepsilon} t^{9/14+\varepsilon}.
\end{align*}

This already improves upon the best known exponent of $27/40$ for a not-necessarily self-dual form obtained in \cite{Agg-IJNT}, but we can do better. Since the above process did not make use of the $q_2$-sum, the previous Cauchy-Schwarz inequality created two copies of it, say $q_2$ and $q_3$ sums. The same analysis is done to the $q_2$-sum without using the $q_3$-sum. We can thus iterate the above process ad infinitum to get further savings. This may be considered the second novelty of our proof.

We highlight that doing so results in increasingly complicated character sums. We bound these exponential sums by using the methods due to Adolphson and Sperber \cite{Adolphson-Sperber} obtained by their extension of Dwork's cohomology theory from smooth projective hypersurfaces in characteristic $p$ to a general class of exponential sums. The associated Newton polyhedron after the first such application is depicted below.

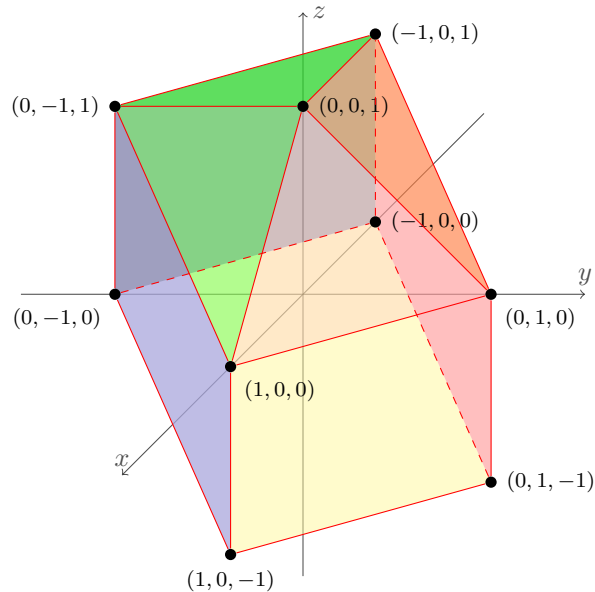
\begin{figure}[h]\label{fig1}
    \centering
    \begin{tikzpicture}[scale=2.5]
%[x=2em,y=2em,z=1.2em,>=stealth]
%https://tex.stackexchange.com/questions/117140/easiest-way-to-draw-a-3d-coordinate-system-with-axis-labels-and-ticks-in-tikz
% The axes
\draw[->,opacity=.7] (xyz cs:x=-1.5) -- (xyz cs:x=1.5) node[above] {$y$};
\draw[->,opacity=.7] (xyz cs:y=-1.5) -- (xyz cs:y=1.5) node[right] {$z$};
\draw[->,opacity=.7] (xyz cs:z=-2.5) -- (xyz cs:z=2.5) node[above] {$x$};

% Color regions
\fill[yellow!50,opacity=.5] (0,0,-1)--(1,-1,0)--(0,-1,1)--(-1,0,0)--cycle; %bottom
\fill[blue!50,opacity=.5] (-1,1,0)--(0,0,1)--(0,-1,1)--(-1,0,0)--cycle; %left
\fill[red!50,opacity=.5] (0,0,-1)--(1,-1,0)--(1,0,0)--(0,1,-1)--cycle; %right
\fill[gray,opacity=.5] (-1,1,0)--(0,1,-1)--(0,0,-1)--(-1,0,0)--cycle; %back
\fill[green,opacity=.5] (-1,1,0)--(0,1,0)--(0,1,-1)--cycle; %top top
\fill[green,opacity=.3] (-1,1,0)--(0,1,0)--(0,0,1)--cycle; %top left
\fill[pink,opacity=.3] (0,0,1)--(0,1,0)--(1,0,0)--cycle; %top front
\fill[orange,opacity=.3] (1,0,0)--(0,1,0)--(0,1,-1)--cycle; %top right

% Points
%(y,z,x)
\node[fill,circle,inner sep=1.5pt,label={right:\footnotesize$(0,0,1)$}] (a) at (0,1,0) {};
\node[fill,circle,inner sep=1.5pt,label={below right:\footnotesize$(1,0,0)$}] (g) at (0,0,1) {};
\node[fill,circle,inner sep=1.5pt,label={right:\footnotesize$(-1,0,0)$}] (c) at (0,0,-1) {};
\node[fill,circle,inner sep=1.5pt,label={below right:\footnotesize$(0,1,0)$}] (d) at (1,0,0) {};
\node[fill,circle,inner sep=1.5pt,label={below left:\footnotesize$(0,-1,0)$}] (h) at (-1,0,0) {};
\node[fill,circle,inner sep=1.5pt,label={below:\footnotesize$(1,0,-1)$}] (f) at (0,-1,1) {};
\node[fill,circle,inner sep=1.5pt,label={right:\footnotesize$(0,1,-1)$}] (e) at (1,-1,0) {};
\node[fill,circle,inner sep=1.5pt,label={right:\footnotesize$(-1,0,1)$}] (b) at (0,1,-1) {};
\node[fill,circle,inner sep=1.5pt,label={left:\footnotesize$(0,-1,1)$}] (i) at (-1,1,0) {};

% Lines
\draw [red] (b)-- (a) -- (i) -- (b) -- (d) -- (e) -- (f) -- (g) -- (d) -- (a) -- (g) --(i) --(h) -- (f);
\draw [red,dashed] (h)-- (c) -- (b);
\draw [red,dashed] (e)-- (c);
\end{tikzpicture}
     \caption{Newton polyhedron for the character sum in the base case}
    \label{fig:my_label}
\end{figure}

We also highlight that we need a careful application of the sophisticated tools developed by Adolphson-Sperber because the size of the Newton polyhedron grows with each application of the Cauchy-Schwarz inequality.

\subsection*{Acknowledgement}
RM is supported by J.C. Bose Fellowship JCB/2021/000018 from SERB, Government of India. We thank Roman Holowinsky for his encouragement and support. We also thank the anonymous referee for careful reading and for their suggestions, which greatly improved the exposition.

\subsection*{Notations} In the rest of the paper, we use the notation $e(x)=e^{2\pi ix}$. For any real numbers $a$ and $b$, we denote $a\sim b$ to mean $k_1<|a/b|<k_2$ for some absolute constants $k_1,k_2>0$. For $j\geq0$, $a\sim_j b$ means $k_1<a/b<k_2$ but $k_1, k_2$ may depend on $j$. We use $\varepsilon$ to denote an arbitrarily small positive constant that can change depending on the context. We must add that for brevity of notation, we assume $t>0$. Indeed, the same analysis holds for $t<0$ by replacing $t$ with $-t$ appropriately.

We will denote by $U(x)$ and $\varphi(x)$ two compactly supported smooth weight functions whose definitions can be revised based on the context. For $\alpha>1$, $C\in\BR$ and $k\in\BZ$, we will refer to the intervals $[\alpha^k C, \alpha^{k+1}C]$ as dyadic intervals even though $\alpha$ may not be equal to $2$.

We say a function $f$ is $Z$-inert if it satisfies the bound $x^{-j}\frac{d^j}{dx^j}f(x)\ll_j Z^j$ for any $j\geq0$. We say $f$ is flat if it is $1$-inert.

\section{Sketch of Proof}

We expand upon the above discussion via a sketch of proof for Theorem \ref{SecondMomentThm}. Let $\log t<M<t^{1-\varepsilon}$ be a parameter and $U\in C_c^\infty([-2,2])$ such that $U(x)=1$ for $-1\leq x\leq 1$. 

\subsection{Short second moment average and approximate functional equation} By the approximate functional equation (\ref{AFE}), it suffices to bound 
\begin{align*}
    S(N):=\int_\BR U\left(\frac{v}{M}\right) \left|\sum_{n=1}^{\infty}\lambda(1,n)\,n^{-i(t+v)}V\left(\frac{n}{N}\right)\right|^2dv,
\end{align*}
where $N\ll t^{3/2+\varepsilon}$ and $V\in C_c^\infty([1,2])$ may depend on $t$ and satisfies $V^{(j)}(x)\ll_j1$ for $j\geq0$. Opening the square and integrating yields a shifted sum (Lemma \ref{SMHLemma}),
\begin{align*}
S(N) \ll Mt^\varepsilon+\sup_{H\ll\frac{Nt^\varepsilon}{M}}\frac{S_{M,H}(N)}{N},
\end{align*}
where we get the shifted sum
\begin{align*}
    S_{M,H}(N)\sim M\sum_{h\sim H}\sum_{n\sim N}\lambda(1,n)\overline{\lambda(1,n+h)}e\left(\frac{th}{2\pi n}\right).
\end{align*}
For sketch, we focus on the generic cases $H=\frac{N}{M}$, i.e. we focus on $S_{M,N/M}(N)$.

We note that Lemma \ref{Ram bound} gives us the trivial bound \begin{align}\label{sketch.TrivialBound}
    S_{M,N/M}(N)\ll N^2t^\varepsilon.
\end{align}

\subsection{Delta method} Next we apply the DFI delta method in Lemma \ref{DFILemma} to detect $m=n+h$, \begin{align*}
    S_{M,N/M}(N)\sim& \frac{M}{Q}\sum_{1\leq q\leq Q}\frac{1}{q}\sum_{h\sim \frac{N}{M}}\sum_{n\sim N}\lambda(1,n)e\left(\frac{th}{2\pi n}\right)\sum_{m\sim N}\overline{\lambda(1,m)}\nonumber\\
    &\times\sumx_{\alpha\bmod q}e\left(\frac{\alpha(n+h-m)}{q}\right)\int_\BR g(q,x)e\left(\frac{(n+h-m)x}{qQ} \right)dx.
\end{align*}
For the purpose of this sketch, we focus on the generic case $x\sim 1$.

\subsection{Dual summation formulas} After separating the variables with the delta method, we apply dual summation formulas. Applying Poisson summation to the $h$-sum, and Voronoi summation (Lemma \ref{gl3voronoi}) to the $m$ and $n$ sums, we roughly get 
\begin{align*}
    S_{M,N/M}(N)\sim&\frac{MN^{4/3}}{Q^2t^{1-\varepsilon}}\sum_{q\sim Q} \sum_{h\sim\frac{qt}{N}}\sum_{m\ll\frac{N^2}{Q^3}}\frac{\overline{\lambda(m)}}{m^{1/3}}\sum_{n\ll \frac{N^2}{Q^3}+\frac{Q^3t^3}{M^3N}}\frac{\lambda(n)}{n^{1/3}}\\
    &\times S\left(\overline{h},m;q\right)S\left(\overline{h},n;q\right)\int_{|u|\ll1}e\left(f_1(m,h,q,u)+f_2(n,h,q,u)\right)du,
\end{align*}
with the analytic oscillations being of size $f_1(m,h,q,u)\sim \frac{N}{qQ}$ and $f_2(n,h,q,u)\sim\frac{(nN)^{1/3}}{q}$. The variable $u$ is a secondary oscillatory terms that is created in the Voronoi summation to the $n$-sum. It does not have any significance in the generic case when $N\sim t^{3/2}$, but we get a little bit extra saving when $N$ is smaller and the main oscillatory terms cancel out in the diagonal case. In the rigorous proof, we do not utilize the cancellations in the integrals until later stages due to the fact that $x$ can be smaller than 1. We choose $Q\ll\sqrt{\frac{MN}{t}}$. Indeed to match the oscillation, we want to choose $Q=\sqrt{\frac{MN}{t}}$, but we will choose $Q$ smaller when $N$ is small. Focusing on the generic case $m\sim\frac{N^2}{Q^3}$ and $n\sim\frac{N^2}{Q^3}+\frac{Q^3t^3}{M^3N}\sim \frac{N^2}{Q^3}$, we roughly have \begin{align*}
    S_{M,N/M}(N)\sim&\frac{M}{t^{1-\varepsilon}}\sum_{q\sim Q} \sum_{h\sim\frac{qt}{N}}\sum_{m\sim\frac{N^2}{Q^3}}\overline{\lambda(m)}\sum_{n\sim \frac{N^2}{Q^3}}\lambda(n)\\
    &\times S\left(\overline{h},m;q\right)S\left(\overline{h},n;q\right)\int_{|u|\ll1}e\left(f_1(m,h,q,u)+f_2(n,h,q,u)\right)du,
\end{align*}
with the analytic oscillations being of size $f_1(m,h,q), f_2(n,h,q)\sim \frac{N}{qQ}$.

\subsection{Cauchy Schwarz inequality} Next we apply Cauchy-Schwarz inequality to take out the $q, h$ sums and the $u$-integral, giving us \begin{align*}
    S_{M,N/M}(N)\ll \sqrt{S_1S_2},
\end{align*}
where roughly \begin{align*}
    S_j\sim \frac{M}{t^{1-\varepsilon}}\sum_{q\sim Q} \sum_{h\sim\frac{qt}{N}}\int_{|u|\ll1}\left|\sum_{m\sim\frac{N^2}{Q^3}}\lambda(m) S\left(\overline{h},m;q\right)e\left(f_j(m,h,q,u)\right)\right|^2du.
\end{align*}

\begin{Remark}\label{CS+Duality}
One should note that the length of the $q$ and $h$ sums are shorter than the $n$ and $m$ sums. Therefore it would have been beneficial to instead take out the $n$ and $m$ sums out of the absolute value squared in the above step. However, our approach benefits by allowing us to separate the $n$ and $m$ sums.
\end{Remark}

\subsection{Duality principle} We apply the Duality principle (Lemma \ref{dualitylemma}) together with Lemma \ref{Ram bound} to take the $n$ and $m$ sums outside the absolute value squared and obtain 
\begin{align*}
    S_j\ll \frac{MN^2}{Q^3t^{1-\varepsilon}}\sup_{\|\beta\|_2=1}\sum_{m\sim\frac{N^2}{Q^3}}\left|\sum_{q\sim Q} \sum_{h\sim\frac{qt}{N}} S\left(\overline{h},m;q\right)\int_{|u|\ll1} \beta(q,h,u)e\left(f_j(m,h,q,u)\right)du\right|^2.
\end{align*}
This also rids us of the $\rm GL_3$ Fourier coefficients, decreasing the complexity of the problem at hand. 

\subsection{Poisson summation formula} Opening the square and applying Poisson summation on the $m$-sum gives us roughly \begin{align*}
    S_j\ll &\frac{MN^4}{Q^8t^{1-\varepsilon}}\sup_{\|\beta\|_2=1}\mathop{\sum\sum}_{q_1,q_2\sim Q} \mathop{\sum\sum}_{h_i\sim\frac{q_i t}{N}}\sum_{|m|\ll\frac{Q^2\sqrt{q_1q_2}}{N}}\sum_{\gamma\bmod q_1q_2}S\left(\overline{h_1},\gamma;q_1\right)S\left(\overline{h_2},\gamma;q_2\right)e\left(\frac{m\gamma}{q_1q_2}\right)\\
    &\times \int_{|u_1|\ll1}\int_{|u_2|\ll1}\beta(q_1,h_1,u_1)\overline{\beta(q_2,h_2,u_2)}\\
    &\times \int_\BR V(y) e\left(f\left(\frac{N^2}{Q^3}y,h_1,q_1,u_1\right)-f\left(\frac{N^2}{Q^3}y,h_2,q_2,u_2\right)-\frac{mN^2y}{q_1q_2Q^3}\right)dydu_1du_2.
\end{align*}

Now we have a spectrum of cases depending on the size of $m$. For the sketch, we shall focus on the two extreme cases, the diagonal corresponding to $m=0$, and the off-diagonal where $m\sim\frac{Q^3}{N}$ is as large as possible. The other cases will yield bounds in between these two cases.

\subsection{Diagonal terms} For the diagonal $m=0$, we essentially have $q_1=q_2(=q), h_1=h_2$ and $|u_1-u_2|\ll\frac{q^2t}{MN}$. In such a case, one can evaluate the $\gamma$-sum and obtain the bound \begin{align*}
    \frac{MN^4}{Q^8t^{1-\varepsilon}}Q^3\frac{Q^2t}{MN}\ll\frac{N^3t^\varepsilon}{Q^3}.
\end{align*}
Here is the only place where we extract an extra saving coming from the secondary oscillatory terms having $u_1,u_2$. This is possible as the main oscillatory terms cancel out when $q_1=q_2,h_1=h_2$. Moreover, we'll see that $\frac{Q^2t}{MN}$ is less than $1$ only when $N$ is smaller than the generic size $t^{3/2}$ by our choice of $Q$.

\subsection{Off-diagonal terms} For the off-diagonal where $m\sim\frac{Q^3}{N}$, evaluating the $\gamma$-sum and apply stationary phase analysis on the $y$-integral gives us roughly \begin{align*}
    &\frac{MN^{7/2}}{Q^5t^{1-\varepsilon}}\sup_{\|\beta\|_2=1}\sum_{m\sim\frac{Q^3}{N}}\mathop{\sum\sum}_{q_1,q_2\sim Q} \mathop{\sum\sum}_{h_1,h_2\sim\frac{qt}{N}}\beta(q_1,h_1)\overline{\beta(q_2,h_2)}\\
    &\times e\left(-\frac{\overline{mh_1}q_2}{q_1}-\frac{\overline{mh_2}q_1}{q_2}\right)e\left(F(m,h_1,h_2,q_1,q_2)\right),
\end{align*}
with $F(m,h_1,h_2,q_1,q_2)\sim \frac{N}{Q^2}$. Here we dropped the $u_1,u_2$-integral in the sketch as we will just bound it trivially with $\|\beta\|_2=1$. The bound at this point is \begin{align*}
    \frac{MN^{7/2}}{Q^5t^{1-\varepsilon}}\frac{Q^3}{N}Q\frac{Qt}{N}\ll MN^{3/2}t^\varepsilon.
\end{align*}
We want to get saving from the $q_j$ sums. Assuming we are in the generic case $(q_1,q_2)=1$ in sketch, reciprocity gives us  \begin{align*}
    &\frac{MN^{7/2}}{Q^5t^{1-\varepsilon}}\sup_{\|\beta\|_2=1}\sum_{m\sim\frac{Q^3}{N}}\mathop{\sum\sum}_{q_1,q_2\sim Q} \mathop{\sum\sum}_{h_1,h_2\sim\frac{q_it}{N}}\beta(q_1,h_1)\overline{\beta(q_2,h_2)}\\
    &\times e\left(\frac{\overline{q_1}q_2}{mh_1}+\frac{\overline{q_2}q_1}{mh_2}\right)e\left(-\frac{q_2}{mh_1q_1}-\frac{q_1}{mh_2q_2}+F(m,h_1,h_2,q_1,q_2)\right).
\end{align*}

\subsection{Infinite Cauchy-Schwarz and Poisson summation} We want to apply Poisson summation to the $q_1,q_2$ sums, but the $\beta$ coefficients obtained from the Duality principle prevent us from doing so. To deal with the $\beta$ coefficients, we have to apply Cauchy-Schwarz inequality to take out the $q_1$-sum while leaving the $q_2$-sum inside the square. Applying Poisson summation afterwards on the $q_1$-sum would give us a saving of \begin{align*}
    \left(\sqrt{\frac{\text{old length}}{\text{new length}}}\right)^{1/2}\sim \left(\frac{N^2}{Qt^2}\right)^{1/4}
\end{align*}
in $S_1(N)$. The process of Cauchy Schwarz creates two copies of $q_2$-sum, say $q_2,q_3$ sums. We then repeat the process of taking out the $q_2$-sum, leaving $q_3$-sum inside the square and applying Poisson summation on the $q_2$-sum. This gives us a saving of \begin{align*}
    \left(\frac{N^2}{Qt^2}\right)^{1/8}
\end{align*}

\begin{Remark}
The above saving is a result of careful analysis of complicated exponential sums, and we need to apply the Weil-Deligne bound via the result of Adolphson-Sperber obtained by their extension of Dwork's cohomology theory from smooth projective hypersurfaces in characteristic $p$ to a general class of exponential sums. 
\end{Remark}

Iterating this process $\nu$ times for $\nu$ sufficiently large yields a total saving of \begin{align*}
    \left(\frac{N^2}{Qt^2}\right)^{1/4}\times\left(\frac{N^2}{Qt^2}\right)^{1/8}\times\cdots \times \left(\frac{N^2}{Qt^2}\right)^{1/2^\nu}\sim \left(\frac{N^2}{Qt^2}\right)^{1/2}t^{-\varepsilon}
\end{align*}
and hence we get the contribution of $m\sim\frac{Q^3}{N}$ is bounded by \begin{align*}
    MN^{3/2}t^{\varepsilon}\left(\frac{Qt^2}{N^2}\right)^{1/2}\ll M\sqrt{QN}t^{1+\varepsilon}.
\end{align*}

\subsection{Final calculations} As a result, assuming all other cases lie between the two cases in sketch, we have for $Q\ll\sqrt{\frac{MN}{t}}$, \begin{align*}
    S_j\ll \frac{N^3t^\varepsilon}{Q^3}+M\sqrt{QN}t^{1+\varepsilon},
\end{align*}
and hence \begin{align*}
    \frac{S_{M,N/M}(N)}{N}\ll \frac{N^2t^\varepsilon}{Q^3}+\frac{\sqrt{Q}Mt^{1+\varepsilon}}{\sqrt{N}}
\end{align*}
We choose $Q=\frac{\sqrt{M}N^{2/3}}{t^{3/4}}\ll\sqrt{\frac{MN}{t}}$, giving us \begin{align*}
    \frac{S_{M,N/M}(N)}{N}\ll \frac{t^{9/4+\varepsilon}}{M^{3/2}}+\frac{M^{5/4}t^{5/8+\varepsilon}}{N^{1/6}}.
\end{align*}
Together with the trivial bound in (\ref{sketch.TrivialBound}), we have \begin{align*}
    \frac{S_{M,N/M}(N)}{N}\ll \frac{t^{9/4+\varepsilon}}{M^{3/2}}+\min\left\{Nt^\varepsilon,\frac{M^{5/4}t^{5/8+\varepsilon}}{N^{1/6}}\right\}\ll \frac{t^{9/4+\varepsilon}}{M^{3/2}}+M^{15/14}t^{15/28+\varepsilon}.
\end{align*}
The optimal choice is $M=t^{2/3}$, which gives us the bound \begin{align*}
    \frac{S_{M,N/M}(N)}{N}\ll t^{5/4+\varepsilon},
\end{align*}
and \begin{align*}
    \left|L\left(\frac{1}{2}+it,\pi\right)\right|^2\ll&\log t\left(1+\int_\BR U\left(\frac{v}{t^{2/3}}\right)\left|L\left(\frac{1}{2}+it+iv,\pi\right)\right|^2dv\right)\ll_{\pi,\varepsilon} t^{5/4+\varepsilon}.
\end{align*}

\section{Preliminaries on automorphic forms}

Let $\pi$ be a Maass form for $\rm SL_3(\BZ)$, which is an eigenfunction for all the Hecke operators. Let the Fourier coefficients be $\lambda(n_1,n_2)$, normalized so that $\lambda(1,1)=1$. The Langlands parameter $(\alpha_1,\alpha_2, \alpha_3)$ associated with $\pi$ satisfies $\alpha_1+\alpha_2+\alpha_3=0$. The dual cusp form $\tilde{\pi}$ has Langlands parameters $(-\alpha_3, -\alpha_2, -\alpha_1)$. The $L$-function $L(s,\pi)$ satisfies a functional equation 
$$\gamma(s,\pi)L(s,\pi) = \gamma(1-s,\pi)L(1-s,\tilde{\pi}), $$
where $\gamma(s,\pi)$ and $\gamma(s,\tilde{\pi})$ are the associated gamma factors. We refer the reader to Goldfeld's book on automorphic forms for $\rm GL_n(\BZ)$ \cite{goldfeldbook} for the theory of automorphic forms on higher rank groups.

\subsection{Approximate functional equation and Voronoi summation formula}
We are interested in bounding $L(s,\pi)$ on the critical line, $Re(s)=1/2$. For that, we approximate $L(1/2+it, \pi)$ by a smoothed sum of length $t^{3/2+\varepsilon}$. This is known as the approximate functional equation and is proved by applying Mellin transformation to $f$ followed by using the above functional equation.

\begin{Lemma}[{\cite[Theorem 5.3]{Iw-Ko}}]\label{AFE}
Let $G(u)$ be an even, holomorphic function bounded in the strip $-4\leq Re(u)\leq 4$ and normalized by $G(0)=1$. Then for $s$ in the strip $0\leq \sigma\leq 1$ 
$$ L(s,\pi) = \sum_{n\geq1} \lambda(1,n)n^{-s}V_s(n) +  \frac{\gamma(1-s,\tilde{\pi})}{\gamma(s,\pi)} \sum_{n\geq1} \overline{\lambda(1,n)}n^{-(1-s)}V_{1-s}(n),$$
where
$$V_s(y) = \frac{1}{2\pi i}\int_{(3)}y^{-u}G(u)\frac{\gamma(s+u,\tilde{\pi})}{\gamma(s,\pi)}\frac{du}{u}, $$
and $\gamma(s,\pi)$ is a product of certain $\Gamma$-functions appearing in the functional equation of $L(s,\pi)$.
\end{Lemma}
\begin{Remark}
On the critical line, Stirling's approximation to $\gamma(s,\pi)$ followed by integration by parts to the integral representation of $V_{1/2\pm it}(n)$ gives arbitrary savings for $n\gg (1+|t|)^{3/2+\varepsilon}$.
\end{Remark}

One of the main tools in our proof is a Voronoi type formula for $\rm GL_3(\BZ)$. Let $h$ be a compactly supported smooth function on  $(0, \infty)$, and let $\tilde{h}(s)=\int_0^\infty h(x)x^{s-1}\ dx$ be its Mellin transform. For $\sigma>-1+\max\{-Re(\alpha_1), -Re(\alpha_2), -Re(\alpha_3)\}$ and $a=0, 1$, define
\begin{equation*}
\gamma_a(s) = \frac{\pi^{-3s-3/2}}{2}\prod_{i=1}^3\frac{\Gamma(\frac{1+s+\alpha_i+a}{2})}{\Gamma(\frac{-s-\alpha_i+a}{2})}.
\end{equation*}
Further set $\gamma_\pm(s) = \gamma_0(s)\mp i\gamma_1(s)$ and let
\begin{equation}\label{H_pm}
H_\pm(y) = \frac{1}{2\pi i}\int_{(\sigma)} y^{-s}\gamma_\pm(s)\tilde{h}(-s)\ ds. 
\end{equation}
We need the following Voronoi type formula (See \cite{Blo12, Li1, Miller-Schmid}).

\begin{Lemma}\label{gl3voronoi}
Let $h$ be a compactly supported smooth function on $(0, \infty)$. We have
\begin{equation*}
\sum_{n=1}^\infty \lambda(1,n) e(an/q)h(n) = q\sum_\pm \sum_{n_0|q}\sum_{n=1}^\infty \frac{\lambda(n,n_0)}{nn_0} S(\overline{a}, \pm n; q/n_0) H_\pm (n_0^2n/q^3),
\end{equation*}
where $(a,q)=1$ and $a\overline{a}\equiv 1\bmod q$.
\end{Lemma}

Stirling approximation of $\gamma_{\pm}(s)$ gives $\gamma_\pm(\sigma+i\tau)\ll 1+|\tau|^{3\sigma+3/2}$. Moreover on $Re(s)=-1/2$
\begin{equation*}
\gamma_\pm(-1/2+i\tau) = (|\tau|/e\pi)^{3i\tau}\Phi_\pm(\tau), \qquad \text{ where } \Phi_\pm'(\tau)\ll |\tau|^{-1}.
\end{equation*}

\begin{Lemma}[{\cite[Lemma 6]{Blo12}}]\label{voronoi}
Let $h$ be a compactly supported smooth function on $[a, b]\subset(0, \infty)$ and $H_\pm$ be defined as in \eqref{H_pm}. Then there exist constants $\gamma_\ell$ depending only on the Langlands parameters $(\alpha_1, \alpha_2, \alpha_3)$ such that 
\begin{equation*}
H_\pm(x) = x\int_0^\infty h(y) \sum_{\ell=1}^L \frac{\gamma_\ell}{(xy)^{\ell/3}}e(\pm 3(xy)^{1/3})\, dy + R(x),
\end{equation*}
where
\begin{equation*}
x^k\frac{d^k}{dx^k}R(x) \ll ||h||_\infty x^{1+(k-L)/3}.
\end{equation*}
The implied constant above depends on $a, b, \alpha_1, \alpha_2, \alpha_3, k$ and $L$.
\end{Lemma}

\subsection{Short second moment average} We bound the square of the $L$-function by a second moment average over a short interval by following Good's arguments \cite{Good}.

\begin{Lemma}\label{ShortMomentLemma}
Let $T/2\leq t\leq T$, then we have \begin{align*}
    \left|L\left(\frac{1}{2}+it,\pi\right)\right|^2\ll \log t\left(1+\int_{-\log T}^{\log T}\left|L\left(\frac{1}{2}+it+iv,\pi\right)\right|^2e^{-v^2/2}\, dv\right).
\end{align*}
\end{Lemma}

\begin{proof}
Let $c = 1/\log t$ for $t\geq 10$. By the residue theorem,
\begin{equation*}
L^2(1/2+it,\pi) = \frac{1}{2\pi i}\int_{(c)} L^2(1/2+it+s, \pi)\frac{e^{s^2}}{s}\, ds - \frac{1}{2\pi i}\int_{(-c)} L^2(1/2+it+s, \pi)\frac{e^{s^2}}{s}\, ds.
\end{equation*}
Functional equation for $L(s,\pi)$ and Stirling's approximation give $$L(1/2-c+i(t+v))\ll (1+|t+v|^{3c})|L(1/2+c+i(t+v))|.$$
Therefore
\begin{equation}\label{square-estimate}
|L(1/2+it, \pi)|^2 \ll \int_{-\infty}^\infty |L(1/2+c+i(t+v))|^2(1+|t+v|^{6c})\frac{e^{-v^2}}{(v^2+c^2)^{1/2}}\, dv.
\end{equation}

Similarly, for $1/2<\sigma<1$,
\begin{equation*}
\begin{split}
L^2(\sigma + i\tau_1, \pi) &= \frac{1}{2\pi i} \int_{(1)}L^2(\sigma + i\tau_1+ s, \pi)\frac{e^{s^2}}{s}\, ds - \frac{1}{2\pi i} \int_{(1/2-\sigma)}L^2(\sigma + i\tau_1+ s, \pi)\frac{e^{s^2}}{s}\, ds \\ 
&\ll 1 + \int_{-\infty}^\infty |L(1/2+i(\tau_1+\tau_2), \pi)|^2 \frac{e^{-\tau_2^2}}{(\tau_2^2+(1/2-\sigma)^2)^{1/2}}\, d\tau_2. 
\end{split}
\end{equation*}
Putting $\sigma=1/2+c$, $\tau_1 =  t+v$, substituting the above estimate into \eqref{square-estimate} and trivially estimating the $\tau_2$-integral, we obtain
\begin{align*}
    \left|L\left(\frac{1}{2}+it,\pi\right)\right|^2\ll \log t \left(1+\int_{-\infty}^{\infty}\left|L\left(\frac{1}{2}+it+iv,\pi\right)\right|^2e^{-v^2/2}\, dv\right).
\end{align*}
Since $L(1/2+it+iv, \pi)$ is polynomially bounded, the integral can be cutoff at $[-\log T, \log T]$ for a negligible error term.

\end{proof}

We will also use Ramanujan bound on average which follows from the Rankin--Selberg theory.
\begin{Lemma}[Ramanujan bound on average] \label{Ram bound}
We have
\begin{equation*}
    \underset{n_1^2n_2\leq x}{\sum\sum}|\lambda(n_2,n_1)|^2 \ll_{\pi,\varepsilon} x^{1+\varepsilon}.
\end{equation*}
\end{Lemma}

\section{Preliminary lemmas}

\begin{Lemma}[DFI delta method]\label{DFILemma}
Let $Q>1$. Then we have
\begin{equation*}
\delta(n=0) = \frac{1}{Q}\sum_{1\leq q\leq Q} \frac{1}{q}\sumx_{a\bmod q}e\bigg(\frac{an}{q}\bigg) \int_{\BR} g(q,x) e\bigg(\frac{nx}{qQ}\bigg)\, dx,
\end{equation*}

where $g(q,x)$ is a smooth function satisfying

\begin{align}\label{gqx}
g(q,x) = 1 + O_A\bigg(\frac{Q}{q}\bigg(\frac{q}{Q}+|x|\bigg)^{A}\bigg), \quad g(q,x)\ll |x|^{-A}, \quad \text{ for any } A\geq1,
\end{align}

and
$$ \frac{\partial^j}{\partial x^j}g(q,x)\ll |x|^{-j}\min(|x|^{-1}, Q/q)\log Q, \quad \text{ for } j\geq 1.$$

Moreover, $g(q,x)$ satisfies
\begin{align}\label{gqxBound}
    g(q,x)\ll Q^\varepsilon.
\end{align}
\end{Lemma}

\begin{proof}
See \cite[Section 20.5]{Iw-Ko}, and \cite[Lemma 15]{Huang1} for minor corrections in the original estimates of the derivatives of $g(q,x)$. To prove \eqref{gqxBound}, we observe that when $q>Q^{1-\varepsilon}$, the first property in \eqref{gqx} gives us $g(q,x)\ll Q^\varepsilon$ by taking $A=1$. In the case $q\leq Q^{1-\varepsilon}$ and $x>Q^\varepsilon$, the second property in \eqref{gqx} gives us $g(q,x)\ll Q^\varepsilon$ by taking $A=1$. In the remaining case where $q,x\leq Q^{1-\varepsilon}$, then the first property gives us $g(q,x)\ll1$ by taking $A$ sufficiently large. Combining all the cases, we have the result.

\end{proof}

Note that the properties $g(q,x)$ stated in Lemma \ref{DFILemma} gives us \begin{align}\label{gDerProp}
    \begin{cases}
        g(q,x)=1+O\left(t^{-A}\right) \text{ for any } A>0 & \text{ if } x<t^{-\varepsilon/10} \text{ and } q<Qt^{-\varepsilon/10}\\
        g(q,x) \text{ is } t^{\varepsilon/10} \text{-inert w.r.t. } x & \text{ otherwise.}
    \end{cases}
\end{align}

\iffalse
{We will also need the $L^2$-norm on $g(q,x)$. Let $f(u)$ and $w(u)$ be compactly supported (or at least rapidly decreasing) functions on $\BR$. Let
\begin{equation}
\Delta_q(u) = \sum_{r\geq1} \frac{1}{qr}(w(qr) - w(|u|/qr)).
\end{equation}
If $w(u)$ is supported in $[Q,2Q]$ and satisfies $w^{(a)}(u)\ll Q^{-a-1}$, then $\Delta_q(u)$ satisfies the bounds 
\begin{equation}
\Delta_q(u) \ll \frac{1}{(q+Q)Q} + \frac{1}{|u|+qQ}
\end{equation}
and
\begin{equation}
\Delta_q^{(a)}(u)\ll (qQ)^{-1}(|u|+qQ)^{-a}.
\end{equation}

By Fourier inversion, we have
\begin{equation}
g(q,x) = \int_{\BR} \Delta_q(u)f(u)e(-ux/qQ)du.
\end{equation}
We can take $f(u) = e^{-u/Q^2}$. Since $||g||_2 = ||\Delta_q\cdot f||_2$, we have,
\begin{equation}\label{g2norm}
\begin{split}
||g(q,\cdot)||_2^2 &= ||qQ\Delta_q(qQ\cdot)f(qQ\cdot)||_2^2 = (qQ)^2\int_\BR |\Delta_q(uqQ)e^{-uq/Q}|^2du\\
&\ll (qQ)^2 \int_{|u|\leq 1} \frac{1}{(qQ)^2}du + (qQ)^2 \int_{|u|>1}\frac{1}{(uqQ)^2} du\ll 1.
\end{split}
\end{equation}

Similarly, we have \begin{align}\label{gDer2norm}
    ||g_y(q,\cdot)||_2^2 &= ||qQ\cdot\Delta_q(qQ\cdot)f(qQ\cdot)||_2^2 = (qQ)^2\int_\BR |u\Delta_q(uqQ)e^{-uq/Q}|^2du\nonumber\\
    &\ll (qQ)^2 \int_{|u|\leq Q/q} \frac{1}{(qQ)^2}du+(qQ)^2 \int_{|u|>Q/q}e^{-uq/Q} du\ll \frac{Q}{q}.
\end{align}
}

\fi

We will also use the Duality principle (see \cite[Chapter 7]{Iw-Ko}).

\begin{Lemma}[Duality principle]\label{dualitylemma}
    Let $\phi:\BZ^2\rightarrow\BC$. For any complex numbers $a_m$, \begin{align*}
        \sum_n\left|\sum_m a_m\phi(m,n)\right|^2\ll \left(\sum_m |a_m|^2\right)\sup_{\|\beta\|_2=1}\sum_m\left|\sum_n\beta(n)\phi(m,n)\right|^2,
    \end{align*}
    where the supremum is taken over all sequences of complex numbers $\beta(n)$ such that $$\|\beta\|_2=\sqrt{\sum_n|\beta(n)|^2}=1.$$
\end{Lemma}

\section{Stationary phase analysis}

We need to use stationary phase analysis for oscillatory integrals. Let $\CI$ be an integral of the form
\begin{equation}\label{eintegral}
\CI = \int_{a}^b w(t)e^{i \phi(t)}\, dt,
\end{equation}
where $w$ and $\phi$ are  real valued smooth functions on  $\BR$. The fundamental estimate for integrals of the form \eqref{eintegral} is the $r^{th}$-derivative test 
\begin{equation*}
\CI\ll \bigg(\underset{[a,b]}{Var}\, w(t)\bigg)\bigg/\bigg(\min_{[a,b]}|\phi^{(r)}(t)|^{1/r}\bigg).
\end{equation*}
We will however need sharper estimates and will use the stationary phase analysis of Blomer--Khan--Young \cite{BKY} to analyze $\CI$ and state these in the language of inert functions as given by Kiral--Petrow--Young \cite{Kiral-Petrow-Young}. %Moreover, in the case when the stationary point lies far enough from the interval $[a,b]$, we use Lemma 8.1 of Blomer, Khan and Young \cite{bky} on stationary phase analysis to show that $\frI$ is arbitrarily small. For completeness, we state the results here in the language of inert functions \cite{Kiral-Petrow-Young}.

\begin{Definition}
Let $\CF$ be an indexing set and $X : \CF \rightarrow [1,\infty)$ be a function of  $T\in\CF$, so that $X_T\in[1, \infty)$. A family $\{w_T\}_{T\in\CF}$ of smooth functions supported on a product of dyadic intervals in $\BR^d_{>0}$ is called $X$-inert if for each $j=(j_1,...,j_d)\in\BZ^d_{\geq0}$, we have
\begin{equation*}
\sup_{T\in\CF} \,\, \sup_{(x_1,...,x_d)\,\in\,\BR^d_{>0}} X_T^{-(j_1+...+j_d)}\big|x_1^{j_1}...\, x_d^{j_d}\, w_T^{(j_1,...,j_d)}(x_1,...,x_d) \big| \ll_{j_1,...j_d} 1.
\end{equation*}
\end{Definition}

\begin{Lemma}\label{statlemma}
Suppose that $w = w_T(t)$ is a family of $X$-inert functions with compact support of $[Z, 2Z]$, so that $w^{(j)}(t)\ll (X/Z)^j$. Also suppose that $\phi$ is smooth and satisfies $\phi^j(t)\ll Y/Z^j$ for some $Y/X^2\geq R\geq 1$ for all $t$ in the support of $w$. Let
$$\CI = \int_{-\infty}^\infty w(t)e^{i \phi(t)}\, dt. $$
\begin{enumerate}
\item If $|\phi'(t)|\gg Y/Z$ in the support of $w$, then $\CI\ll_A ZR^{-A}$ for arbitrarily large $A$. Moreover, the statement also holds under the weaker condition $Y/X\geq R$.

\item If $\phi'(t)\gg Y/Z^2$ in the support of $w$, and there exists some not necessarily unique $t_0\in\BR$ such that $\phi'(t_0) = 0$, then
\begin{equation*}
\CI = \frac{e^{i\phi(t_0)}}{\sqrt{\phi''(t_0)}}F_T(t_0) + O_A(ZR^{-A})
\end{equation*}
for any $A\geq0$, where $F_T$ is a family of $X$-inert functions (possibly depending on $A$) supported on $t_0\sim Z$.

\item Same setting as (2). Let $U\gg R^\varepsilon(\phi''(t_0))^{-1/2}$ and let $W_0\in C_c^\infty([-2,2])$ such that $W_0(x)=1$ for $|x|\leq 1$. Then we have \begin{align*}
    \CI = \int_{-\infty}^\infty w(t)W_0\left(\frac{t-t_0}{U}\right)e^{i \phi(t)}\, dt+O_A\left(R^{-A}\right)
\end{align*}
for any $A\geq0$.
\end{enumerate}
\end{Lemma}

\section{Main Calculations}

\subsection{A Short Second Moment}

We first apply Lemma \ref{ShortMomentLemma} to bound the square of $L(1/2+it,\pi)$ by a short second moment. Let $T/2\leq t\leq T$, then we have \begin{align}\label{ApplyingShortMomentLemma}
    \left|L\left(\frac{1}{2}+it,\pi\right)\right|^2\ll \log t\left(1+\int_{-\log T}^{\log T} \left|L\left(\frac{1}{2}+it+iv,\pi\right)\right|^2e^{-v^2/2}\, dv\right).
\end{align}
Let $\varepsilon>0$ and let $\log t<M<t^{1-\varepsilon}$ be a parameter. Let $U\in C_c^\infty(\R)$ be a fixed function supported in $[-2, 2]$ and satisfying $U(x)=1$ for $-1\leq x\leq 1$, then (\ref{ApplyingShortMomentLemma}) gives us 
\begin{align*}
    \left|L\left(\frac{1}{2}+it,\pi\right)\right|^2\ll\log t\left(1+\int_\BR U\left(\frac{v}{M}\right)\left|L\left(\frac{1}{2}+it+iv,\pi\right)\right|^2dv\right).
\end{align*}
Applying Lemma \ref{AFE}, we get \begin{align*}
    \left|L\left(\frac{1}{2}+it,\pi\right)\right|^2\ll&\log t\left(1+\int_\BR U\left(\frac{v}{M}\right)\left|L\left(\frac{1}{2}+it+iv,\pi\right)\right|^2dv\right)\nonumber\\
    \ll& \log t\left(1+\sup_{1\leq N\ll t^{3/2+\varepsilon}}\frac{S(N)}{N}\right),
\end{align*}
where 
\begin{align*}
    S(N):=\int_\BR U\left(\frac{v}{M}\right) \left|\sum_{n=1}^{\infty}\lambda(1,n)\,n^{-i(t+v)}V\left(\frac{n}{N}\right)\right|^2dv,
\end{align*}
with $V\in C_c^\infty([1,2])$ depending on $t$ satisfying $V^{(j)}(x)\ll_j1$ for any $j\geq0$ ({See e.g. \cite[Proposition 5.4]{Iw-Ko}}). Opening the square yields 
\begin{align*}
    S(N)=\sum_n\lambda(1,n)V\left(\frac{n}{N}\right)\sum_m\overline{\lambda(1,m)}\overline{V\left(\frac{m}{N}\right)}\left(\frac{m}{n}\right)^{it}\int_\BR U\left(\frac{v}{M}\right) \left(\frac{m}{n}\right)^{iv} dv.
\end{align*}
By repeated integration by parts on the $v$-integral, we obtain arbitrary savings unless \begin{align*}
    |m-n|\ll\frac{Nt^\varepsilon}{M}.
\end{align*}
Hence we have 
\begin{align*}
    S(N)=&M\sum_{|h|\ll\frac{Nt^\varepsilon}{M}}\sum_n\lambda(1,n)\overline{\lambda(1,n+h)}V\left(\frac{n}{N}\right)\overline{V\left(\frac{n+h}{N}\right)}\left(\frac{n+h}{n}\right)^{it}\nonumber\\
    &\times\int_\BR U\left(v\right) \left(\frac{n+h}{n}\right)^{iMv} dv + O(t^{-999}).
\end{align*}
For $h=0$, (\ref{Ram bound}) on the $n$-sum gives us \begin{align*}
    M\sum_n|\lambda(1,n)|^2\left|V\left(\frac{n}{N}\right)\right|^2\int_\BR U\left(v\right) dv\ll MNt^\varepsilon.
\end{align*}
Now performing a smooth dyadic subdivision of $h$-sum for $h\neq0$ yields \begin{align*}
    S(N)\ll&\sup_{H\ll\frac{Nt^\varepsilon}{M}}\sum_\pm M\sum_h\sum_n\lambda(1,n)\overline{\lambda(1,n+h)} V\left(\frac{n}{N}\right)\overline{V\left(\frac{n+h}{N} \right)}\varphi_\pm\left(\frac{h}{H}\right)\left(\frac{n+h}{n}\right)^{it}\nonumber\\
    &\times\int_\BR U\left(v\right) \left(\frac{n+h}{n}\right)^{iMv} dv+MNt^\varepsilon,
\end{align*}
where $\varphi_\pm(y)=\varphi(\pm y)$ for some fixed function $\varphi\in C_c^\infty([1/2,2])$. Finally, we rewrite \begin{align*}
    \left(\frac{n+h}{n} \right)^{it} = e\left(\frac{t}{2\pi}\log\left(1+\frac{h}{n}\right)\right) = e\left(\frac{th}{2\pi n}\right)e\left(\frac{t}{2\pi}\left(\log\left(1+\frac{h}{n}\right)  - \frac{h}{n}\right)\right).
\end{align*}
Notice that when multiplied by the weight function $V(n/N)\varphi_\pm(h/H)$, the integral $\int_\BR U\left(v\right) \left(\frac{n+h}{n}\right)^{iMv}dv$ is $t^\varepsilon$-inert as a function of $n$ and $h$, while $e\left(\frac{t}{2\pi}\left(\log\left(1+\frac{h}{n}\right)-\frac{h}{n}\right)\right)$ is $t^\varepsilon$-inert as a function of $n$ and $h$ when $M\gg t^{1/2+\varepsilon}$. For brevity of notation, we define
\begin{equation*}
W_\pm(x,y) = V(x)\overline{V\left(x+\frac{Hy}{N}\right)}\varphi_\pm(y)e\left(\frac{t}{2\pi}\left(\log\left(1+\frac{Hy}{Nx}\right)- \frac{Hy}{Nx}\right)\right)\int_\BR U\left(v\right) \left(\frac{Nx+Hy}{Nx}\right)^{iMv} dv.
\end{equation*}
As a result, we obtain the following lemma.

\begin{Lemma}\label{SMHLemma}
    Let $\varepsilon, A>0$ and let $t^{1/2+\varepsilon}\ll M\ll t^{1-\varepsilon}$. Then there exist $t^\varepsilon$-inert functions $W_\pm\in C_c^\infty([1,2]\times[\pm1,\pm2])$ such that \begin{align*}
        \left|L\left(\frac{1}{2}+it,\pi\right)\right|^2\ll&\log t\left(1+\int_\BR U\left(\frac{v}{M}\right)\left|L\left(\frac{1}{2}+it+iv,\pi\right)\right|^2dv\right)\nonumber\\
        \ll& \log t\left(Mt^\varepsilon+\sup_{H\ll\frac{Nt^\varepsilon}{M}}\sup_{1\leq N\ll t^{3/2+\varepsilon}}\sum_\pm\frac{S_{M,H}^\pm(N)}{N}\right),
    \end{align*}
    where \begin{align*}
        S_{M,H}^\pm(N):=M\sum_h\sum_n\lambda(1,n)\overline{\lambda(1,n+h)}W_\pm\left(\frac{n}{N},\frac{h}{H}\right)e\left(\frac{th}{2\pi n}\right).
    \end{align*}
\end{Lemma}

\subsection{Trivial Bound}

Applying Cauchy-Schwarz inequality together with Lemma \ref{Ram bound}, we have the trivial bound \begin{align}\label{trivialBound}
    S_{M,H}^\pm(N)\ll M\sum_{|h|\ll H}\left(\sum_{n\ll N}|\lambda(1,n)|^2\right)^{1/2}\left(\sum_{n\ll N}|\lambda(1,n+h)|^2\right)^{1/2}\ll MHNt^\varepsilon.
\end{align}
We will apply this bound for the cases with $H\ll\frac{N}{t^{1-\varepsilon}}$. Hence we will restrict our attention to the case $H\gg\frac{N}{t^{1-\varepsilon}}$ in the rest of the paper.

\subsection{Application of the Delta Method}

Now we focus on the case $H\gg\frac{N}{t^{1-\varepsilon}}$. To bound $S_{M,H}^\pm(N)$, we start by applying the DFI delta method (Lemma \ref{DFILemma}) to separate the oscillations. Let $U\in C_c^\infty([1/2,5/2])$ be fixed such that $U(x)=1$ for $3/4\leq x\leq 9/4$. Then \begin{align*}
    S_{M,H}^\pm(N)=& M\sum_h\sum_n \lambda(1,n)W_\pm\left(\frac{n}{N},\frac{h}{H}\right)e\left(\frac{th}{2\pi n}\right)\sum_m\overline{\lambda(1,m)}U\left(\frac{m}{N}\right) \delta(m = n+h) \nonumber\\
    =&\frac{M}{Q}\sum_{1\leq q\leq Q}\frac{1}{q}\sum_h\sum_n\lambda(1,n)W_\pm\left(\frac{n}{N},\frac{h}{H}\right)e\left(\frac{th}{2\pi n}\right)\nonumber\\
    &\times\sum_m\overline{\lambda(1,m)}U\left(\frac{m}{N}\right)\sumx_{\alpha\bmod q}e\left(\frac{\alpha(n+h-m)}{q}\right)\int_\BR g(q,x)e\left(\frac{(n+h-m)x}{qQ} \right)dx.
\end{align*}
We divide the $q$-sum into segments $C\leq q<(1+10^{-10})C$ to get
\begin{align}\label{qDyadicSub}
    S_{M,H}^\pm(N) \ll t^\varepsilon\sup_{1\leq C\ll Q} S_{M,H,C}^\pm(N),
\end{align}
where \begin{align*}
    S_{M,H,C}^\pm(N)=&\frac{M}{Q}\sum_{C\leq q< (1+10^{-10})C}\frac{1}{q}\sum_h\sum_n\lambda(1,n)W_\pm\left(\frac{n}{N},\frac{h}{H}\right)e\left(\frac{th}{2\pi n}\right)\nonumber\\
    &\times\sum_m\overline{\lambda(1,m)}U\left(\frac{m}{N}\right)\sumx_{\alpha\bmod q}e\left(\frac{\alpha(n+h-m)}{q}\right)\int_\BR g(q,x)e\left(\frac{(n+h-m)x}{qQ} \right)dx.
\end{align*}

\subsection{Dual summation formulas}
We shall now apply dual summation formulas to the $h$, $m$ and $n$ sums.

\subsubsection{Analysis of the $h$-sum}
We start with an application of the Poisson summation formula to the $h$-sum $\bmod\, q$, giving us \begin{align*}
    &\sum_h W_\pm\left(\frac{n}{N},\frac{h}{H}\right)e\left(\frac{\alpha h}{q}+\frac{th}{2\pi n}+\frac{hx}{qQ}\right)\\
    =&\sum_h\frac{1}{q}\sum_{\gamma\bmod q}e\left(\frac{(\alpha+h)\gamma}{q}\right)\int_\BR W_\pm\left(\frac{n}{N},\frac{y}{H}\right)e\left(\frac{ty}{2\pi n} +\frac{xy}{qQ}-\frac{hy}{q}\right)dy\\
    =&H\sum_h\delta\left(\alpha\equiv - h\bmod q\right)\int_\BR W_\pm\left(\frac{n}{N},y\right)e\left(\frac{tHy}{2\pi n}+\frac{Hxy}{qQ}-\frac{hHy}{q}\right)dy.
\end{align*}
The above congruence condition and $(\alpha,q)=1$ imply $(h,q)=1$, so that when $h=0$, the terms with $q>1$ vanish. Since $H\gg\frac{N}{t^{1-\varepsilon}}$, choosing $Q\gg \frac{N}{t^{1-\varepsilon}}$ and repeated integration by parts gives us arbitrary savings unless \begin{align*}
    h\neq 0 \text{ and } h\sim\frac{qt}{N}\sim \frac{Ct}{N}.
\end{align*}
This gives arbitrary savings unless $C\gg\frac{N}{t^{1+\varepsilon}}$ and yields 
\begin{align*}
    S_{M,H,C}^\pm(N)=&\frac{MH}{Q}\sum_{\substack{H'\sim \frac{Ct}{N}\\ dyadic}}\sum_{C\leq q<(1+10^{-10})C}\frac{1}{q}\sum_{\substack{H'\leq h<(1+10^{-10})H'\\(h,q)=1}}\sum_n\lambda(1,n)\sum_m\overline{\lambda(1,m)}U\left(\frac{m}{N}\right)e\left(\frac{h(m-n)}{q}\right)\nonumber\\
    &\times  \int_\BR\int_\BR g(q,x)W_\pm\left(\frac{n}{N},y\right)e\left(\frac{(n + Hy-m)x}{qQ}+\frac{tHy}{2\pi n}-\frac{hHy}{q}\right)dxdy+O\left(t^{-999}\right).
\end{align*}

\subsubsection{Analysis of the $m$-sum}
Next we apply the Voronoi summation formula (Lemma \ref{gl3voronoi}) to the $m$-sum,
\begin{align*}
    &\sum_m\overline{\lambda(1,m)}U\left(\frac{m}{N}\right)e\left(\frac{hm}{q}-\frac{mx}{qQ}\right)\nonumber\\
    =&q\sum_{\eta_1=\pm}\sum_{m_0|q}\sum_m\frac{\overline{\lambda(m,m_0)}}{m_0m}S\left( \overline{h},\eta_1 m;\frac{q}{m_0}\right)\tilde{U}_{\eta_1}\left(\frac{m_0^2m}{q^3},q,x\right)+O\left(t^{-999}\right),
\end{align*}
where 
\begin{align*}
    \tilde{U}_{\eta_1}\left(a,b,c\right)=a^\frac{2}{3}\int_0^\infty U\left(\frac{w}{N}\right)\overline{U}_0\left(a,w\right)w^{-\frac{1}{3}}e\left(-\frac{cw}{bQ}+\eta_13\left(aw\right)^\frac{1}{3}\right)dw,
\end{align*}
and 
\begin{align}\label{Voronoi-weight-fn}
U_0(u,w) = \sum_{\ell=1}^L \frac{\gamma_\ell}{(uw)^{\frac{\ell-1}{3}}},
\end{align}
for a fixed large $L$, and $\gamma_\ell$ as given in Lemma \ref{voronoi}. Note that $U_0(u,w)$ is a fixed flat function. By a change of variables, we obtain 
\begin{align*}
    \tilde{U}_{\eta_1}\left(\frac{m_0^2m}{q^3},q,x\right)=\left(\frac{m_0^2mN}{q^3}\right)^\frac{2}{3}\int_0^\infty U\left(w\right)\overline{U}_0\left(\frac{m_0^2m}{q^3},Nw\right)w^{-\frac{1}{3}}e\left(f_1(w,x)\right)dw,
\end{align*}
where the phase function is 
\begin{align*}
    f_1(w,x)=\eta_13\left(\frac{m_0^2mNw}{q^3}\right)^\frac{1}{3} - \frac{Nwx}{qQ}.
\end{align*}
Repeated integration by parts gives us arbitrary savings unless 
\begin{align*}
    m_0^2m\ll \frac{N^2t^\varepsilon}{Q^3}.
\end{align*}
Writing \begin{align}\label{U1Def}
    U_{\eta_1}\left(\frac{m_0^2m}{q^3},q,x\right)=\int_0^\infty U\left(w\right)\overline{U}_0\left(\frac{m_0^2m}{q^3},Nw\right)w^{-\frac{1}{3}}e\left(f_1(w,x)\right)dw
\end{align}
we obtain 
\begin{align*}
    S_{M,H,C}^\pm(N)=&\frac{MHN^{2/3}t^\epsilon}{Q}\sum_{\eta_1=\pm}\sup_{H'\sim \frac{Ct}{N}}\sum_{C\leq q<(1+10^{-10})C}\frac{1}{q^2} \sum_{\substack{H'\leq h<(1+10^{-10})H'\\(h,q)=1}}\sum_n\lambda(1,n)\nonumber\\
    &\times \sum_{m_0|q}\sum_{m_0^2m\ll\frac{N^2t^\varepsilon}{Q^3}}\overline{\lambda(m,m_0)}\left(\frac{m_0}{m}\right)^{1/3} S\left(\overline{h},\eta_1 m;\frac{q}{m_0}\right)e\left(\frac{-hn}{q}\right)\nonumber\\
    &\times\int_\BR\int_\BR g(q,x)W_\pm\left(\frac{n}{N},y\right)U_{\eta_1}\left(\frac{m_0^2m}{q^3},q,x\right)\nonumber\\
    &\times e\left(\frac{(n+ Hy)x}{qQ}+\frac{tHy}{2\pi n}-\frac{hHy}{q}\right)dxdy+O\left(t^{-99}\right).
\end{align*}

\subsubsection{Analysis of the $n$-sum} 
Finally, we apply Voronoi formula (Lemma \ref{gl3voronoi}) to the $n$-sum. Applying similar analysis as above yields
\begin{align*}
    &\sum_n\lambda(1,n)W_\pm\left(\frac{n}{N},y\right)e\left(\frac{-hn}{q}+\frac{nx}{qQ}+\frac{tHy}{2\pi n}\right)\\
    =&\frac{N^{2/3}}{q}\sum_{\eta_2=\pm}\sum_{n_0|q}\sum_n\lambda(n,n_0)\left(\frac{n_0}{n}\right)^{1/3}S\left(-\overline{h},\eta_2 n;\frac{q}{n_0}\right)W_{0,-\eta_2}^\pm(n_0^2n,q,x,y)+O\left(t^{-999}\right),
\end{align*}
where 
\begin{align}\label{W0Def}
    W_{0,\eta_2}^\pm(n_0^2n,q,x,y)=&\int_0^\infty W_\pm\left(w,y\right)U_{0}\left(\frac{n_0^2n}{q^3},Nw\right)w^{-\frac{1}{3}}\nonumber\\
    &\times e\left(\frac{Nwx}{qQ}+\frac{tHy}{2\pi Nw}-\eta_23\left(\frac{n_0^2nNw}{q^3}\right)^\frac{1}{3}\right)dw,
\end{align}
where $U_{0}$ is given in \eqref{Voronoi-weight-fn}. Since $H\gg\frac{N}{t^{1-\varepsilon}}$, repeated integration by parts gives us arbitrary savings unless \begin{align*}
    n_0^2n\ll \left(\frac{N^2}{Q^3}+\frac{(CHt)^3}{N^4}\right)t^\varepsilon.
\end{align*}
By choosing $Q^2\ll NM/t^{1-\varepsilon}$, the first term on the right side dominates, so that $n_0^2n\ll N^2t^\varepsilon/Q^3$. Changing $\eta_2$ to $-\eta_2$ for ease of notation, this yields 
\begin{align}\label{SMHAfterAllDual}
    S_{M,H,C}^\pm(N)=&\frac{MHN^{4/3}t^\epsilon}{Q}\sum_{\eta_1,\eta_2=\pm}\sup_{H'\sim \frac{Ct}{N}}\,\sum_{C\leq q<(1+10^{-10})C} \frac{1}{q^3}\sum_{\substack{H'\leq h<(1+10^{-10})H'\\(h,q)=1}}\sum_{n_0|q}\,\sum_{n_0^2n\ll \frac{N^2t^\varepsilon}{Q^3}}\lambda(n,n_0)\left(\frac{n_0}{n}\right)^{1/3}\nonumber\\
    &\times \sum_{m_0|q}\sum_{m_0^2m\ll\frac{N^2t^\varepsilon}{Q^3}}\overline{\lambda(m,m_0)}\left(\frac{m_0}{m}\right)^{1/3}S\left(\overline{h},\eta_1 m;\frac{q}{m_0}\right)S\left(\overline{h},\eta_2 n;\frac{q}{n_0}\right)\nonumber\\
    &\times \int_\BR\int_\BR g(q,x)W_{0,\eta_2}^\pm(n_0^2n,q,x,y)U_{\eta_1}\left(\frac{m_0^2m}{q^3},q,x\right)e\left(\frac{Hxy}{qQ}-\frac{hHy}{q}\right)dydx+O\left(t^{-99}\right).
\end{align}

\subsubsection{Analysis of integral transforms obtained by dual summations} 

We now consider the $y,w$-integral. Then (\ref{W0Def}) and a change of variable gives us \begin{align*}
    &\int_\BR W_{0,\eta_2}^\pm(n_0^2n,q,x,y)e\left(\frac{Hxy}{qQ}-\frac{hHy}{q}\right)dy\\
    =&\int_0^\infty U_0\left(\frac{n_0^2n}{q^3},Nw\right)w^{-\frac{1}{3}}e\left(\frac{Nwx}{qQ}-\eta_23\left(\frac{n_0^2nNw}{q^3}\right)^\frac{1}{3}\right)\\
    &\times\int_\BR W_\pm(w,y)e\left(\frac{tHy}{2\pi Nw}+\frac{Hxy}{qQ}-\frac{hHy}{q}\right)dydw\\
    =&e\left(\frac{tx}{2\pi hQ}\right)\int_\BR U_0\left(\frac{n_0^2n}{q^3},\frac{qt}{2\pi h}+Nu\right)\left(\frac{qt}{2\pi hN}+u\right)^{-\frac{1}{3}}\\
    &\times e\left(\frac{Nxu}{qQ}-\eta_23\left(\frac{n_0^2nN}{q^3}\left(\frac{qt}{2\pi hN}+u\right)\right)^\frac{1}{3}\right)\\
    &\times \int_\BR W_\pm\left(\frac{qt}{2\pi hN}+u,y\right)e\left(\frac{tHy}{2\pi N}\left(\frac{qt}{2\pi hN}+u\right)^{-1}+\frac{Hxy}{qQ}-\frac{hHy}{q}\right)dydu.
\end{align*}
Since $C\gg\frac{N}{t^{1+\varepsilon}}$ and $H\ll\frac{Nt^\varepsilon}{M}$, choosing $MQ>t^{1+\varepsilon}$ gives us $\frac{Hxy}{qQ}\ll\frac{t^{1+\varepsilon}}{MQ}\ll1$. Hence repeated integration by parts on the $y$-integral gives us arbitrary savings unless \begin{align*}
    \left|\frac{tH}{2\pi N}\left(\frac{qt}{2\pi hN}+u\right)^{-1}-\frac{hH}{q}\right|\ll t^{\varepsilon/2} & \Longleftrightarrow  \left|\frac{1}{1+\frac{2\pi hNu}{qt}}-1\right|\ll\frac{qt^\varepsilon}{hH}\ll\frac{N}{Ht^{1-\varepsilon/2}}\\
    &\Longleftrightarrow |u|\ll\frac{N}{Ht^{1-\varepsilon/2}}.
\end{align*}
Take a fixed function $U\in C_c([-2,-2])$ such that $U(x)=1$ for $-1\leq x\leq 1$, and let \begin{align*}
    \phi_{1}\left(a,b,c,d\right)=&\left(\frac{bt}{2\pi cN}+\frac{Na}{Ht^{1-\varepsilon}}\right)^{-\frac{1}{3}}\\
    &\times\int_\BR W_\pm\left(\frac{bt}{2\pi cN}+\frac{Na}{Ht^{1-\varepsilon}},y\right)e\left(\frac{tHy}{2\pi N}\left(\frac{bt}{2\pi cN}+\frac{Na}{Ht^{1-\varepsilon}}\right)^{-1}+\frac{dHy}{bQ}-\frac{cHy}{b}\right)dy.
\end{align*}
Then $\phi_{1}$ is $t^\varepsilon$-inert and the $y,w$-integral becomes \begin{align*}
    &e\left(\frac{tx}{2\pi hQ}\right)\int_\BR U\left(\frac{uHt^{1-\varepsilon}}{N}\right)\phi_{1}\left(\frac{uHt^{1-\varepsilon}}{N},q,h,x\right)U_{0}\left(\frac{n_0^2n}{q^3},\frac{qt}{2\pi h}+Nu\right)\nonumber\\
    &\times e\left(\frac{Nxu}{qQ}-\eta_23\left(\frac{n_0^2nN}{q^3}\left(\frac{qt}{2\pi hN}+u\right)\right)^\frac{1}{3}\right)du+O\left(t^{-999}\right)\nonumber\\
    =&\frac{N}{Ht^{1-\varepsilon}}e\left(\frac{tx}{2\pi hQ}\right)\int_\BR U(u) \phi_{1}\left(u,q,h,x\right)U_{0}\left(\frac{n_0^2n}{q^3},\frac{qt}{2\pi h}+\frac{N^2u}{Ht^{1-\varepsilon}}\right)\nonumber\\
    &\times e\left(\frac{N^2xu}{qQHt^{1-\varepsilon}}-\eta_23\left(\frac{n_0^2nN}{q^3}\left(\frac{qt}{2\pi hN}+\frac{Nu}{Ht^{1-\varepsilon}}\right)\right)^\frac{1}{3}\right)du+O\left(t^{-999}\right).
\end{align*}
Defining \begin{align}\label{W1Def}
    W_{\eta_2}^\pm(n_0^2n,q,h,u):=&U_{0}\left(\frac{n_0^2n}{q^3},\frac{qt}{2\pi h}+\frac{N^2u}{Ht^{1-\varepsilon}}\right)e\left(-\eta_23\left(\frac{n_0^2nN}{q^3}\left(\frac{qt}{2\pi hN}+\frac{Nu}{Ht^{1-\varepsilon}}\right)\right)^\frac{1}{3}\right),
\end{align}
we have \begin{align*}
    S_{M,H,C}^\pm(N)=&\frac{MN^{7/3}t^\epsilon}{Qt^{1-\varepsilon}}\sum_{\eta_1,\eta_2=\pm}\sup_{H'\sim \frac{Ct}{N}}\,\sum_{C\leq q<(1+10^{-10})C} \frac{1}{q^3}\sum_{\substack{H'\leq h<(1+10^{-10})H'\\(h,q)=1}}\sum_{n_0|q}\sum_{n_0^2n\ll \left(\frac{N^2}{Q^3}+\frac{(CHt)^3}{N^4}\right)t^\varepsilon}\nonumber\\
    &\times \lambda(n,n_0)\left(\frac{n_0}{n}\right)^{1/3}\sum_{m_0|q}\sum_{m_0^2m\ll\frac{N^2t^\varepsilon}{Q^3}}\overline{\lambda(m,m_0)}\left(\frac{m_0}{m}\right)^{1/3}S\left(\overline{h},\eta_1 m;\frac{q}{m_0}\right)S\left(\overline{h},\eta_2 n;\frac{q}{n_0}\right)\nonumber\\
    &\times \int_{\BR}U(u)W_{\eta_2}^\pm(n_0^2n,q,h,u)\int_\BR g(q,x)\phi_{1}\left(u,q,h,x\right)\nonumber\\
    &\times U_{\eta_1}\left(\frac{m_0^2m}{q^3},q,x\right)e\left(\frac{N^2xu}{qQHt^{1-\varepsilon}}+\frac{tx}{2\pi hQ}\right)dxdu+O\left(t^{-99}\right).
\end{align*}
Applying smooth dyadic subdivisions on the $m_0^2m,n_0^2n$ sums, we get 
\begin{align*}
    S_{M,H,C}^\pm(N)\ll& t^\varepsilon\sup_{\eta_1,\eta_2=\pm}\sup_{N_1\ll\frac{N^2t^\varepsilon}{Q^3}}\sup_{N_2\ll\frac{N^2t^\varepsilon}{Q^3}}\sup_{H'\sim\frac{Ct}{N}}\frac{MN^{7/3}}{Qt}\sum_{C\leq q<(1+10^{-10})C}\frac{1}{q^3}\sum_{\substack{H'\leq h\leq (1+10^{-10})H'\\(h,q)=1}}\nonumber\\
    &\times \sum_{m_0|q}\sum_m\overline{\lambda(m,m_0)}\left(\frac{m_0}{m}\right)^{1/3}\varphi\left(\frac{m_0^2m}{N_1}\right)\sum_{n_0|q}\sum_n \lambda(n,n_0)\left(\frac{n_0}{n}\right)^{1/3}\varphi\left(\frac{n_0^2n}{N_2}\right)\\ &\times S\left(\overline{h},\eta_1 m;\frac{q}{m_0}\right) S\left(\overline{h},\eta_2 n;\frac{q}{n_0}\right)\int_\BR U(u) W_{\eta_2}^\pm(n_0^2n,q,h,u)\nonumber\\
    &\times \int_\BR g(q,x)\phi_{1}\left(u,q,h,x\right)U_{\eta_1}\left(\frac{m_0^2m}{q^3},q,x\right)e\left(\frac{N^2xu}{qQHt^{1-\varepsilon}}+\frac{tx}{2\pi hQ}\right)dxdu+t^{-99}
\end{align*}
for some fixed $\varphi\in C_c^\infty([1/2,2])$. 

\subsubsection{$x$-integral} We get arbitrary savings for $|x|\gg t^\epsilon$ due to bounds on $g(q,x)$ as given in \eqref{gqx}. So we can restrict $|x|\ll t^\epsilon$. Moreover, we expect to obtain cancellations in the $x$-integral when $|x|\gg CQt^\epsilon/N$. Therefore we use a smooth partition of unity to split the $x$-integral into dyadic segments and a small `remainder' part of size $CQt^\epsilon/N$. 

Let $\varphi:(0,\infty)\rightarrow[0,1]$ be a smooth compactly supported function satisfying 
\begin{equation}\label{PhiXDef}
supp(\varphi) \subset [1/2, 2] \quad \text{ and } \quad \sum_{j\in\BZ}\varphi\bigg(\frac{y}{2^j}\bigg) = 1 \, \text{ for } \, y\in(0,\infty).
\end{equation}
Define $\varphi_0(0)=1$, and for $y\in \BR\backslash\{0\}$
\begin{equation*}
\varphi_0(y) := \sum_{i\geq1} \varphi\bigg(\frac{2^{i}y}{CQt^\epsilon/N}\bigg) + \varphi\bigg(\frac{-2^{i}y}{CQt^\epsilon/N}\bigg) \quad \text{ and } \quad \varphi_X(y) := \varphi\bigg(\frac{y}{X}\bigg) \, \text{ for } |X|\geq \frac{2^{j-1}CQt^\epsilon}{N}\, \text{ and }\, j\geq1.
\end{equation*}
Then $\varphi_0(y)$ is smooth and supported in $\displaystyle \bigg(\frac{-CQt^\epsilon}{N}, \frac{CQt^\epsilon}{N}\bigg)$ with $\varphi_0(y)=1$ for $\displaystyle y\in\bigg[\frac{-CQt^\epsilon}{2N}, \frac{CQt^\epsilon}{2N}\bigg]$.

\iffalse
i.e. let $\varphi_{Y}(y)=\varphi\left(\frac{y}{Y}\right)$ with the same fixed $\varphi\in C_c^\infty([1/2,5/2])$, we have \begin{align*}
    \varphi_0\left(x\right):=1-\sum_{\eta_0=\pm}\sum_{\substack{\frac{CQt^\varepsilon}{N}\leq X \leq t^\varepsilon\\\text{Dyadic}}}\varphi_{\eta_0X}\left(x\right)
\end{align*}
is smooth and supported on $$(-\infty,-t^\varepsilon)\bigcup\left(-\frac{CQt^\varepsilon}{N},\frac{CQt^\varepsilon}{N}\right)\bigcup(t^\varepsilon,\infty).$$
\fi

Hence we have 
\begin{align*}
    S_{M,H,C}^\pm(N)\ll& t^\varepsilon\sup_{\substack{\frac{CQt^\varepsilon}{N}\leq |X| \leq t^\varepsilon\\ \text{or } X=0}}\sup_{N_1\ll\frac{N^2t^\varepsilon}{Q^3}}\sup_{N_2\ll\frac{N^2t^\varepsilon}{Q^3}}\sup_{H'\sim\frac{Ct}{N}}\sup_{\eta_1,\eta_2=\pm}\frac{MN^{7/3}}{Qt}\Bigg|\sum_{C\leq q<(1+10^{-10})C}\frac{1}{q^3}\nonumber\\
    &\times \sum_{\substack{H'\leq h<(1+10^{-10})H'\\(h,q)=1}}\sum_{m_0|q}\sum_m\overline{\lambda(m,m_0)}\left(\frac{m_0}{m}\right)^{1/3}\varphi\left(\frac{m_0^2m}{N_1}\right)\sum_{n_0|q}\sum_n\lambda(n,n_0)\left(\frac{n_0}{n}\right)^{1/3}\varphi\left(\frac{n_0^2n}{N_2}\right)\nonumber\\
    &\times S\left(\overline{h},\eta_1 m;\frac{q}{m_0}\right) S\left(\overline{h},\eta_2 n;\frac{q}{n_0}\right)\int_\BR U(u) W_{\eta_2}^\pm(n_0^2n,q,h,u)\nonumber\\
    &\times \int_\BR \varphi_X(x)g(q,x)\phi_{1}\left(u,q,h,x\right)U_{\eta_1}\left(\frac{m_0^2m}{q^3},q,x\right)e\left(\frac{N^2xu}{qQHt^{1-\varepsilon}}+\frac{tx}{2\pi hQ}\right)dxdu\Bigg|+t^{-99}.
\end{align*}

\iffalse
Note that by choosing $Q^2<Nt^{-\varepsilon}$, the bounds for $g(q,x)$ in Lemma \ref{DFILemma} and (\ref{gqxBound}) implies the bound \begin{align}\label{phi0supp}
    \varphi_0(x)g(q,x)\ll t^\varepsilon\delta\left(|x|\ll\frac{CQt^\varepsilon}{N}\right)+t^{-A}
\end{align}
for any $A>0$.
\fi

\subsection{Cauchy-Schwarz inequality and the Duality principle}

We now apply Cauchy-Schwarz inequality to take out everything in (\ref{SMHAfterAllDual}) except the $m,m_0,n,n_0$-sums and the $x$-integral, giving us 
\begin{align*}
    S_{M,H,C}^\pm(N)\ll& t^\varepsilon\sup_{\eta_1,\eta_2=\pm}\sup_{\substack{\frac{CQt^\varepsilon}{N}\leq |X| \leq t^\varepsilon\\ \text{or } X=0}}\sup_{N_1\ll\frac{N^2t^\varepsilon}{Q^3}}\sup_{N_2\ll\frac{N^2t^\varepsilon}{Q^3}}\sup_{H'\sim\frac{Ct}{N}}\left(\tilde{S}_{1,X}(N)\tilde{S}_{2}(N)\right)^{1/2}+t^{-99},
\end{align*}
where 
\begin{align*}
    \tilde{S}_{1,X}(N):=&\tilde{S}_{1,M,H,H',C,X}^{\pm,\eta_1}(N,N_1)\nonumber\\
    =&\frac{MN^{7/3}}{Qt}\sum_{C\leq q<(1+10^{-10})C} \frac{1}{q^3} \sum_{\substack{H'\leq h<(1+10^{-10})H'\\(h,q)=1}}\int_\BR U(u)\nonumber\\
    &\times \left|\sum_{m_0|q}\sum_m\overline{\lambda(m,m_0)}\left(\frac{m_0}{m}\right)^{1/3}\varphi\left(\frac{m_0^2m}{N_1}\right)S\left(\overline{h},\eta_1 m;\frac{q}{m_0}\right)\right.\nonumber\\
    &\times \left.\int_\BR \varphi_X(x)g(q,x)\phi_{1}\left(u,q,h,x\right)U_{\eta_1}\left(\frac{m_0^2m}{q^3},q,x\right)e\left(\frac{N^2xu}{qQHt^{1-\varepsilon}}+\frac{tx}{2\pi hQ}\right)dx\right|^2du
\end{align*}
and 
\begin{align*}
    \tilde{S}_{2}(N):=&\tilde{S}_{2,M,H,H',C}^{\pm,\eta_2}(N,N_2)\nonumber\\
    =&\frac{MN^{7/3}}{Qt}\sum_{C\leq q<(1+10^{-10})C} \frac{1}{q^3} \sum_{\substack{H'\leq h<(1+10^{-10})H'\\(h,q)=1}}\int_\BR U(u)\nonumber\\
    &\times\left|\sum_{n_0|q}\sum_n\lambda(n,n_0)\left(\frac{n_0}{n}\right)^{1/3}\varphi\left(\frac{n_0^2n}{N_2}\right)S\left(\overline{h},\eta_2 n;\frac{q}{n_0}\right) W_{\eta_2}^\pm(n_0^2n,q,h,u)\right|^2du.
\end{align*}

By the Duality principle (Lemma \ref{dualitylemma}) and Lemma \ref{Ram bound}, we have the following theorem.

\begin{Theorem}\label{S1S2Thm}
For any positive numbers $\sqrt{t}\ll M\ll N$, $\frac{N}{t^{1-\varepsilon}}\ll H\ll\frac{Nt^\varepsilon}{M}$, $Q\gg \frac{N}{t^{1-\varepsilon}}$, $C\ll Q$, we have 
\begin{align*}
    S_{M,H,C}^\pm(N)\ll t^{-99}
\end{align*}
unless $C\gg\frac{N}{t^{1+\varepsilon}}$. When $C\gg\frac{N}{t^{1+\varepsilon}}$, we have 
\begin{align*}
    S_{M,H,C}^\pm(N)\ll& t^\varepsilon\sup_{|u|\ll1}\sup_{\eta_1,\eta_2=\pm}\sup_{\substack{\frac{CQt^\varepsilon}{N}\leq |X| \leq t^\varepsilon\\ \text{or } X=0}}\sup_{N_1\ll\frac{N^2t^\varepsilon}{Q^3}}\sup_{N_2\ll\frac{N^2t^\varepsilon}{Q^3}}\sup_{H'\sim\frac{Ct}{N}}\left(S_{1,X}(N)S_{2}(N)\right)^{1/2}+t^{-99},
\end{align*}
where 
\begin{align}\label{S1AfterDuality}
        S_{1,X}(N):=&S_{1,M,H,H',C,X}^{\pm,\eta_1}(N,N_1)\nonumber\\
        =&\frac{MN_1^{1/3}N^{7/3}}{Qt^{1-\varepsilon}}\sup_{\|\beta\|_2=1}\sum_{m_0}\frac{1}{m_0}\sum_m\varphi\left(\frac{m_0^2m}{N_1}\right)\left|\sum_{\frac{C}{m_0}\leq q< (1+10^{-10})\frac{C}{m_0}}\frac{1}{q^{3/2}}\sum_{\substack{H'\leq h<(1+10^{-10})H'\\(h,qm_0)=1}}\right.\nonumber\\
        &\times S\left(\overline{h} ,\eta_1 m;q\right)\int_\BR U(u) \beta(m_0q,h,u)\int_\BR \varphi_X(x)g(m_0q,x)\phi_{1}\left(u,m_0q,h,x\right)\nonumber\\
        &\times \left.U_{\eta_1} \left(\frac{m}{q^3m_0},m_0q,x\right)e\left(\frac{N^2xu}{m_0qQHt^{1-\varepsilon}}+\frac{tx}{2\pi hQ}\right)dxdu\right|^2,
    \end{align}
    and \begin{align}\label{S2AfterDuality}
        S_{2}(N):=&S_{2,M,H,H',C,u}^{\pm,\eta_2}(N,N_2)\nonumber\\
        =&\frac{MN_2^{1/3}N^{7/3}}{Qt^{1-\varepsilon}}\sup_{\|\beta\|_2=1}\sum_{n_0}\frac{1}{n_0}\sum_n\varphi\left(\frac{n_0^2n}{N_2}\right)\left|\sum_{\frac{C}{n_0}\leq q< (1+10^{-10})\frac{C}{n_0}}\frac{1}{q^{3/2}}\sum_{\substack{H'\leq h<(1+10^{-10})H'\\(h,qn_0)=1}}S\left(\overline{h},\eta_2 n;q\right)\right.\nonumber\\
        &\times 
        \left.\int_\BR U(u)\beta(n_0q,h,u) W_{\eta_2}^\pm(n_0^2n,n_0q,h,u)du\right|^2,
    \end{align}
    with $U_{\eta_1}, W_{\eta_2}^\pm$ as defined in (\ref{U1Def}) and (\ref{W1Def}) respectively.
\end{Theorem}

\section{Opening the square and Poisson summation to \texorpdfstring{$S_{1,X}(N)$}{S1X(N)}}

In this section we focus on the analysis of $S_{1,X}(N)$. We will then apply similar treatment to $S_{2}(N)$ in the next section.

\subsection{Opening the square and Poisson summation}

Opening the square in (\ref{S1AfterDuality}), we get \begin{align*}
    S_{1,X}(N)=&\frac{MN_1^{1/3}N^{7/3}}{Qt^{1-\varepsilon}}\sup_{\|\beta\|_2=1}\sum_{m_0}\frac{1}{m_0}\sum_m\varphi\left(\frac{m_0^2m}{N_1}\right)\mathop{\sum\sum}_{\frac{C}{m_0}\leq q_1,q_2< (1+10^{-10})\frac{C}{m_0}}\frac{1}{(q_1q_2)^{3/2}}\nonumber\\
    &\times \mathop{\sum\sum}_{\substack{H'\leq h_1, h_2<(1+10^{-10})H'\\(h_1,q_1m_0)=1\\(h_2,q_2m_0)=1}}S\left(\overline{h_1} ,\eta_1 m;q_1\right)S\left(\overline{h_2} ,\eta_1 m;q_2\right)\nonumber\\
    &\times \int_\BR U(u_1)\beta(m_0q_1,h_1,u_1)\int_\BR \varphi_X(x_1)g(m_0q_1,x_1)\phi_{1}\left(u_1,m_0q_1,h_1,x_1\right) \nonumber\\
    &\times U_{\eta_1} \left(\frac{m}{q_1^3m_0},m_0q_1,x_1\right)e\left(\frac{N^2x_1u_1}{m_0q_1QHt^{1-\varepsilon}}+\frac{tx_1}{2\pi h_1Q}\right)dx_1du_1\nonumber\\
    &\times \int_\BR U(u_2)\overline{\beta(m_0q_2,h_2,u_2)}\int_\BR \varphi_X(x_2)\hspace{0.1cm} \overline{g(m_0q_2,x_2)\phi_{1}\left(u_2,m_0q_2,h_2,x_2\right)} \nonumber\\
    &\times \overline{U_{\eta_1} \left(\frac{m}{q_2^3m_0},m_0q_2,x_2\right)}e\left(-\frac{N^2x_2u_2}{m_0q_2QHt^{1-\varepsilon}}-\frac{tx_2}{2\pi h_2Q}\right)dx_2du_2.
\end{align*}
Applying Poisson summation to the $m$-sum, we have 
\begin{align*}
    &\sum_m\varphi\left(\frac{m_0^2m}{N_1}\right)S\left(\overline{h_1} ,\eta_1 m;q_1\right)S\left(\overline{h_2} ,\eta_1 m;q_2\right) U_{\eta_1} \left(\frac{m}{q_1^3m_0},m_0q_1,x_1\right)\overline{U_{\eta_1} \left(\frac{m}{q_2^3m_0},m_0q_2,x_2\right)}\nonumber\\
    =&\frac{N_1}{m_0^2q_1q_2}\sum_m\mathcal{C}_{\eta_1}(m,h_1,h_2,q_1q_2)\mathcal{J}_0(m,m_0,q_1,q_2,x_1,x_2),
\end{align*}
where 
\begin{align}\label{CharDef}
    \mathcal{C}_{\eta_1}(m,h_1,h_2,q_1q_2)=\sum_{\gamma\bmod q_1q_2}S\left(\overline{h_1} ,\eta_1 \gamma;q_1\right)S\left(\overline{h_2} ,\eta_1 \gamma;q_2\right)e\left(\frac{m\gamma}{q_1q_2}\right)
\end{align}
and \begin{align}\label{J11Def}
    &\mathcal{J}_0(m,m_0,q_1,q_2,x_1,x_2)\nonumber\\
    =&\int_0^\infty \varphi(z)U_{\eta_1} \left(\frac{N_1z}{(m_0q_1)^3},m_0q_1,x_1\right)\overline{U_{\eta_1} \left(\frac{N_1z}{(m_0q_2)^3},m_0q_2,x_2\right)}e\left(-\frac{mN_1z}{m_0^2q_1q_2}\right)dz.
\end{align}

Now applying stationary phase analysis with the calculations shown in appendix \ref{sect.J0IntegralAnalysis} gives us
\begin{align}\label{S10}
    S_{1,0}(N)\ll&\frac{QMN_1^{4/3}N^{1/3}}{Ct^{1-\varepsilon}}\delta\left(N_1\ll \frac{C^3t^\varepsilon}{N}\right)\sup_{\|\beta\|_2=1}\sum_{m_0}\mathop{\sum\sum}_{\frac{C}{m_0}\leq q_1,q_2< (1+10^{-10})\frac{C}{m_0}}\frac{1}{q_1q_2}\mathop{\sum\sum}_{\substack{H'\leq h_1, h_2<(1+10^{-10})H'\\(h_1,q_1m_0)=1\\(h_2,q_2m_0)=1}}\nonumber\\
    &\times \int_{|u_1|\ll1}\int_{|u_2|\ll1}|\beta(m_0q_1,h_1,u_1)\beta(m_0q_2,h_2,u_2)|du_2du_1\sum_{|m|\ll\frac{C^2t^\varepsilon}{N_1}}|\mathcal{C}_{\eta_1}(m,h_1,h_2,q_1q_2)|+t^{-99},
\end{align}
and for $X\neq0$,  \begin{align}\label{S1XAfterFirstCauchy}
    S_{1,X}(N)\ll&\frac{CMN_1N^2}{Qt^{1-\varepsilon}}\delta(\eta_1X>0)\sup_{\|\beta\|_2=1}\sum_{m_0}m_0^{-3}\mathop{\sum\sum}_{\frac{C}{m_0}\leq q_1,q_2< (1+10^{-10})\frac{C}{m_0}}(q_1q_2)^{-5/2}\nonumber\\
    &\times \mathop{\sum\sum}_{\substack{H'\leq h_1, h_2<(1+10^{-10})H'\\(h_1,q_1m_0)=1\\(h_2,q_2m_0)=1}}\sum_{|m|\ll\frac{CQ^2}{NX^2}t^\varepsilon}\mathcal{C}_{\eta_1}(m,h_1,h_2,q_1q_2)\nonumber\\
    &\times \int_\BR\int_\BR U(u_1)U(u_2)\beta(m_0q_1,h_1,u_1)\overline{\beta(m_0q_2,h_2,u_2)}\mathcal{J}(q_1,q_2,\vec{v}\,)du_2du_1+t^{-99},
\end{align}
where for brevity of notation, we write $\vec{v}=(n,n_0,h_1,h_2,u_1,u_2)$,  \begin{align}\label{JDef}
    \mathcal{J}(q_1,q_2,\vec{v}\,)=&\int_\BR \varphi_X(x_1)g(m_0q_1,x_1)\phi_{1}\left(u_1,m_0q_1,h_1,x_1\right)\nonumber\\
    &\times \int_\BR \varphi_X(x_2)\overline{g(m_0q_2,x_2)\phi_{1}\left(u_2,m_0q_2,h_2,x_2\right)}\nonumber\\
    &\times \int_0^\infty \varphi(z)U_3 \left(\frac{N_1Q^3z}{N^2|x_1|^3},m_0q_1,x_1\right)\overline{U_3 \left(\frac{N_1Q^3z}{N^2|x_2|^3},m_0q_2,x_2\right)}\nonumber\\
    &\times e\left(2\sqrt{\frac{N_1Qz}{m_0^2q_1^2|x_1|}}-2\sqrt{\frac{N_1Qz}{m_0^2q_2^2|x_2|}}-\frac{mN_1z}{m_0^2q_1q_2}\right)dz\nonumber\\
    &\times e\left(\frac{N^2x_1u_1}{m_0q_1QHt^{1-\varepsilon}}-\frac{N^2x_2u_2}{m_0q_2QHt^{1-\varepsilon}}+\frac{tx_1}{2\pi h_1Q}-\frac{tx_2}{2\pi h_2Q}\right)dx_2dx_1,
\end{align}
and $U_3$ is some $t^\varepsilon$-inert function supported on $[1/4,25/4]\times\R\times\R$.

For $X\neq0$, applying stationary phase analysis on the $x_1,x_2$-integrals, calculations in appendix \ref{sect.JIntegralAnalysis} implies that there exists a $t^\varepsilon$-inert function $\Phi$ supported on $[1/4,25/4]\times\BR^4$ such that \begin{align}\label{S1XAfterIntegral}
    S_{1,X}(N)=&\frac{C^2M\sqrt{N_1}N^2X^{5/2}}{Q^{3/2}t^{1-\varepsilon}}\delta(\eta_1X>0)\sup_{\|\beta\|_2=1}\sum_{m_0}m_0^{-3}\mathop{\sum\sum}_{\frac{C}{m_0}\leq q_1,q_2< (1+10^{-10})\frac{C}{m_0}}(q_1q_2)^{-5/2}\mathop{\sum\sum}_{\substack{H'\leq h_1, h_2<(1+10^{-10})H'\\(h_1,q_1m_0)=1\\(h_2,q_2m_0)=1}}\nonumber\\
    &\times \sum_{|m|\ll\frac{CQ^2}{NX^2}t^\varepsilon}\mathcal{C}_{\eta_1}(m,h_1,h_2,q_1q_2)\int_\BR\int_\BR U(u_1)U(u_2) \beta(m_0q_1,h_1,u_1)\overline{\beta(m_0q_2,h_2,u_2)}\nonumber\\
    &\times \int_0^\infty \varphi(z)\Phi \left(\frac{N_1Q^3z}{N^2X^3},m_0q_1,m_0q_2,u_1,u_2\right)g(m_0q_1,Xx_{1,0})g(m_0q_2,Xx_{2,0})\nonumber\\
    &\times e\left(3\left(\frac{N_1tz}{2\pi h_1(m_0q_1)^2}\left(1+\frac{2\pi h_1N^2u_1}{m_0q_1Ht^{2-\varepsilon}}\right)\right)^{1/3}\right)\nonumber\\
    &\times e\left(-3\left(\frac{N_1tz}{2\pi h_2(m_0q_2)^2}\left(1+\frac{2\pi h_2N^2u_2}{m_0q_2Ht^{2-\varepsilon}}\right)\right)^{1/3}-\frac{mN_1z}{m_0^2q_1q_2}\right)dzdu_2du_1+O\left(t^{-99}\right),
\end{align}
and for $j=1,2$, \begin{align}\label{xj0Def}
    x_{j,0}=\left(\frac{2\pi h_j}{m_0q_jt}\right)^{2/3}\frac{Q(N_1z)^{1/3}}{X}\left(1+\frac{2\pi h_jN^2u_j}{m_0q_jHt^{2-\varepsilon}}\right)^{-2/3}.
\end{align}

\subsection{Bounding \texorpdfstring{$S_{1,0}(N)$}{S10(N)}}

To bound $S_{1,0}(N)$, we are left to bound the character sum $\mathcal{C}_{\eta_1}(m,h_1,h_2,q_1q_2)$ in (\ref{S10}). Recalling the definition in (\ref{CharDef}), we get for $m=0$,  
\begin{align}\label{CharSum0}
    \mathcal{C}_{\eta_1}(0,h_1,h_2,q_1q_2)=&\sum_{\gamma\bmod q_1q_2}S\left(\overline{h_1} ,\eta_1 \gamma;q_1\right)S\left(\overline{h_2} ,\eta_1 \gamma;q_2\right)\nonumber\\
    =&\sumx_{\alpha_1\bmod q_1}e\left(\frac{\alpha_1\overline{h_1}}{q_1}\right)\sumx_{\alpha_2\bmod q_2}e\left(\frac{\alpha_2\overline{h_2}}{q_2}\right)\sum_{\gamma\bmod q_1q_2}e\left(\eta_1\frac{(\overline{\alpha_1}q_2+\overline{\alpha_2}q_1)\gamma}{q_1q_2}\right)\nonumber\\
    =&q_1q_2\sumx_{\alpha_1\bmod q_1}e\left(\frac{\alpha_1\overline{h_1}}{q_1}\right)\sumx_{\alpha_2\bmod q_2}e\left(\frac{\alpha_2\overline{h_2}}{q_2}\right)\delta\left(\overline{\alpha_1}q_2\equiv-\overline{\alpha_2}q_1\bmod q_1q_2\right)\nonumber\\
    =&q_1^2\sumx_{\alpha\bmod q_1}e\left(\frac{\alpha(\overline{h_1}-\overline{h_2})}{q_1}\right)\delta(q_1=q_2)\nonumber\\
    =&q_1^2\sum_{q_1'|q_1}q_1'\mu\left(\frac{q}{q_1'}\right)\delta\left(h_1\equiv h_2\bmod q_1', q_1=q_2\right).
\end{align}
Hence the contribution of $m=0$ to $S_{1,0}(N)$ is bounded by \begin{align*}
    \ll&\frac{QMN_1^{4/3}N^{1/3}}{Ct^{1-\varepsilon}}\delta\left(N_1\ll \frac{C^3t^\varepsilon}{N}\right)\sup_{\|\beta\|_2=1}\sum_{m_0}\sum_{\frac{C}{m_0}\leq q< (1+10^{-10})\frac{C}{m_0}}\sum_{q'|q}q'\mathop{\sum\sum}_{\substack{H'\leq h_1, h_2<(1+10^{-10})H'\\(h_1,q_1m_0)=1\\(h_2,q_2m_0)=1\\ h_1\equiv h_2\bmod{q'}}}\\
    &\times \int_{|u_1|\ll1}\int_{|u_2|\ll1}|\beta(m_0q,h_1,u_1)\beta(m_0q,h_2,u_2)|du_2du_1.
\end{align*}
Applying the AM-GM inequality \begin{align*}
    |\beta(m_0q,h_1,u_1)\beta(m_0q,h_2,u_2)|\ll |\beta(m_0q,h_1,u_1)|^2+|\beta(m_0q,h_2,u_2)|^2,
\end{align*}
we get the above is bounded by \begin{align*}
    \ll& \frac{QMN_1^{4/3}N^{1/3}}{Ct^{1-\varepsilon}}\delta\left(N_1\ll\frac{C^3t^\varepsilon}{N}\right)\sup_{\|\beta\|_2=1}(|\mathfrak{s}_{0,1}|+|\mathfrak{s}_{0,2}|),
\end{align*}
where for brevity of notation, we have temporarily defined \begin{align*}
    \mathfrak{s}_{0,\nu}:=\sum_{m_0}\sum_{\frac{C}{m_0}\leq q< (1+10^{-10})\frac{C}{m_0}}\sum_{q'|q}q'\mathop{\sum\sum}_{\substack{H'\leq h_1, h_2<(1+10^{-10})H'\\h_1\equiv h_2\bmod{q'}}}\int_{|u_1|\ll1}\int_{|u_2|\ll1}|\beta(m_0q,h_\nu,u_\nu)|^2du_2du_1,
\end{align*}
for $\nu=1,2$. Hence the contribution of $m=0$ to $S_{1,0}(N)$ is bounded by
\begin{align}\label{S100Bound}
    \ll&\frac{QMN_1^{4/3}N^{1/3}}{t^{1-\varepsilon}}\left(1+\frac{t}{N}\right)\delta\left(N_1\ll\frac{C^3t^\varepsilon}{N}\right).
\end{align}

For the case $m\neq 0$, the character sum is \begin{align}\label{CharSumNon0}
    \mathcal{C}_{\eta_1}(m,h_1,h_2,q_1q_2)=&\sum_{\gamma\bmod q_1q_2}S\left(\overline{h_1} ,\eta_1 \gamma;q_1\right)S\left(\overline{h_2} ,\eta_1 \gamma;q_2\right)e\left(\frac{m\gamma}{q_1q_2}\right)\nonumber\\
    =&\sumx_{\alpha_1\bmod{q_1}}e\left(\frac{\overline{\alpha_1h_1}}{q_1}\right)\sumx_{\alpha_2\bmod{q_2}}e\left(\frac{\overline{\alpha_2h_2}}{q_2}\right)\sum_{\gamma\bmod{q_1q_2}}e\left(\frac{(\eta_1(\alpha_1q_2+\alpha_2q_1)+m)\gamma}{q_1q_2}\right)\nonumber\\
    \ll&q_1q_2\sumx_{\alpha_1\bmod{q_1}}\sumx_{\alpha_2\bmod{q_2}}\delta\left(-\eta_1 m\equiv \alpha_1q_2+\alpha_2q_1\bmod{q_1q_2}\right).
\end{align}
Hence (\ref{S10}) gives us the contribution of $m\neq0$ to $S_{1,0}(N)$ is bounded by \begin{align*}
    \ll& \frac{QMN_1^{4/3}N^{1/3}}{Ct^{1-\varepsilon}}\delta\left(N_1\ll \frac{C^3t^\varepsilon}{N}\right)\sup_{\|\beta\|_2=1}\sum_{m_0}\mathop{\sum\sum}_{\frac{C}{m_0}\leq q_1,q_2< (1+10^{-10})\frac{C}{m_0}}\mathop{\sum\sum}_{\substack{H'\leq h_1, h_2<(1+10^{-10})H'\\(h_1,q_1m_0)=1\\(h_2,q_2m_0)=1}}\\
    &\times \int_{|u_1|\ll1}\int_{|u_2|\ll1}|\beta(m_0q_1,h_1,u_1)\beta(m_0q_2,h_2,u_2)|du_2du_1\nonumber\\
    &\times \sum_{|m|\ll\frac{C^2t^\varepsilon}{N_1}}\sumx_{\alpha_1\bmod{q_1}}\sumx_{\alpha_2\bmod{q_1}}\delta\left(-\eta_1 m\equiv \alpha_1q_2+\alpha_2q_1\bmod{q_1q_2}\right)+t^{-99}.
\end{align*}
Performing the AM-GM inequality as the $m=0$ case, we get the above is bounded by \begin{align*}
    \ll& \frac{QMN_1^{4/3}N^{1/3}}{Ct^{1-\varepsilon}}\delta\left(N_1\ll\frac{C^3t^\varepsilon}{N}\right)\sup_{\|\beta\|_2=1}(|\mathfrak{s}_{1,1}|+|\mathfrak{s}_{1,2}|),
\end{align*}
where for brevity of notation, we have temporarily defined
\begin{align*}
    \mathfrak{s}_{1,\nu}:=&\sum_{m_0}\mathop{\sum\sum}_{\frac{C}{m_0}\leq q_1,q_2< (1+10^{-10})\frac{C}{m_0}}\mathop{\sum\sum}_{H'\leq h_1, h_2<(1+10^{-10})H'}\sum_{|m|\ll\frac{C^2t^\varepsilon}{N_1}}\nonumber\\
    &\times \sumx_{\alpha_1\bmod{q_1}}\sumx_{\alpha_2\bmod{q_2}}\delta\left(-\eta_1 m\equiv \alpha_1q_2+\alpha_2q_1\bmod{q_1q_2}\right)\int_{|u_\nu|\ll1}|\beta(m_0q_\nu,h_\nu,u_\nu)|^2du_\nu.
\end{align*}
To analyse the congruence condition, we let $q_0=(q_1,q_2)$, $q_1=q_0q_{1,0}q_1'$, $q_2=q_0q_{2,0}q_2'$ such that $q_{1,0},q_{2,0}|q_0^\infty$ and $(q_0,q_1')=(q_0,q_2')=1$. Then the congruence condition decompose into $q_0|m$ and \begin{align}\label{CongruenceAnalysis}
    -\eta_1\frac{m}{q_0}\equiv& \alpha_1q_{2,0}q_2'+\alpha_2q_{1,0}q_1'\bmod{q_0q_{1,0}q_{2,0}}\nonumber\\
    \alpha_1\equiv& -\eta_1\frac{m\overline{q_{2,0}q_2'}}{q_0}\bmod{q_1'}\nonumber\\
    \alpha_2\equiv& -\eta_1\frac{m\overline{q_{1,0}q_1'}}{q_0}\bmod{q_2'}.
\end{align}
This leads to \begin{align}\label{DetailedCountingComputation}
    \mathfrak{s}_{1,\nu}\ll &\mathop{\sum_{m_0}\sum_{q_0}\sum_{q_{1,0},q_{2,0}|q_0^\infty}\sum_{\substack{(q_1',q_2')=1\\(q_1'q_2',q_0)=1}}}_{m_0q_0q_{1,0}q_1',m_0q_0q_{2,0}q_2'\sim C}\mathop{\sum\sum}_{h_1, h_2\sim H'}\sum_{|m'|\ll\frac{C^2t^\varepsilon}{q_0N_1}}\mathop{\sumx_{\alpha_1\bmod{q_0q_{1,0}}}\sumx_{\alpha_2\bmod{q_0q_{2,0}}}}_{-\eta_1m'\equiv\alpha_1q_{2,0}q_2'+\alpha_2q_{1,0}q_1'\bmod{q_0q_{1,0}q_{2,0}}}\nonumber\\
    &\times \int_{|u_\nu|\ll1}|\beta(m_0q_0q_{\nu,0}q_\nu',h_\nu,u_\nu)|^2du_\nu\nonumber\\
    \ll & \frac{C^2T^\varepsilon}{N_1}H'\mathop{\sum_{m_0}\sum_{q_0}\sum_{q_{1,0},q_{2,0}|q_0^\infty}\sum_{\substack{(q_1',q_2')=1\\(q_1'q_2',q_0)=1}}}_{m_0q_0q_{1,0}q_1',m_0q_0q_{2,0}q_2'\sim C}\mathop{\sum}_{h_\nu\sim H'}\int_{|u_\nu|\ll1}|\beta(m_0q_0q_{\nu,0}q_\nu',h_\nu,u_\nu)|^2du_\nu.
\end{align}
Here we used for any $m_1', q_0, q_{1,0}, q_{2,0}, q_1', q_2'$, \begin{align*}
    \mathop{\sumx_{\alpha_1\bmod{q_0q_{1,0}}}\sumx_{\alpha_2\bmod{q_0q_{2,0}}}}_{-\eta_1m'\equiv\alpha_1q_{2,0}q_2'+\alpha_2q_{1,0}q_1'\bmod{q_0q_{1,0}q_{2,0}}}\ll q_0, 
\end{align*}
which holds as the congruence condition determines $\alpha_1$ mod $q_{1,0}$, and $\alpha_2$ is completely determined once $\alpha_1$ is chosen. Since $\|\beta\|_2=1$ and $H'\sim \frac{Ct}{N}$, we have 
\begin{align*}
    \mathfrak{s}_{1,\nu}\ll \frac{C^4t^{1+\varepsilon}}{N_1N}
\end{align*}
for $\nu=1,2$ and hence the contribution of $m\neq0$ to $S_{1,0}(N)$ is bounded by \begin{align}\label{S10non0Bound}
    \ll \frac{QMN_1^{4/3}N^{1/3}}{Ct^{1-\varepsilon}}\delta\left(N_1\ll\frac{C^3t^\varepsilon}{N}\right)\frac{C^4t}{N_1N}\ll \frac{C^3QMN_1^{1/3}t^\varepsilon}{N^{2/3}}\delta\left(N_1\ll\frac{C^3t^\varepsilon}{N}\right).
\end{align}

We remark that the point counting computation done in \eqref{CongruenceAnalysis} and \eqref{DetailedCountingComputation} will be repeated several times in a similar fashion throughout the paper. Since such computations are uninspiring and tedious, we will simply refer to these lines whenever we apply a similar point counting bound.

Combining (\ref{S100Bound}) and (\ref{S10non0Bound}), we get \begin{align}\label{S10Bound}
    S_{1,0}(N)\ll \frac{C^3QMN_1^{1/3}t^\varepsilon}{N^{2/3}}\delta\left(N_1\ll\frac{C^3t^\varepsilon}{N}\right).
\end{align}

\section{Opening the square and Poisson summation to \texorpdfstring{$S_{2}(N)$}{S2(N)}}

In this section, we apply similar treatment to $S_{2}(N)$ as done in the previous section. Recall that in (\ref{S2AfterDuality}) we have \begin{align*}
    S_{2}(N)=&\frac{MN_2^{1/3}N^{7/3}}{Qt^{1-\varepsilon}}\sup_{\|\beta\|_2=1}\sum_{n_0}\frac{1}{n_0}\sum_n\varphi\left(\frac{n_0^2n}{N_2}\right)\left|\sum_{\frac{C}{n_0}\leq q< (1+10^{-10})\frac{C}{n_0}}\frac{1}{q^{3/2}}\sum_{\substack{H'\leq h<(1+10^{-10})H'\\(h,qn_0)=1}}S\left(\overline{h},\eta_2 n;q\right)\right.\nonumber\\
    &\times \left.\int_\BR U(u)\beta(n_0q,h,u)W_{\eta_2}^\pm(n_0^2n,n_0q,h,u)du\right|^2.
\end{align*}
Opening the square, we get \begin{align}\label{S2AfterOpenSquare}
    S_{2}(N)=&\frac{MN_2^{1/3}N^{7/3}}{Qt^{1-\varepsilon}}\sup_{\|\beta\|_2=1}\sum_{n_0}\frac{1}{n_0}\sum_n\varphi\left(\frac{n_0^2n}{N_2}\right)\mathop{\sum\sum}_{\frac{C}{n_0}\leq q_1,q_2< (1+10^{-10})\frac{C}{n_0}}\frac{1}{(q_1q_2)^{3/2}}\mathop{\sum\sum}_{\substack{H'\leq h_1, h_2<(1+10^{-10})H'\\(h_1,q_1n_0)=1\\(h_2,q_2n_0)=1}}\nonumber\\
    &\times S\left(\overline{h_1} ,\eta_2 n;q_1\right)S\left(\overline{h_2} ,\eta_2 n;q_2\right)\int_\BR\int_\BR U(u_1)U(u_2)\beta(n_0q_1,h_1,u_1)\overline{\beta(n_0q_2,h_2,u_2)}\nonumber\\
    &\times W_{\eta_2}^\pm \left(n_0^2n,n_0q_1,h_1,u_1\right)\overline{W_{\eta_2}^\pm \left(n_0^2n,n_0q_2,h_2,u_2\right)}du_2du_1.
\end{align}
Applying Poisson summation to the $n$-sum, we have 
\begin{align*}
    &\sum_n \varphi\left(\frac{n_0^2n}{N_2}\right)S\left(\overline{h_1} ,\eta_2 n;q_1\right)S\left(\overline{h_2} ,\eta_2 n;q_2\right) W_{\eta_2}^\pm \left(n_0^2n,n_0q_1,h_1,u_1\right)\overline{W_{\eta_2}^\pm \left(n_0^2n,n_0q_2,h_2,u_2\right)}\nonumber\\
    =&\frac{N_2}{n_0^2q_1q_2}\sum_n\mathcal{C}_{\eta_2}(n,h_1,h_2,q_1q_2)\mathcal{J}_2(q_1,q_2,\vec{v}),
\end{align*}
where 
\begin{align*}
    \mathcal{C}_{\eta_2}(n,h_1,h_2,q_1q_2)=\sum_{\gamma\bmod q_1q_2}S\left(\overline{h_1} ,\eta_2 \gamma;q_1\right)S\left(\overline{h_2} ,\eta_2 \gamma;q_2\right)e\left(\frac{n\gamma}{q_1q_2}\right)
\end{align*}
is the same character sum defined in (\ref{CharDef}) and 
\begin{align*}
    &\mathcal{J}_2(q_1,q_2,\vec{v})\\
    =&\int_0^\infty \varphi(z)W_{\eta_2}^\pm \left(N_2z,n_0q_1,h_1,u_1\right)\overline{W_{\eta_2}^\pm \left(N_2z,n_0q_2,h_2,u_2\right)}e\left(-\frac{nN_2z}{n_0^2q_1q_2}\right)dz
\end{align*}
with $\vec{v}=(n,n_0,h_1,h_2,u_1,u_2)$ as before. Repeated integration by parts in the $z$-integral gives us arbitrary savings unless \begin{align*}
    |n|\ll\frac{CN^{1/3}t^\varepsilon}{N_2^{2/3}}.
\end{align*}
Inserting this back into (\ref{S2AfterOpenSquare}), we get \begin{align}\label{S2After1stCauchy}
    S_{2}(N)=&\frac{MN_2^{4/3}N^{7/3}}{Qt^{1-\varepsilon}}\sup_{\|\beta\|_2=1}\sum_{n_0}n_0^{-3}\sum_{|n|\ll\frac{CN^{1/3}t^\varepsilon}{N_2^{2/3}}}\mathop{\sum\sum}_{\frac{C}{n_0}\leq q_1,q_2< (1+10^{-10})\frac{C}{n_0}}(q_1q_2)^{-5/2}\nonumber\\
    &\times \mathop{\sum\sum}_{\substack{H'\leq h_1, h_2<(1+10^{-10})H'\\(h_1,q_1n_0)=1\\(h_2,q_2n_0)=1}}\mathcal{C}_{\eta_2}(n,h_1,h_2,q_1q_2)\int_\BR\int_\BR U(u_1)U(u_2)\beta(n_0q_1,h_1,u_1)\overline{\beta(n_0q_2,h_2,u_2)} \nonumber\\
    &\times \mathcal{J}_2(q_1,q_2,\vec{v}\,)du_2du_1+O\left(t^{-99}\right).
\end{align}

\section{Analysis of \texorpdfstring{$S_{1,X}(N)$}{S1X(N)} and \texorpdfstring{$S_{2}(N)$}{S2(N)}}

Recall that $N_2\ll \frac{N^2t^\varepsilon}{Q^3}$. Using $N_1\sim\frac{N^2X^3}{Q^3}$, splitting the $n$-sum into $n\geq0$ and $n<0$, and collecting the above analysis (\ref{W1Def}),  (\ref{S1XAfterIntegral}), (\ref{S10Bound}) and (\ref{S2After1stCauchy}), we have the following Lemma.

\begin{Lemma}\label{LemmaAfterFirstStep}
    For $0\leq |X|\ll t^\epsilon$, we have \begin{align*}
        S_{1,X}(N)\ll&\frac{C^2QMN_1^{4/3}N^{1/3}}{t^{1-\varepsilon}}\delta\left(\eta_1 X\geq 0, N_1\sim\frac{N^2X^3}{Q^3}\right)\\
        &\times \left(\frac{Ct}{N_1N}\delta\left(N_1\ll\frac{C^3t^\varepsilon}{N}\right)+\sup_\pm|\mathcal{R}_1^\pm(N)|\right)+O\left(t^{-99}\right)
    \end{align*}
    and \begin{align*}
        S_2(N)\ll\frac{MN_2^{4/3}N^{7/3}}{Qt^{1-\varepsilon}}\sup_\pm|\mathcal{R}_2^\pm(N)|+O\left(t^{-99}\right),
    \end{align*}
    where for $j=1,2$, \begin{align*}
        \mathcal{R}_j(N)^\pm=&\sup_{\|\beta\|_2=1}\sum_{n_0}n_0^{-3}\sum_{0\leq n\ll\frac{CN^{1/3}t^\varepsilon}{N_j^{2/3}}}\mathop{\sum\sum}_{\frac{C}{n_0}\leq q_1,q_2< (1+10^{-10})\frac{C}{n_0}}(q_1q_2)^{-5/2}\mathop{\sum\sum}_{\substack{H'\leq h_1, h_2<(1+10^{-10})H'\\(h_1,q_1n_0)=1\\(h_2,q_2n_0)=1}}\mathcal{C}_{\eta_j}(\pm n,h_1,h_2,q_1q_2)\nonumber\\
        &\times \int_\BR\int_\BR U(u_1)U(u_2)\beta(n_0q_1,h_1,u_1)\overline{\beta(n_0q_2,h_2,u_2)} \mathcal{J}_j(q_1,q_2,\vec{v}\,)du_2du_1,
    \end{align*}
    and \begin{align*}
        &\mathcal{J}_j(q_1,q_2,\vec{v}\,)\\
        =&\int_0^\infty \varphi(z)\Phi_{j,\eta_j}\left(n_0q_1,n_0q_2,u_1,u_2,z\right)(g(n_0q_1,Xx_{1,0})g(n_0q_2,Xx_{2,0}))^{2-j}\nonumber\\
        &\times e\left(-\eta_j'3\left(\frac{N_jNz}{(n_0q_1)^3}\left(\frac{n_0q_1t}{2\pi h_1N}+\frac{Nu_1}{Ht^{1-\varepsilon}}\right)\right)^\frac{1}{3}+\eta_j'3\left(\frac{N_jNz}{(n_0q_2)^3}\left(\frac{n_0q_2t}{2\pi h_2N}+\frac{Nu_2}{Ht^{1-\varepsilon}}\right)\right)^\frac{1}{3}\mp\frac{nN_jz}{n_0^2q_1q_2}\right)dz,
    \end{align*}
    with $\eta_1'=1, \eta_2'=\eta_2$, $x_{1,0},x_{2,0}$ are defined in (\ref{xj0Def}) and $\Phi_{j,\eta_j}$ is a $t^\varepsilon$-inert function. Moreover, $x_{1,0},x_{2,0}$ are flat with respect to $q_1,q_2,z$.
\end{Lemma}

With the above lemma, we are left to bound $\mathcal{R}_j^\pm(N)$. Since the treatment for $\pm=+$ and $\pm=-$ is the same, we will WLOG assume $\pm=+$ and denote $\mathcal{R}_j(N)=\mathcal{R}_j^+(N)$ henceforth. We separate the analysis into 4 cases depending on the size of $n$.

\subsection{Diagonal Contribution}

We first deal with the case $n=0$. Its contribution to $\mathcal{R}_j(N)$ is equal to \begin{align*}
    &\sup_{\|\beta\|_2=1}\sum_{n_0}n_0^{-3}\mathop{\sum\sum}_{\frac{C}{n_0}\leq q_1,q_2< (1+10^{-10})\frac{C}{n_0}}(q_1q_2)^{-5/2}\mathop{\sum\sum}_{\substack{H'\leq h_1, h_2<(1+10^{-10})H'\\(h_1,q_1n_0)=1\\(h_2,q_2n_0)=1}}\mathcal{C}_{\eta_j}(0,h_1,h_2,q_1q_2)\nonumber\\
    &\times \int_\BR\int_\BR U(u_1)U(u_2)\beta(n_0q_1,h_1,u_1)\overline{\beta(n_0q_2,h_2,u_2)}\mathcal{J}_j(q_1,q_2,\vec{v}_0)du_2du_1,
\end{align*}
where $\vec{v}_0=(0,n_0,h_1,h_2,u_1,u_2)$. Recall that (\ref{CharSum0}) gives us \begin{align*}
    \mathcal{C}_{\eta_j}(0,h_1,h_2,q_1q_2)=q_1^2\sum_{q_1'|q_1}q_1'\mu\left(\frac{q}{q_1'}\right)\delta\left(h_1\equiv h_2\bmod q_1', q_1=q_2\right).
\end{align*}
Using (\ref{gDerProp}) and $x_{j,0}$ is flat with respect to $z$, repeated integration by parts on the $z$-integral in $\mathcal{J}_j$ with $q_1=q_2=q$ gives us arbitrary savings unless \begin{align*}
    &\left|\left(\frac{N_jN}{(n_0q)^3}\left(\frac{n_0qt}{2\pi h_1N}+\frac{Nu_1}{Ht^{1-\varepsilon}}\right)\right)^\frac{1}{3}-\left(\frac{N_jN}{(n_0q)^3}\left(\frac{n_0qt}{2\pi h_2N}+\frac{Nu_2}{Ht^{1-\varepsilon}}\right)\right)^\frac{1}{3}\right|\ll t^\varepsilon\\
    \Longrightarrow & \left|\frac{n_0qt}{2\pi N}(h_1^{-1}-h_2^{-1})+\frac{N}{Ht^{1-\varepsilon}}(u_1-u_2)\right|\ll\frac{Ct^\varepsilon}{(N_jN)^{1/3}}\\
    \Longrightarrow & 0\neq |h_2-h_1|\ll\frac{C^2t^{1+\varepsilon}}{N_j^{1/3}N^{4/3}}+\frac{Ct^\varepsilon}{H} \text{ or } \left( h_1=h_2 \text{ and } |u_2-u_1|\ll\frac{CHt^{1+\varepsilon}}{N_j^{1/3}N^{4/3}} \right).
\end{align*}
Note that (\ref{gqxBound}) gives us $\mathcal{J}_j(q_1,q_2,\vec{v}_0)\ll t^\varepsilon$ for both $j=1,2$.

Applying the AM-GM inequality \begin{align*}
    |\beta(n_0q_1,h_1,u_1)\beta(n_0q_2,h_2,u_2)|\ll |\beta(n_0q_1,h_1,u_1)|^2+|\beta(n_0q_2,h_2,u_2)|^2
\end{align*} gives us the contribution of $n=0$ to $\mathcal{R}_j(N)$ is bounded by \begin{align*}
    \ll & \frac{t^\varepsilon}{C^3}\sup_{\|\beta\|_2=1}\left(|\mathfrak{s}_{4,1}|+|\mathfrak{s}_{4,2}|\right),
\end{align*}    
where for brevity of notation, we have temporarily defined
\begin{align*}
    \mathfrak{s}_{4,\nu}:=&\sum_{n_0}\sum_{\frac{C}{n_0}\leq q< (1+10^{-10})\frac{C}{n_0}}\sum_{q'|q}q'\mathop{\sum\sum}_{\substack{H'\leq h_1, h_2<(1+10^{-10})H'\\h_1\equiv h_2\bmod q'\\|h_2-h_1|\ll\frac{C^2t^{1+\varepsilon}}{N_j^{1/3}N^{4/3}}+\frac{Ct^\varepsilon}{H}}}\int_{|u_1|\ll1}\int_{|u_2|\ll1}|\beta(n_0q,h_\nu,u_\nu)|^2\nonumber\\
    &\times \left(\delta(h_1\neq h_2)+\delta\left(h_1=h_2,|u_2-u_1|\ll\frac{CHt^{1+\varepsilon}}{N_j^{1/3}N^{4/3}}\right)\right)du_2du_1
\end{align*}
for $\nu=1,2$. Applying a similar point counting argument as in \eqref{DetailedCountingComputation}, the contribution of $n=0$ to $\mathcal{R}_j(N)$ is bounded by
\begin{align}\label{DiagonalBound}
    \ll& \frac{t^\varepsilon}{C^2H}+\frac{Ht^{1+\varepsilon}}{CN_j^{1/3}N^{4/3}}.
\end{align}

\subsection{Contribution when \texorpdfstring{$n$}{n} is small}

Now we deal with the case $0\neq nN_j\ll C^2t^\varepsilon$. In such a case, $e\left(-\frac{nN_jz}{n_0^2q_1q_2}\right)$ is $t^\varepsilon$-inert. Using (\ref{gDerProp}) and $x_{j,0}$ is flat with respect to $z$, repeated integration by parts in the $z$-integral in $\mathcal{J}_j$ gives us arbitrary savings unless \begin{align*}
    &\left|\left(\frac{N_jN}{(n_0q_1)^3}\left(\frac{n_0q_1t}{2\pi h_1N}+\frac{Nu_1}{Ht^{1-\varepsilon}}\right)\right)^\frac{1}{3}-\left(\frac{N_jN}{(n_0q_2)^3}\left(\frac{n_0q_2t}{2\pi h_2N}+\frac{Nu_2}{Ht^{1-\varepsilon}}\right)\right)^\frac{1}{3}\right|\ll t^\varepsilon\\
    \Longrightarrow & \left|\frac{t}{2\pi n_0^2N}\left(\frac{1}{q_1^2h_1}-\frac{1}{q_2^2h_2}\right)+\frac{N}{n_0^3Ht^{1-\varepsilon}}\left(\frac{u_1}{q_1^3}-\frac{u_2}{q_2^3}\right)\right|\ll\frac{t^\varepsilon}{C^2(N_jN)^{1/3}}\\
    \Longrightarrow & \left|q_2-\sqrt{\frac{h_1}{h_2}}q_1\right|\ll\frac{C^2t^\varepsilon}{n_0(N_jN)^{1/3}}+\frac{CN}{n_0Ht^{1-\varepsilon}}.
\end{align*}
Together with (\ref{CharSumNon0}) giving us \begin{align*}
    \mathcal{C}_{\eta_j}(n,h_1,h_2,q_1,q_2)\ll q_1q_2\sumx_{\alpha_1\bmod{q_1}}\sumx_{\alpha_2\bmod{q_2}}\delta\left(n\equiv \alpha_1q_2+\alpha_2q_1\bmod{q_1q_2}\right),
\end{align*}
the contribution of $0\neq nN_j\ll C^2t^\varepsilon$ to $\mathcal{R}_j(N)$ is bounded by \begin{align*}
    \ll&t^\varepsilon\sup_{\|\beta\|_2=1}\sum_{n_0}n_0^{-3}\sum_{0\neq n\ll \frac{C^2t^\varepsilon}{N_j}}\mathop{\sum\sum}_{H'\leq h_1, h_2<(1+10^{-10})H'}\mathop{\sum\sum}_{\substack{\frac{C}{n_0}\leq q_1,q_2< (1+10^{-10})\frac{C}{n_0}\\\left|q_2-\sqrt{\frac{h_1}{h_2}}q_1\right|\ll\frac{C^2t^\varepsilon}{n_0(N_jN)^{1/3}}+\frac{CN}{n_0Ht^{1-\varepsilon}}}}(q_1q_2)^{-3/2}\nonumber\\
    &\times \sumx_{\alpha_1\bmod{q_1}}\sumx_{\alpha_2\bmod{q_2}}\int_{|u_1|\ll1}\int_{|u_2|\ll1} |\beta(n_0q_1,h_1,u_1)||\beta(n_0q_2,h_2,u_2)|du_2du_1\nonumber\\
    &\times \left(g(n_0q_1,Xx_{1,0})g(n_0q_2,Xx_{2,0})\right)^{2-j}\delta\left(n\equiv \alpha_1q_2+\alpha_2q_1\bmod{q_1q_2}\right).
\end{align*}
Applying the AM-GM inequality \begin{align*}
    |\beta(n_0q_1,h_1,u_1)\beta(n_0q_2,h_2,u_2)|\ll |\beta(n_0q_1,h_1,u_1)|^2+|\beta(n_0q_2,h_2,u_2)|^2
\end{align*}
with (\ref{gqxBound}) giving us $g(m_0q_j,Xx_{j,0})\ll t^\varepsilon$, the same point counting analysis as in $S_{1,0}(N)$ using (\ref{CongruenceAnalysis}) and \eqref{DetailedCountingComputation} implies that the contribution of $0\neq nN_j\ll C^2t^\varepsilon$ to $\mathcal{R}_j(N)$ is bounded by
\begin{align}\label{SmallnBound}
    \ll \frac{t^\varepsilon}{C^3}\frac{C^2}{N_j}\frac{Ct}{N}\left(1+\frac{C^2}{(N_jN)^{1/3}}+\frac{CN}{Ht}\right)\ll\frac{t^{1+\varepsilon}}{N_jN}\left(1+\frac{C^2}{(N_jN)^{1/3}}+\frac{CN}{Ht}\right).
\end{align}
Notice we have this contribution only when $N_j\ll C^2t^\varepsilon$.

\subsection{Integral Analysis}

Now we continue with the analysis when $nN_j\gg C^2t^\varepsilon$. Recalling the definition of $\mathcal{J}_j(q_1,q_2,\vec{v}\,)$ in Lemma \ref{LemmaAfterFirstStep}, applying stationary phase analysis with the calculations shown in appendix \ref{sect.JjIntegralAnalysis} yields
\begin{align}\label{Jjasymp}
    &\mathcal{J}_j(q_1,q_2,\vec{v}\,)\nonumber\\
    =&\frac{C}{\sqrt{nN_j}}(g(n_0q_1,Xx_1)g(n_0q_2,Xx_2))^{2-j} \Phi_2(z_0(n_0q_1,n_0q_2),q_1,q_2; \nu)e\left((z_0(n_0q_1,n_0q_2))\right)+O\left(t^{-A}\right),
\end{align}
for any $A>0$, with some flat function $\Phi_2$ supported on $[1/2,5/2]\times\BR^7$, and 
\begin{align*}
    x_j:=x_j(z_0(n_0q_1,n_0q_2))=\left(\frac{2\pi h_j}{n_0q_jt}\right)^{2/3}\frac{Q(N_1z_0(n_0q_1,n_0q_2))^{1/3}}{X}\left(1+\frac{2\pi h_jN^2u}{n_0q_jHt^{2-\varepsilon}}\right)^{-2/3},
\end{align*}
for $j=1,2$. Inserting this back into (\ref{S2After1stCauchy}) yields the contribution of $nN_j\gg C^2t^\varepsilon$ to $\mathcal{R}_j(N)$ is equal to \begin{align*}
    &\frac{1}{C^2\sqrt{N_j}}\sup_{\|\beta\|_2=1}\sum_{n_0}\sum_{\frac{C^2t^\varepsilon}{N_j}\ll n\ll \frac{CN^{1/3}t^\varepsilon}{N_j^{2/3}}}\frac{1}{\sqrt{n}}\mathop{\sum\sum}_{\frac{C}{n_0}\leq q_1,q_2< (1+10^{-10})\frac{C}{n_0}}\frac{1}{q_1q_2}\mathop{\sum\sum}_{\substack{H'\leq h_1, h_2<(1+10^{-10})H'\\(h_1,q_1n_0)=1\\(h_2,q_2n_0)=1}}\nonumber\\
    &\times \mathcal{C}_{\eta_j}(n,h_1,h_2,q_1q_2)\int_\BR\int_\BR U(u_1)U(u_2)\beta(n_0q_1,h_1,u_1)\overline{\beta(n_0q_2,h_2,u_2)}\nonumber\\
    &\times  \Phi(z_0(n_0q_1,n_0q_2),q_1,q_2;\nu)e\left(F_3(z_0(n_0q_1,n_0q_2))\right)du_2du_1+O\left(t^{-999}\right),
\end{align*}
with 
\begin{align*}
    \Phi(z_0(n_0q_1,n_0q_2),q_1,q_2;\nu):=&\frac{C^3}{(n_0^2q_1q_2)^{3/2}}\Phi_2(z_0(n_0q_1,n_0q_2)q_1,q_2; \nu)(g(n_0q_1,Xx_1)g(n_0q_2,Xx_2))^{2-j}.
\end{align*}
Using (\ref{gqxBound}) and (\ref{gDerProp}), $\Phi$ is a $t^\varepsilon$-inert function supported on $[1/2,5/2]\times\BR^7$ up to an arbitrarily small error.

\subsection{Contribution when \texorpdfstring{$n$}{n} is of intermediate size}

For a technical reason, we deal with the case \begin{align*}
    \frac{C^2t^\varepsilon}{N_j}\ll n\ll\frac{CN^{4/3}}{N_j^{2/3}Ht^{1-\varepsilon}}
\end{align*}
before proceeding further. Note that the restriction of $z_0(n_0q_1,n_0q_2)\sim1$ implies \begin{align*}
    &\left|n_0q_2\left(\frac{n_0q_1t}{2\pi h_1N}+\frac{Nu_1}{Ht^{1-\varepsilon}}\right)^{1/3}-n_0q_1\left(\frac{n_0q_2t}{2\pi h_2N}+\frac{Nu_2}{Ht^{1-\varepsilon}}\right)^{1/3}\right|\sim \frac{nN_j^{2/3}}{N^{1/3}}\\
    \Longrightarrow & \left|n_0q_2\left(\frac{n_0q_1t}{h_1N}\right)^{1/3}-n_0q_1\left(\frac{n_0q_2t}{h_2N}\right)^{1/3}\right|\ll\frac{nN_j^{2/3}}{N^{1/3}}+\frac{CN}{Ht^{1-\varepsilon}}\ll\frac{CN}{Ht^{1-\varepsilon}}\\
    \Longrightarrow& \left|q_2-\sqrt{\frac{h_1}{h_2}}q_1\right|\ll\frac{CN}{n_0Ht^{1-\varepsilon}}.
\end{align*}
Performing similar procedure as the case $0\neq nN_j\ll C^2t^\varepsilon$, the AM-GM inequality \begin{align*}
    |\beta(n_0q_1,h_1,u_1)\beta(n_0q_2,h_2,u_2)|\ll |\beta(n_0q_1,h_1,u_1)|^2+|\beta(n_0q_2,h_2,u_2)|^2
\end{align*} and (\ref{CharSumNon0}) gives us the contribution of $\frac{C^2t^\varepsilon}{N_j}\ll n\ll\frac{CN^{4/3}}{N_j^{2/3}Ht^{1-\varepsilon}}$ to $\mathcal{R}_j(N)$ is bounded by \begin{align*}
    \ll& \frac{t^\varepsilon}{C^2\sqrt{N_j}}\sup_{\|\beta\|_2=1}(|\mathfrak{s}_{3,1}|+|\mathfrak{s}_{3,2}|),
\end{align*}
where for brevity of notation, we have temporarily defined \begin{align*}
    \mathfrak{s}_{3,\nu}:=&\sum_{0\neq n\ll \frac{CN^{4/3}}{N_j^{2/3}Ht^{1-\varepsilon}}}\frac{1}{\sqrt{n}}\sum_{n_0}\mathop{\sum\sum}_{H'\leq h_1, h_2<(1+10^{-10})H'}\mathop{\sum\sum}_{\substack{\frac{C}{n_0}\leq q_1,q_2< (1+10^{-10})\frac{C}{n_0}\\\left|q_2-\sqrt{\frac{h_1}{h_2}}q_1\right|\ll\frac{CN}{n_0Ht^{1-\varepsilon}}}}\nonumber\\
    &\times \sumx_{\alpha_1\bmod{q_1}}\sumx_{\alpha_2\bmod{q_2}}\int_{|u_1|\ll1} |\beta(n_0q_\nu,h_\nu,u_\nu)|^2du_1\delta\left(n\equiv \alpha_1q_2+\alpha_2q_1\bmod{q_1q_2}\right)
\end{align*}
for $\nu=1,2$. Applying the same analysis as in $S_{1,0}(N)$ using (\ref{CongruenceAnalysis}), the contribution of $\frac{C^2t^\varepsilon}{N_j}\ll n\ll\frac{CN^{4/3}}{N_j^{2/3}Ht^{1-\varepsilon}}$ to $\mathcal{R}_j(N)$ is bounded by 
\begin{align}\label{MidnBound}
    \ll \frac{t^\varepsilon}{C^2\sqrt{N_j}}\frac{\sqrt{C}N^{2/3}}{N_j^{1/3}\sqrt{Ht}}\frac{Ct}{N}\left(1+\frac{CN}{Ht}\right)=\frac{t^{1/2+\varepsilon}}{\sqrt{CH}N_j^{5/6}N^{1/3}}\left(1+\frac{CN}{Ht}\right).
\end{align}
Notice we have this contribution only when $N_j\ll\frac{C^{3/2}N^2}{(Ht)^{3/2}}t^\varepsilon$.

\subsection{Contribution when \texorpdfstring{$n$}{n} is big}

Denote \begin{align}\label{NSNBDef}
    N_S:=\left(\frac{C^2}{N_j}+\frac{CN^{4/3}}{N_j^{2/3}Ht}\right)t^\varepsilon \quad \text{ and } \quad N_B:=\frac{CN^{1/3}t^\varepsilon}{N_j^{2/3}}.
\end{align}
Denote by $\mathcal{R}_j^*(N)$ the contribution of $n\gg N_S$ to $\mathcal{R}_j(N)$ in Lemma \ref{LemmaAfterFirstStep}, i.e. \begin{align*}
    \mathcal{R}_j^*(N)=&\sup_{\|\beta\|_2=1}\sum_{n_0}n_0^{-3}\sum_{N_S\ll n\ll N_B}\mathop{\sum\sum}_{\frac{C}{n_0}\leq q_1,q_2< (1+10^{-10})\frac{C}{n_0}}(q_1q_2)^{-5/2}\nonumber\\
    &\times \mathop{\sum\sum}_{\substack{H'\leq h_1, h_2<(1+10^{-10})H'\\(h_1,q_1n_0)=1\\(h_2,q_2n_0)=1}}\mathcal{C}_{\eta_j}(n,h_1,h_2,q_1q_2)\nonumber\\
    &\times \int_\BR\int_\BR U(u_1)U(u_2)\beta(n_0q_1,h_1,u_1)\overline{\beta(n_0q_2,h_2,u_2)} \mathcal{J}_j(q_1,q_2,\vec{v}\,)du_2du_1.
\end{align*}

\subsubsection{Character Sum Analysis}

Recall that (\ref{CharSumNon0}) gives us \begin{align*}
    \mathcal{C}_{\eta_j}(n,h_1,h_2,q_1q_2)=q_1q_2\sumx_{\alpha_1\bmod{q_1}}e\left(\frac{\overline{\alpha_1h_1}}{q_1}\right)\sumx_{\alpha_2\bmod{q_2}}e\left(\frac{\overline{\alpha_2h_2}}{q_2}\right)\delta\left(-\eta_j n\equiv \alpha_1q_2+\alpha_2q_1\bmod{q_1q_2}\right).
\end{align*}
We apply the same analysis as in (\ref{CongruenceAnalysis}): Write $(q_1,q_2)=q_0$, let $q_{1,0},q_1',q_{2,0},q_2'$ be unique positive integers such that $q_{1,0},q_{2,0}|q_0^\infty$, $(q_1'q_2',q_0)=1$ with $q_1=q_0q_{1,0}q_1'$ and $q_2=q_0q_{2,0}q_2'$. Then the congruence condition above breaks down into $q_0|n$ and \begin{align*}
    -\eta_j\frac{n}{q_0}\equiv& \alpha_1q_{2,0}q_2'+\alpha_2q_{1,0}q_1'\bmod{q_0q_{1,0}q_{2,0}}\\
    \alpha_1\equiv& -\eta_j\frac{n}{q_0}\overline{q_{2,0}q_2'}\bmod{q_1'}\\
    \alpha_2\equiv& -\eta_j\frac{n}{q_0}\overline{q_{1,0}q_1'}\bmod{q_2'}.
\end{align*}
Then the Chinese Remainder Theorem  yields \begin{align*}
    \mathcal{C}_{\eta_j}(n,h_1,h_2,q_1q_2)=&q_1q_2e\left(-\eta_j\frac{\overline{nh_1q_{1,0}}q_{2,0}q_2'}{q_1'}-\eta_j\frac{\overline{nh_2q_{2,0}}q_{1,0}q_1'}{q_2'}\right)\\
    &\times \sumx_{\beta_1\bmod{q_0q_{1,0}}}\sumx_{\beta_2\bmod{q_0q_{2,0}}}e\left(\frac{\beta_1\overline{h_1q_1'}}{q_0q_{1,0}}+\frac{\beta_2\overline{h_2q_2'}}{q_0q_{2,0}}\right)\\
    &\times \delta\left(q_0|n, -\eta_j\frac{n}{q_0}\equiv \overline{\beta_1}q_{2,0}q_2'+\overline{\beta_2}q_{1,0}q_1'\bmod{q_0q_{1,0}q_{2,0}}\right).
    \end{align*}
Applying reciprocity together with a change of variable,
\begin{align*}
    \mathcal{C}_{\eta_j}(n,h_1,h_2,q_1q_2)=&q_1q_2e\left(\eta_j\frac{\overline{q_1'}q_{2,0}q_2'}{nh_1q_{1,0}}+\eta_j\frac{\overline{q_2'}q_{1,0}q_1'}{nh_2q_{2,0}}-\eta_j\frac{q_{2,0}q_2'}{nh_1q_{1,0}q_1'}-\eta_j\frac{q_{1,0}q_1'}{nh_2q_{2,0}q_2'}\right)\\
    &\times \sumx_{\beta_1\bmod{q_0q_{1,0}}}\sumx_{\beta_2\bmod{q_0q_{2,0}}}e\left(\frac{\beta_1}{q_0q_{1,0}}+\frac{\beta_2}{q_0q_{2,0}}\right)\\
    &\times \delta\left(q_0|n, -\eta_j\frac{n}{q_0}\equiv \overline{\beta_1h_1q_1'}q_{2,0}q_2'+\overline{\beta_2h_2q_2'}q_{1,0}q_1'\bmod{q_0q_{1,0}q_{2,0}}\right).
\end{align*}

Before we proceed further, we first remove the coprimality condition $(q_1',q_2')=1$ by M\"{o}bius inversion, i.e. for any function $F$, \begin{align*}
    \sum_{(q_1',q_2')=1}F(q_1',q_2')=\sum_r\mu(r)\sum_{q_1',q_2'}F(q_1'r,q_2'r),
\end{align*}
this yields \begin{align*}
    \mathcal{R}_j^*(N)=&\sup_{\|\beta\|_2=1}\sum_{n_0}n_0^{-3}\sum_{N_S\ll n \ll N_B}\mathop{\sum\sum\sum\sum\sum\sum}_{\substack{q_1=q_0q_{1,0}q_1'r, q_2=q_0q_{2,0}q_2'r\\ q_{1,0}q_{2,0}|q_0^\infty, (q_1'q_2'r,q_0)=1 \\ \frac{C}{n_0}\leq q_1,q_2< (1+10^{-10})\frac{C}{n_0}}}\mu(r)(q_1q_2)^{-5/2}\nonumber\\
    &\times \mathop{\sum\sum}_{\substack{H'\leq h_1, h_2<(1+10^{-10})H'\\(h_1,q_1n_0)=1\\(h_2,q_2n_0)=1}}\mathcal{C}_{\eta_j}(n,h_1,h_2,q_1q_2)\nonumber\\
    &\times \int_\BR\int_\BR U(u_1)U(u_2)\beta(n_0q_1,h_1,u_1)\overline{\beta(n_0q_2,h_2,u_2)} \mathcal{J}_j(q_1,q_2,\vec{v}\,)du_2du_1.
\end{align*}

To prepare ourselves for a character sum analysis in later steps, we pull out $(n,q_0^\infty)$, $(h_2,(nh_1)^\infty)$ and the square-full part of $nh_1$. Write $n=n_1n_2n'$, $h_1=h_{1,0}h_1'$ and $h_2=h_{2,0}h_2'$ such that $n_1|q_0^\infty$, $(n_2h_{1,0}h_{2,0},n'h_1'h_2'q_0)=(n'h_1',h_2'q_0)=1$, $h_{2,0}|(n_2h_{1,0})^\infty$, $n'h_1'$ is square-free and $n_2h_{1,0}h_{2,0}$ is square full, i.e. for any prime $p|n_2h_{1,0}h_{2,0}$, we have $p^2|n_2h_{1,0}h_{2,0}$. With this notation, we then arrive at \begin{align*}
    \mathcal{R}_j^*(N)=&\sup_{\|\beta\|_2=1}\sum_{n_0}n_0^{-3}\sum_{q_0}\sum_{q_0|n_1|q_0^\infty}\mathop{\sum\sum\sum}_{\substack{n_2h_{1,0}h_{2,0} \square\text{-full}\\ h_{2,0}|(n_2h_{1,0})^\infty\\ (n_2,q_0)=1}}\mathop{\sum\sum\sum}_{\substack{N_S\ll n=n_1n_2n' \ll N_B\\ H'\leq h_1=h_{1,0}h_1', h_2=h_{2,0}h_2'<(1+10^{-10})H'\\ n'h_1' \square\text{-free}, (n_1'h_1',h_2')=1 \\(n'h_1'h_2',n_1h_{1,0}h_{2,0}q_0)=1}}\nonumber\\
    &\times \mathop{\sum\sum\sum\sum\sum}_{\substack{q_1=q_0q_{1,0}q_1'r, q_2=q_0q_{2,0}q_2'r\\ q_{1,0}q_{2,0}|q_0^\infty, (q_1'q_2'r,q_0)=1 \\ (q_1,h_1)=(q_2,h_2)=1\\ \frac{C}{n_0}\leq q_1,q_2< (1+10^{-10})\frac{C}{n_0}}}\mu(r)(q_1q_2)^{-5/2}\mathcal{C}_{\eta_j}(n,h_1,h_2,q_1q_2)\nonumber\\
    &\times \int_\BR\int_\BR U(u_1)U(u_2)\beta(n_0q_1,h_1,u_1)\overline{\beta(n_0q_2,h_2,u_2)} \mathcal{J}_j(q_1,q_2,\vec{v}\,)du_2du_1.
\end{align*}

\subsubsection{Second Cauchy-Schwarz Inequality and Poisson Summation}

Notice that by the same point counting argument as \eqref{DetailedCountingComputation}, we have
\begin{align*}
    &\sup_{\|\beta\|_2=1}\mathop{\sum\sum\sum}_{n_2h_{1,0}h_{2,0} \square\text{-full}}\frac{1}{\sqrt{n_2h_{2,0}}}\sum_{q_1\ll C}\mathop{\sum\sum\sum\sum}_{\substack{n_0q_0q_{1,0}r|q_1 \\q_{1,0}|q_0^\infty}}\sum_{\substack{q_{2,0}\ll C\\q_{2,0}|q_0^\infty}}\mathop{\sum\sum}_{\substack{0\neq n_1n'\ll N_B\\q_0|n_1|q_0^\infty}}\frac{1}{n'}\\
    &\times \mathop{\sum\sum}_{ h_1',h_2'\ll H'}\frac{1}{h_2'}\int_{|u_1|\ll1}|\beta(q_1,h_{1,0}h_1',u_1)|^2du_1\int_{|u_2|\ll1}du_2\ll t^\varepsilon,
\end{align*}
Let $\Upsilon\in C_c^\infty([1-10^{-9},1+10^{-9}])$ be a fixed function such that $\Upsilon(x)=1$ for $1\leq x<1+10^{-10}$. Now we apply Cauchy-Schwarz inequality to take out $n_0,n_1, n', q_0,q_{1,0},q_{2,0},q_1',r, h_{1,0}, h_1', h_{2,0}, h_2'$-sums and the $u_1,u_2$-integrals. Together with (\ref{Jjasymp}) and the character sum analysis above, we get 
\begin{align}\label{S2After2ndCauchy}
    \mathcal{R}_j^*(N)\ll &\frac{t^{\varepsilon}}{C^2\sqrt{N_j}} \sup_{\|\beta\|_2=1} \sup_{|u_1|\ll1}\Bigg\{\int_{|u_2|\ll1}\sum_{n_0}\sum_{q_0}\sum_{q_0|n_1|q_0^\infty}\sum_r\mathop{\sum\sum\sum}_{\substack{n_2h_{1,0}h_{2,0} \square\text{-full}\\ h_{2,0}|(n_2h_{1,0})^\infty\\ (n_2,q_0)=1}}\mathop{\sum\sum\sum}_{\substack{N_S\ll n=n_1n_2n' \ll N_B\\ H'\leq h_1=h_{1,0}h_1', h_2=h_{2,0}h_2'<(1+10^{-10})H'\\ n'h_1' \square\text{-free}, (n_1'h_1',h_2')=1 \\(n'h_1'h_2',n_1h_{1,0}h_{2,0}q_0)=1}} \nonumber\\
    &\times \mathop{\sum\sum}_{q_{1,0},q_{2,0}|q_0^\infty}\frac{\sqrt{h_{2,0}}h_2'}{n_1\sqrt{n_2}}\sum_{q_1'}\Upsilon\left(\frac{n_0q_0q_{1,0}q_1'r}{C}\right)\mathop{\sum\sum}_{\substack{C\leq q_2=n_0q_0q_{2,0}q_2'r<(1+10^{-10})C\\C\leq q_3=n_0q_0q_{2,0}q_3'r<(1+10^{-10})C}}\nonumber\\
    &\times \mathcal{D}_1(q_1')\overline{\beta(q_2,h_2,u_2)}\beta(q_3,h_2,u_2)\Psi(n_0q_0q_{1,0}q_1'r,q_2,q_3)\nonumber\\
    &\times e\left(F_3(z_0(n_0q_0q_{1,0}q_1'r,q_2))-F_3(z_0(n_0q_0q_{1,0}q_1'r,q_3))\right)du_2\Bigg\}^{1/2}+t^{-999},
\end{align}
where \begin{align}\label{D1Def}
    \mathcal{D}_1(q_1')=&e\left(\eta_j\frac{\overline{q_1'}q_{2,0}(q_2'-q_3')}{nh_1q_{1,0}}+\eta_j\frac{(\overline{q_2'}-\overline{q_3'})q_{1,0}q_1'}{nh_2q_{2,0}}\right)\nonumber\\
    &\times \sumx_{\beta_1\bmod{q_0q_{1,0}}}\sumx_{\beta_1'\bmod{q_0q_{1,0}}}\sumx_{\beta_2\bmod{q_0q_{2,0}}}\sumx_{\beta_2'\bmod{q_0q_{2,0}}}e\left(\frac{\beta_1-\beta_1'}{q_0q_{1,0}}+\frac{\beta_2-\beta_2'}{q_0q_{2,0}}\right)\nonumber\\
    &\times \delta\left(-\eta_j\frac{n}{q_0}\equiv \overline{\beta_1h_1q_1'}q_{2,0}q_2'+\overline{\beta_2h_2q_2'}q_{1,0}q_1'\bmod{q_0q_{1,0}q_{2,0}},\right.\nonumber\\
    &\left. -\eta_j\frac{n}{q_0}\equiv \overline{\beta_1'h_1q_1'}q_{2,0}q_3'+\overline{\beta_2'h_2q_3'}q_{1,0}q_1'\bmod{q_0q_{1,0}q_{2,0}}\right),
\end{align}
\begin{align}\label{PsiDef}
    \Psi(n_0q_0q_{1,0}q_1'r,q_2,q_3)=&\Phi(z_0(n_0q_0q_{1,0}q_1'r,q_2),n_0q_0q_{1,0}q_1'r,q_2;\nu)\overline{\Phi(z_0(n_0q_0q_{1,0}q_1'r,q_3),n_0q_0q_{1,0}q_1'r,q_3;\nu)}\nonumber\\
    &\times \frac{C^3}{(q_2q_3)^{3/2}}e\left(-\eta_j\frac{q_{2,0}(q_2'-q_3')}{nh_1q_{1,0}q_1'}-\eta_j\frac{q_{1,0}q_1'}{nh_2q_{2,0}q_2'}+\eta_j\frac{q_{1,0}q_1'}{nh_2q_{2,0}q_3'}\right),
\end{align}
with \begin{align*}
    z_0(x,q)=\frac{\sqrt{N}}{n^{3/2}N_j}\left(-\eta_j' q\left(\frac{xt}{2\pi h_1N}+\frac{Nu_1}{Ht^{1-\varepsilon}}\right)^{1/3}+\eta_j'x\left(\frac{qt}{2\pi h_2N}+\frac{Nu_2}{Ht^{1-\varepsilon}}\right)^{1/3}\right)^{3/2},
\end{align*}
and 
\begin{align*}
    F_3(z_0(x,q))=&\frac{2nN_j}{xq}z_0(x,q)
\end{align*}
as before. Here coprimalities between variables are assumed such that $D_1(q_1')$ is well-defined. In particular, the condition $(q_1',nh_1q_0)=(q_2'q_3',nh_2q_0)=1$ is assumed.

Consider the $q_1'$-sum. Notice that $\mathcal{D}_1(q_1')$ is determined by $q_1'$ mod $nh_1h_2q_0q_{1,0}q_{2,0}$, hence applying Poisson summation gives us \begin{align*}
    &\sum_{q_1'}\Upsilon\left(\frac{n_0q_0q_{1,0}q_1'r}{C}\right)\mathcal{D}_1(q_1')\varphi\left(\frac{n_0q_0q_{1,0}q_1'r}{C}\right)\Psi(n_0q_0q_{1,0}q_1'r,q_2,q_3)\nonumber\\
    &\times e\left(F_3(z_0(n_0q_0q_{1,0}q_1'r,q_2))-F_3(z_0(n_0q_0q_{1,0}q_1'r,q_3))\right)\nonumber\\
    =&\frac{C}{n_0nh_1h_2q_0^2q_{1,0}^2q_{2,0}r}\sum_{q_1'}\sum_{\gamma\bmod{nh_1h_2q_0q_{1,0}q_{2,0}}}\mathcal{D}_1(\gamma)e\left(\frac{q_1'\gamma}{nh_1h_2q_0q_{1,0}q_{2,0}}\right)I_2(q_1',q_2,q_3),
\end{align*}
with \begin{align}\label{I2Def}
    I_2(q_1',q_2,q_3)=\int_0^\infty\Upsilon(y)\Psi(Cy,q_2,q_3)e\left(F_3(z_0(Cy,q_2))-F_3(z_0(Cy,q_3))-\frac{q_1'Cy}{n_0nh_1h_2q_0^2q_{1,0}^2q_{2,0}r}\right)dy.
\end{align}
Differentiating gives us
\begin{multline}\label{z_0Diff}
    \frac{d}{dy}z_0(Cy,q)\\
    = z_0(Cy,q)^{1/3}\bigg[ \frac{-\eta_j'qCt}{4\pi nh_1N^\frac23 N_j^\frac23}\left(\frac{Cyt}{2\pi h_1N}+\frac{Nu_1}{Ht^{1-\varepsilon}}\right)^{-2/3}
    +\frac{3\eta_j'N^{1/3}C}{2nN_j^{2/3}}\left(\frac{qt}{2\pi h_2N}+\frac{Nu_2}{Ht^{1-\varepsilon}}\right)^{1/3}\bigg],
\end{multline}

In order to bound the higher derivatives, we temporarily introduce the following notation.

\begin{align*}
G & := \frac{N}{n^3N_j^2}q^3\frac{Ct}{2\pi h_1N} \sim \frac{N}{n^3N_j^2}q^3,\\
K & := \frac{N}{n^3N_j^2}q^3 \frac{Nu_1}{Ht^{1-\epsilon}} \ll Gt^{-\epsilon/2},\\
J & := \frac{CN^{1/3}}{nN_j^{2/3}}\bigg( \frac{qt}{2\pi h_2N} + \frac{Nu_2}{Ht^{1-\epsilon
}} \bigg)^{1/3} \sim G^{1/3}, 
\end{align*}
where we have used the conditions $H\gg\frac{N}{t^{1-\varepsilon}}$ and $q\sim C$. Then we may write
$$ z_0(Cy, q) = J^{3/2}\bigg(-\eta_j'\bigg(\frac{G}{J^3}y + \frac{K}{J^3}\bigg)^{\frac13} + \eta_j'y\bigg)^{\frac32}, $$
so that the coefficients of $y$ are of size $O(1)$ and one can see that $\frac{d^k}{dy^k}z^{2/3}\ll J$ for any $k\geq0$. By the Fa\`{a} di Bruno's formula, we have
\begin{align}\label{z_0DiffBound}
    \left|\frac{d^k}{dy^k}z_0\left(Cy,q\right)\right|, \left|\frac{d^k}{dy^k}\Psi(Cy,q_2,q_3)\right|\ll_k  \left(\frac{CN^{1/3}}{nN_j^{2/3}}\right)^k \text{ and } \left|\frac{d^k}{dy^k}F_3\left(z_0\left(Cy,q\right)\right)\right|\ll_k  \left(\frac{(N_jN)^{1/3}}{C}\right)^{k}
\end{align}
for any $k\geq0$ and $q\sim C$. Hence repeated integration by parts gives us arbitrary savings unless
\begin{align*}
    |q_1'|\ll& \frac{n_0nh_1h_2q_0^2q_{1,0}^2q_{2,0}r}{C}\left(1+\frac{CN^{1/3}}{nN_j^{2/3}}+\frac{(N_jN)^{1/3}}{C}\right)t^\varepsilon\ll Q_1Q_B,
\end{align*}
where 
\begin{align}\label{Q1QBDef}
    Q_1:=n_1n_2h_{1,0}h_{2,0}q_0q_{1,0}^2 \quad \text{ and }\quad Q_B:=\frac{n_0n'q_0q_{2,0}rN_j^{1/3}t^{2+\varepsilon}}{h_{1,0}h_{2,0}N^{5/3}}.
\end{align}
Here we used $\frac{C^2t^\varepsilon}{N_j}=N_S\ll n$ and $h_1, h_2\sim H'\sim\frac{Ct}{N}$. Inserting the above back into (\ref{S2After2ndCauchy}) yields \begin{align}\label{S2After2ndPoisson}
    \mathcal{R}_j^*(N)\ll &\frac{t^{\varepsilon}}{C^{3/2}\sqrt{N_j}} \sup_{\|\beta\|_2=1} \sup_{|u_1|\ll1}\Bigg\{\int_{|u_2|\ll1}\sum_{n_0}\sum_{q_0}\sum_{q_0|n_1|q_0^\infty}\sum_r\mathop{\sum\sum\sum}_{\substack{n_2h_{1,0}h_{2,0} \square\text{-full}\\ h_{2,0}|(n_2h_{1,0})^\infty\\ (n_2,q_0)=1}}\mathop{\sum\sum\sum}_{\substack{N_S\ll n=n_1n_2n' \ll N_B\\ H'\leq h_1=h_{1,0}h_1', h_2=h_{2,0}h_2'<(1+10^{-10})H'\\ n'h_1' \square\text{-free}, (n_1'h_1',h_2')=1 \\(n'h_1'h_2',n_1h_{1,0}h_{2,0}q_0)=1}} \nonumber\\
    &\times \mathop{\sum\sum}_{q_{1,0},q_{2,0}|q_0^\infty}\mathop{\sum\sum}_{\substack{C\leq q_2=n_0q_0q_{2,0}q_2'r<(1+10^{-10})C\\C\leq q_3=n_0q_0q_{2,0}q_3'r<(1+10^{-10})C}}\frac{\overline{\beta(q_2,h_2,u_2)}\beta(q_3,h_2,u_2)}{n_0n_1^2n_2^{3/2}n'h_1\sqrt{h_{2,0}}q_0^2q_{1,0}^2q_{2,0}r}\nonumber\\
    &\times \sum_{|q_1'|\ll Q_1Q_B}\sum_{\gamma\bmod{nh_1h_2q_0q_{1,0}q_{2,0}}}\mathcal{D}_1(\gamma)e\left(\frac{q_1'\gamma}{nh_1h_2q_0q_{1,0}q_{2,0}}\right)I_2(q_1',q_2,q_3)du_2\Bigg\}^{1/2}+t^{-999}.
\end{align}

We first deal with the character sum \begin{align*}
    \sum_{\gamma\bmod{nh_1h_2q_0q_{1,0}q_{2,0}}}\mathcal{D}_1(\gamma)e\left(\frac{q_1'\gamma}{nh_1h_2q_0q_{1,0}q_{2,0}}\right).
\end{align*}
Recalling the definition of $\mathcal{D}_1$ in (\ref{D1Def}), the Chinese Remainder Theorem yields \begin{align*}
    &\sum_{\gamma\bmod{nh_1h_2q_0q_{1,0}q_{2,0}}}\mathcal{D}_1(\gamma)e\left(\frac{q_1'\gamma}{nh_1h_2q_0q_{1,0}q_{2,0}}\right)\nonumber\\
    =&h_2'S(q_1'\overline{n_1n_2h_{1,0}h_2q_0q_{1,0}q_{2,0}}+\eta_j(\overline{q_2'}-\overline{q_3'})q_{1,0}\overline{n_1n_2h_2q_{2,0}},\eta_j (q_2'-q_3')q_{2,0}\overline{n_1n_2h_{1,0}q_{1,0}};n'h_1')\nonumber\\
    &\times \delta\left(q_1'\equiv -\eta_j q_0q_{1,0}^2h_1(\overline{q_2'}-\overline{q_3'})\bmod{h_2'}\right)\mathcal{D}_0(q_2',q_3'),
\end{align*}
where with $M_0=n_1n_2h_{1,0}h_{2,0}q_0q_{1,0}q_{2,0}$,
\begin{align}\label{D0Def}
    \mathcal{D}_0(q_2',q_3')=&\sum_{\gamma\bmod{M_0}}e\left(\frac{q_1'\gamma\overline{n'h_1'h_2'}}{M_0}+\eta_j\frac{q_{2,0}(q_2'-q_3')\overline{\gamma n'h_1'}}{n_1n_2h_{1,0}q_{1,0}}+\eta_j\frac{(\overline{q_2'}-\overline{q_3'})q_{1,0}\gamma\overline{n'h_2'}}{n_1n_2h_{2,0}q_{2,0}}\right)\nonumber\\
    &\times \sumx_{\beta_1\bmod{q_0q_{1,0}}}\sumx_{\beta_1'\bmod{q_0q_{1,0}}}\sumx_{\beta_2\bmod{q_0q_{2,0}}}\sumx_{\beta_2'\bmod{q_0q_{2,0}}}e\left(\frac{\beta_1-\beta_1'}{q_0q_{1,0}}+\frac{\beta_2-\beta_2'}{q_0q_{2,0}}\right)\nonumber\\
    &\times \delta\left(-\eta_j\frac{n}{q_0}\equiv \overline{\beta_1h_1\gamma}q_{2,0}q_2'+\overline{\beta_2h_2q_2'}q_{1,0}\gamma \bmod{q_0q_{1,0}q_{2,0}},\right.\nonumber\\
    &\left. -\eta_j\frac{n}{q_0}\equiv \overline{\beta_1'h_1\gamma}q_{2,0}q_3'+\overline{\beta_2'h_2q_3'}q_{1,0}\gamma \bmod{q_0q_{1,0}q_{2,0}}\right)\ll n_1n_2h_{1,0}h_{2,0}q_0^3q_{1,0}q_{2,0}.
\end{align}
Hence the Weil bound for Kloosterman sums gives us 
\begin{align}\label{D1Bound}
    &\sum_{\gamma\bmod{nh_1h_2q_0q_{1,0}q_{2,0}}}\mathcal{D}_1(\gamma)e\left(\frac{q_1'\gamma}{nh_1h_2q_0q_{1,0}q_{2,0}}\right)\nonumber\\
    \ll & n_1n_2h_{1,0}h_2q_0^3q_{1,0}q_{2,0}t^\varepsilon\sqrt{n'h_1'(q_1',q_2'-q_3',n'h_1')}\delta\left(q_1'\equiv -\eta_j q_0q_{1,0}^2h_1(\overline{q_2'}-\overline{q_3'})\bmod{h_2'}\right).
\end{align}

Next we analyse the $y$-integral with the phase function being \begin{align*}
    &F_3(z_0(Cy,q_2))-F_3(z_0(Cy,q_3))-\frac{q_1'Cy}{n_0nh_1h_2q_0^2q_{1,0}^2q_{2,0}r}\\
    =&\frac{2nN_j}{Cy}\left(\frac{z_0(Cy,q_2)}{q_2}-\frac{z_0(Cy,q_3)}{q_3}\right)-\frac{q_1'Cy}{n_0nh_1h_2q_0^2q_{1,0}^2q_{2,0}r}.
\end{align*}
Denote \begin{align}\label{GDef}
    G\left(y,z\right)=&G_{h_1,h_2}(y,z)=F_3(z_0(Cy,Cz))\nonumber\\
    =&\sqrt{\frac{2t}{\pi n}}\left(-\eta_j' \left(\frac{z}{h_1y}+\frac{2\pi N^2u_1z}{CHt^{2-\varepsilon}y^2}\right)^{1/3}+\eta_j'\left(\frac{y}{h_2z}+\frac{2\pi N^2u_2y}{CHt^{2-\varepsilon}z^2}\right)^{1/3}\right)^{3/2}.
\end{align}
We start with analysing the case when $|q_1'|\ll Q_1Q_S$ (including $q_1'=0$), where $Q_1$ is defined in \eqref{Q1QBDef} and 
\begin{align}\label{QsDef}
    Q_S:=\frac{n_0q_0q_{2,0}rC^3t^{2+\varepsilon}}{n_1n_2nh_{1,0}h_{2,0}(N_jN)^{4/3}}.
\end{align}

\begin{Remark}
    This threshold of $Q_1Q_S$ is chosen such that when $|q_1'|\gg Q_1Q_S$, the phase $\frac{q_1'Cy}{n_0nh_1h_2q_0^2q_{1,0}^2q_{2,0}r}$ has magnitude $\gg \left(\frac{CN^{1/3}}{nN_j^{2/3}}\right)^2t^\varepsilon$. Since the weight function $\Psi(Cy,q_2,q_3)$ consists of an inert function with $z_0$ as an argument, we can apply Lemma \ref{statlemma} (2) and (3) in the range $|q_1'|\gg Q_1Q_S$. This condition is ultimately used in Lemma \ref{InuBound}.
\end{Remark}

Applying stationary phase analysis with the calculations presented in appendix \ref{sect.yIntegralAnalysis} implies that the $y$-integral is bounded by \begin{align*}
    \ll \frac{nN_j^{2/3}t^\varepsilon}{CN^{1/3}}\delta\left(|q_2-q_3|\ll \frac{C^3t^\varepsilon}{nN_j}\right)+t^{-999}.
\end{align*}
Together with (\ref{D1Bound}), this gives us the contribution of $|q_1'|\ll Q_1Q_S$ to $S_{2}(N)^*$ in (\ref{S2After2ndPoisson}) is bounded by 

\begin{align*}
    \ll &\frac{t^{\varepsilon}}{C^{3/2}\sqrt{N_j}} \sup_{\|\beta\|_2=1} \sup_{|u_1|\ll1}\Bigg\{\int_{|u_2|\ll1}\sum_{n_0}\sum_{q_0}\sum_{q_0|n_1|q_0^\infty}\sum_r\mathop{\sum\sum\sum}_{\substack{n_2h_{1,0}h_{2,0} \square\text{-full}\\ h_{2,0}|(n_2h_{1,0})^\infty\\ (n_2,q_0)=1}}\mathop{\sum\sum\sum}_{\substack{N_S\ll n=n_1n_2n' \ll N_B\\ H'\leq h_1=h_{1,0}h_1', h_2=h_{2,0}h_2'<(1+10^{-10})H'\\ n'h_1' \square\text{-free}, (n_1'h_1',h_2')=1 \\(n'h_1'h_2',n_1h_{1,0}h_{2,0}q_0)=1}} \nonumber\\
    &\times \mathop{\sum\sum}_{q_{1,0},q_{2,0}|q_0^\infty}\mathop{\sum\sum}_{\substack{C\leq q_2=n_0q_0q_{2,0}q_2'r<(1+10^{-10})C\\C\leq q_3=n_0q_0q_{2,0}q_3'r<(1+10^{-10})C\\|q_2-q_3|\ll\frac{C^3t^\varepsilon}{nN_j}}}\sum_{\substack{|q_1'|\ll Q_1Q_S\\q_1'\equiv -\eta_j q_0h_1(\overline{q_2'}-\overline{q_3'})\bmod{h_2'}}}|\beta(q_2,h_2,u_2)\beta(q_3,h_2,u_2)|du_2\nonumber\\
    &\times \frac{N_j^{2/3}}{CN^{1/3}}\frac{\sqrt{n_2n'h_{2,0}}h_2'q_0}{n_0\sqrt{h_1'}q_{1,0}r}\sqrt{(q_1',q_2'-q_3',n'h_1')}\Bigg\}^{1/2}.
\end{align*}
Recall the definition of $N_S$ and $N_B$ in \eqref{NSNBDef}. Applying the AM-GM inequality \begin{align*}
    |\beta(q_2,h_2,u_2)\beta(q_3,h_2,u_2)|\ll |\beta(q_2,h_2,u_2)|^2+|\beta(q_3,h_2,u_2)|^2,
\end{align*}
the same point counting argument as \eqref{DetailedCountingComputation} implies that the contribution of $|q_1'|\ll Q_1Q_S$ to $S_{2}(N)^*$ in (\ref{S2After2ndPoisson}) is bounded by
\begin{align}\label{B1Def}
    \ll &\frac{t^{\varepsilon}}{C^{3/2}\sqrt{N_j}} \Bigg\{\underbrace{\frac{H'^2N_B^2N_j^{2/3}}{CN^{1/3}}}_{q_1'=0, q_2=q_3}+ \underbrace{\frac{C^2\sqrt{H'N_B}}{(N_jN)^{1/3}}}_{q_1'=0, q_2\neq q_3}+\underbrace{\frac{C^2\sqrt{H'N_B}t^2}{N_j^{2/3}N^{5/3}}}_{q_1'\neq0, q_2=q_3}+\underbrace{\frac{C^5\sqrt{H'}t^2}{\sqrt{N_S}(N_jN)^{5/3}}}_{q_1'\neq0, q_2\neq q_3}\Bigg\}^{1/2}\nonumber\\
    \ll & t^\varepsilon\left(\frac{t}{(N_jN)^{5/6}}+\frac{t^{1/4}}{N_j^{5/6}N^{1/3}}+ \frac{t^{5/4}}{N_jN}+ \frac{C^{3/4}t^{5/4}}{(N_jN)^{13/12}}\right) =: \mathcal{B}_1.
\end{align}
Inserting the above analysis back into (\ref{S2After2ndPoisson}) yields
\begin{align*}
    \mathcal{R}_j^*(N)\ll &\frac{t^{\varepsilon}}{C^{3/2}\sqrt{N_j}} \sup_{\|\beta\|_2=1} \sup_{|u_1|\ll1}\Bigg\{\int_{|u_2|\ll1}\sum_{n_0}\sum_{q_0}\sum_{q_0|n_1|q_0^\infty}\sum_r\mathop{\sum\sum\sum}_{\substack{n_2h_{1,0}h_{2,0} \square\text{-full}\\ h_{2,0}|(n_2h_{1,0})^\infty\\ (n_2,q_0)=1}}\mathop{\sum\sum\sum}_{\substack{N_S\ll n=n_1n_2n' \ll N_B\\ H'\leq h_1=h_{1,0}h_1', h_2=h_{2,0}h_2'<(1+10^{-10})H'\\ n'h_1' \square\text{-free}, (n_1'h_1',h_2')=1 \\(n'h_1'h_2',n_1h_{1,0}h_{2,0}q_0)=1}} \nonumber\\
    &\times \mathop{\sum\sum}_{q_{1,0},q_{2,0}|q_0^\infty}\mathop{\sum\sum}_{\substack{C\leq q_2=n_0q_0q_{2,0}q_2'r<(1+10^{-10})C\\C\leq q_3=n_0q_0q_{2,0}q_3'r<(1+10^{-10})C}}\frac{h_2'\overline{\beta(q_2,h_2,u_2)}\beta(q_3,h_2,u_2)}{n_0n_1^2n_2^{3/2}n'h_1\sqrt{h_{2,0}}q_0^2q_{1,0}^2q_{2,0}r}\nonumber\\
    &\times \sum_{Q_1Q_S\ll |q_1'|\ll Q_1Q_B}\mathcal{D}_2(q_2',q_3')I_2(q_1',q_2,q_3)du_2\Bigg\}^{1/2}+\mathcal{B}_1,
\end{align*}
with 
\begin{align}\label{D2Def}
    \mathcal{D}_2(q_2',q_3')=&S(q_1'\overline{n_1n_2h_{1,0}h_2q_0q_{1,0}q_{2,0}}+\eta_j(\overline{q_2'}-\overline{q_3'})q_{1,0}\overline{n_1n_2h_2q_{2,0}},\eta_j (q_2'-q_3')q_{2,0}\overline{n_1n_2h_{1,0}q_{1,0}};n'h_1')\nonumber\\
    &\times \delta\left(q_1'\equiv -\eta_j q_0q_{1,0}^2h_1(\overline{q_2'}-\overline{q_3'})\bmod{h_2'}\right)\mathcal{D}_0(q_2',q_3')
\end{align}
and $\mathcal{D}_0(q_2',q_3')$ as in (\ref{D0Def}).

\subsubsection{Third Cauchy Schwarz Inequality and Poisson Summation}

Now we consider the remaining case $|q_1'|\gg Q_1Q_S$ by iterating the above process.

Notice that by the same counting argument as in \eqref{DetailedCountingComputation}, 
\begin{align*}
    &\sup_{\|\beta\|_2=1}\sum_{q_2\ll C}\mathop{\sum\sum\sum\sum}_{\substack{n_0q_0q_{2,0}q_2'r=q_2\\q_{2,0}|q_0^\infty}}\sum_{\substack{q_{1,0}\ll C\\q_{1,0}|q_0^\infty}}\sum_{q_0|n_1|q_0^\infty}\mathop{\sum\sum\sum}_{\substack{n_2h_{1,0}h_{2,0} \square\text{-full}\\ h_{2,0}|(n_2h_{1,0})^\infty}}\sum_{n'\ll N_B}\mathop{\sum\sum}_{h_1',h_2'\ll H'}\sum_{0<|q_1'|\ll Q_1Q_B}\\
    &\times \frac{1}{n_1n_2^{3/2}n'h_{1,0}^{3/2}h_{2,0}h_1'q_0q_{1,0}q_{2,0}|q_1'|}\int_{|u_2|\ll1}|\beta(q_2,h_{2,0}h_2',u_2)|^2du_2\sum_{\gamma\bmod{n_1n_2h_{1,0}h_{2,0}q_0q_{1,0}q_{2,0}}}\\
    &\times \left|\sumstar_{\beta_1\bmod{q_0q_{1,0}}}\sumstar_{\beta_2\bmod{q_0q_{2,0}}}\delta\left(-\eta_j\frac{n_1n_2n'}{q_0}\equiv \overline{\beta_1h_{1,0}h_1'\gamma}q_{2,0}q_2'+\overline{\beta_2h_{2,0}h_2'q_2'}q_{1,0}\gamma \bmod{q_0q_{1,0}q_{2,0}}\right)\right|^2\ll t^\varepsilon.
\end{align*}
Recall the definition of $\mathcal{D}_2$ in (\ref{D2Def}), applying Cauchy-Schwarz inequality to take out $n_0,n_1,n_2,n',q_0,q_{1,0},q_{2,0},q_1',q_2', h_{1,0},h_1', h_{2,0}, h_2', \gamma$-sums, we get for the same fixed $\Upsilon\in C_c^\infty([1-10^{-9},1+10^{-9}])$ as before,
\begin{align}\label{S2After3rdCS}
    \mathcal{R}_j^*(N)\ll &\frac{t^{\varepsilon}}{C^{3/2}\sqrt{N_j}} \sup_{\|\beta\|_2=1} \sup_{|u_1|\ll1}\Bigg\{\int_{|u_2|\ll1}\sum_{n_0}\sum_{q_0}\mathop{\sum\sum}_{q_{1,0},q_{2,0}|q_0^\infty}\sum_{q_0|n_1|q_0^\infty}\sum_r \nonumber\\
    &\times \mathop{\sum\sum\sum}_{\substack{n_2h_{1,0}h_{2,0} \square\text{-full}\\ h_{2,0}|(n_2h_{1,0})^\infty\\ (n_2,q_0)=1}}\mathop{\sum\sum\sum}_{\substack{N_S\ll n=n_1n_2n' \ll N_B\\ H'\leq h_1=h_{1,0}h_1', h_2=h_{2,0}h_2'<(1+10^{-10})H'\\ n'h_1' \square\text{-free}, (n_1'h_1',h_2')=1 \\(n'h_1'h_2',n_1h_{1,0}h_{2,0}q_0)=1}}\mathop{\sum\sum}_{\substack{C\leq q_3=n_0q_0q_{2,0}q_3'r<(1+10^{-10})C\\C\leq q_4=n_0q_0q_{2,0}q_4'r<(1+10^{-10})C}}\nonumber\\
    &\times \sum_{Q_1Q_S\ll |q_1'|\ll Q_1Q_B}\sum_{q_2=n_0q_0q_{2,0}q_2'r}\Upsilon\left(\frac{q_2}{C}\right)\frac{h_2'^2|q_1'|\beta(q_3,h_2,u_2)\overline{\beta(q_4,h_2,u_2)}}{n_0^2n_1^3n_2^{3/2}n'\sqrt{h_{1,0}}h_1'q_0^3q_{1,0}^3q_{2,0}r^2}\nonumber\\
    &\times \mathcal{D}_2'(q_2',q_3')\mathcal{D}_2'(q_2',q_4')\overline{\mathcal{D}_0(q_3',q_4')}I_2(q_1',q_2,q_3)\overline{I_2(q_1',q_2,q_4)}du_2\Bigg\}^{1/4}+\mathcal{B}_1,
\end{align}
where \begin{align*}
    \mathcal{D}_2'(q_2',q')=&S(q_1'\overline{n_1n_2h_{1,0}h_2q_0q_{1,0}q_{2,0}}+\eta_j(\overline{q_2'}-\overline{q'})q_{1,0}\overline{n_1n_2h_2q_{2,0}},\eta_j (q_2'-q')q_{2,0}\overline{n_1n_2h_{1,0}q_{1,0}};n'h_1')\nonumber\\
    &\times \delta\left(q_1'\equiv -\eta_j q_0q_{1,0}^2h_1(\overline{q_2'}-\overline{q'})\bmod{h_2'}\right),
\end{align*}
with $\mathcal{D}_0(q_3',q_4')$ as in (\ref{D0Def}). Notice that $\mathcal{D}_2'$ is defined mod $n'h_1'h_2'$. Now we apply Poisson summation on the $q_2'$-sum to get \begin{align*}
    &\sum_{q_2'}\Upsilon\left(\frac{n_0q_0q_{2,0}q_2'r}{C}\right)\mathcal{D}_2'(q_2',q_3')\overline{\mathcal{D}_2'(q_2',q_4')}I_2(q_1',q_2,q_3)\overline{I_2(q_1',q_2,q_4)}\\
    =&\frac{C}{n_0n'h_1'h_2'q_0q_{2,0}r}\sum_{q_2'}\sum_{\gamma \bmod{n'h_1'h_2'}}\mathcal{D}_2'(\gamma,q_3')\overline{\mathcal{D}_2'(\gamma,q_4')}e\left(\frac{\gamma q_2'}{n'h_1'h_2'}\right)\\
    &\times \int_0^\infty \Upsilon(y_3)I_2(q_1',Cy_3,q_3)\overline{I_2(q_1',Cy_3,q_4)}e\left(-\frac{q_2'Cy_3}{n_0n'h_1'h_2'q_0q_{2,0}r}\right)dy_3.
\end{align*}
Applying a similar computation as the computation on $\mathcal{D}_1$ in the previous section with the Chinese remainder theorem, the character sum simplifies to 
\begin{align*}
    &\sum_{\gamma \bmod{n'h_1'h_2'}}\mathcal{D}_2'(\gamma,q_3')\overline{\mathcal{D}_2'(\gamma,q_4')}e\left(\frac{\gamma q_2'}{n'h_1'h_2'}\right)\nonumber\\
    =&\delta\left(q_3'\equiv q_4'\bmod{h_2'}\right)e\left(\frac{h_{1,0}q_0q_{1,0}^2q_2'\overline{n'(h_1q_0q_{1,0}^2\overline{q_3'}-\eta_jq_1')}}{h_2'}\right)\mathcal{D}_3(q_3',q_4'),
\end{align*}
where 
\begin{align}\label{D3Def}
    \mathcal{D}_3(q_3',q_4')=&\sumstar_{\gamma \bmod{n'h_1'}}e\left(\frac{\gamma q_2'\overline{h_2'}}{n'h_1'}\right)\nonumber\\
    &\times S(q_1'\overline{n_1n_2h_{1,0}h_2q_0q_{1,0}q_{2,0}}+\eta_j(\overline{\gamma}-\overline{q_3'})q_{1,0}\overline{n_1n_2h_2q_{2,0}},\eta_j (\gamma-q_3')q_{2,0}\overline{n_1n_2h_{1,0}q_{1,0}};n'h_1')\nonumber\\
    &\times S(q_1'\overline{n_1n_2h_{1,0}h_2q_0q_{1,0}q_{2,0}}+\eta_j(\overline{\gamma}-\overline{q_4'})q_{1,0}\overline{n_1n_2h_2q_{2,0}},\eta_j (\gamma-q_4')q_{2,0}\overline{n_1n_2h_{1,0}q_{1,0}};n'h_1').
\end{align}
Applying Lemma \ref{CharSumLemmaFirstCase}, we have \begin{align}\label{D3Bound}
    \mathcal{D}_3(q_3',q_4')\ll &(n'h_1')^2t^\varepsilon\delta\left(*_0\right)+(n'h_1')^{3/2}t^\varepsilon,
\end{align}
where $*_0$ is the condition 
\begin{align*}
    *_0=\Bigg\{& q_3'\equiv q_4'\bmod n'h_1'\nonumber\\
    &\text{or } q_1'\overline{h_{1,0}q_0q_{1,0}}\pm q_{1,0}(\overline{q_3'}-\overline{q_4'})\equiv 0 \bmod n'h_1'\nonumber\\
    &\text{or } q_1'\overline{h_{1,0}q_0q_{1,0}}-\eta_jq_{1,0}\overline{q_3'}\equiv 0\bmod n'h_1'\nonumber\\
    &\text{or } q_1'\overline{h_{1,0}q_0q_{1,0}}-\eta_jq_{1,0}\overline{q_4'}\equiv 0\bmod n'h_1'\Bigg\}.
\end{align*}

Using (\ref{z_0DiffBound}) as before, repeated integration by parts now gives us arbitrary savings unless \begin{align*}
    |q_2'|\ll& \frac{n_0n'h_1'h_2'q_0q_{2,0}r}{C}\left(1+\frac{CN^{1/3}}{nN_j^{2/3}}+\frac{(N_jN)^{1/3}}{C}\right)t^\varepsilon\ll Q_B,
\end{align*}
with $Q_B$ defined in (\ref{Q1QBDef}). Inserting this back into (\ref{S2After3rdCS}), we get \begin{align}\label{S2After3rdPoisson}
    \mathcal{R}_j^*(N)\ll &\frac{t^{\varepsilon}}{C^{3/2}\sqrt{N_j}} \sup_{\|\beta\|_2=1} \sup_{|u_1|\ll1}\Bigg\{C\int_{|u_2|\ll1}\sum_{n_0}\sum_{q_0}\mathop{\sum\sum}_{q_{1,0},q_{2,0}|q_0^\infty}\sum_{q_0|n_1|q_0^\infty}\sum_r \nonumber\\
    &\times \mathop{\sum\sum\sum}_{\substack{n_2h_{1,0}h_{2,0} \square\text{-full}\\ h_{2,0}|(n_2h_{1,0})^\infty\\ (n_2,q_0)=1}}\mathop{\sum\sum\sum}_{\substack{N_S\ll n=n_1n_2n' \ll N_B\\ H'\leq h_1=h_{1,0}h_1', h_2=h_{2,0}h_2'<(1+10^{-10})H'\\ n'h_1' \square\text{-free}, (n_1'h_1',h_2')=1 \\(n'h_1'h_2',n_1h_{1,0}h_{2,0}q_0)=1}}\mathop{\sum\sum}_{\substack{C\leq q_3=n_0q_0q_{2,0}q_3'r<(1+10^{-10})C\\C\leq q_4=n_0q_0q_{2,0}q_4'r<(1+10^{-10})C\\ q_3'\equiv q_4'\bmod{h_2'}}}\nonumber\\
    &\times \sum_{Q_1Q_S\ll |q_1'|\ll Q_1Q_B}\sum_{|q_2'|\ll Q_B}\frac{h_2'|q_1'|\beta(q_3,h_2,u_2)\overline{\beta(q_4,h_2,u_2)}}{n_0^3n_1^3n_2^{3/2}n'^2\sqrt{h_{1,0}}h_1'^2q_0^4q_{1,0}^3q_{2,0}^2r^3}\nonumber\\
    &\times \mathcal{D}_3(q_3',q_4')\mathcal{D}_{3,0}(q_3',q_4')I_3(q_1',q_2',q_3,q_4)du_2\Bigg\}^{1/4}+\mathcal{B}_1,
\end{align}
where \begin{align}\label{D30Def}
    \mathcal{D}_{3,0}(q_3',q_4')=e\left(\frac{h_{1,0}q_0q_{1,0}^2q_2'\overline{n'(h_1q_0q_{1,0}^2\overline{q_3'}-\eta_jq_1')}}{h_2'}\right)\overline{\mathcal{D}_0(q_3',q_4')}\ll n_1n_2h_{1,0}h_{2,0}q_0^3q_{1,0}q_{2,0}
\end{align}
by (\ref{D0Def}), and \begin{align}\label{I3Def}
    I_3(q_1',q_2',q_3,q_4)=&\int_0^\infty \Upsilon(y_3)I_2\left(q_1',Cy_3,q_3\right)\overline{I_2\left(q_1',Cy_3,q_4\right)}e\left(-\frac{q_2'Cy_3}{n_0n'h_1'h_2'q_0q_{2,0}r}\right)dy_3.
\end{align}

Now we want to again separate the contribution when $|q_2'|$ is small, applying similar treatment as the case $|q_1'|\ll Q_1Q_S$. Precisely, we first deal with the case \begin{align*}
    |q_2'|\ll Q_S=\frac{n_0q_0q_{2,0}rC^3t^{2+\varepsilon}}{n_1n_2nh_{1,0}h_{2,0}(N_jN)^{4/3}}
\end{align*}
including the case $q_2'=0$. Applying (1) in Lemma \ref{statlemma} on the $y_3$-integral, we get arbitrary savings unless there exists $y_3'\in [1/2,5/2]$ such that \begin{align*}
    \left|\frac{d}{dy_3}\bigg|_{y=y_3'}\left(G(y_1,y_3)-G(y_2,y_3)\right)\right|\ll \frac{C^2N^{2/3}t^\varepsilon}{n^2N_j^{4/3}}.
\end{align*}
Applying the same analysis as in the previous subsection (to get (\ref{q2q30close})) with Lemma \ref{OGfDer}, (\ref{z0condition}) on $z_0(Cy_1,y_3),z_0(Cy_2,y_3)\sim1$ in $\Psi$, the above bound for the derivative yields 
\begin{align}\label{y_1y_20close}
    |y_1-y_2|\ll \frac{C^2t^\varepsilon}{nN_j}.
\end{align}
Inserting these conditions into $I_3(q_1',q_2',q_3,q_4)$, writing $y_2=y_1+v$ and bounding the $y_3$-integral trivially, we have for $|q_2'|\ll Q_S$, \begin{align*}
    &I_3(q_1',q_2',q_3,q_4)\\
    \ll& t^\varepsilon\sup_{y_3\sim 1}\Bigg| \int_{|v|\ll\frac{C^2t^\varepsilon}{nN_j}}\int_0^\infty\Upsilon(y_1)\Upsilon(y_1+v)\Psi(Cy_1,Cy_3,q_3)\overline{\Psi(C(y_1+v),Cy_3,q_4)}\nonumber\\
    &\times e\left(G(y_1,y_3)-G\left(y_1,\frac{q_3}{C}\right)-G(y_1+v,y_3)+G\left(y_1+v,\frac{q_4}{C}\right)+\frac{q_1'Cv}{n_0nh_1h_2q_0^2q_{1,0}^2q_{2,0}r}\right)dy_1dv\Bigg|+ t^{-A}
\end{align*}
for any $A>0$. Applying Taylor expansion on the functions involving $1+\frac{v}{y_1}$ and $1+\frac{v}{y_1}$, using $H\gg\frac{N}{t^{1-\varepsilon}}$, and bounding the $v_1$-integral trivially, there exists some $\frac{CN}{n^{3/2}N_j\sqrt{t}}t^\varepsilon$-inert function $\varphi$ such that \begin{align*}
    I_3(q_1',q_2',q_3,q_4)\ll& \frac{C^2t^\varepsilon}{nN_j}\sup_{|v_1|\ll \frac{C^2t^\varepsilon}{nN_j}}\left|\int_0^\infty\Upsilon(y)\Psi_2(Cy,q_3,q_4) e\left(-G\left(y,\frac{q_3}{C}\right)+G\left(y,\frac{q_4}{C}\right)\right)dy\right|+t^{-A},
\end{align*}
with 
\begin{align*}
    \Psi_2(Cy,q_3,q_4)=\Psi(Cy,Cy_3,q_3)\overline{\Psi(C(y+v),Cy_3,q_4)}\varphi(y,y_3,v,q_3,q_4),
\end{align*}
being a $\left(\frac{CN^{1/3}}{nN_j^{2/3}}+\frac{CN}{n^{3/2}N_j\sqrt{t}}\right)t^\varepsilon$-inert function. Applying the same repeated integration by parts process on the $y$-integral using \eqref{z_0DiffBound}, we get 
\begin{align*}
    I_3(q_1',q_2',q_3,q_4)\ll& \frac{C^2t^\varepsilon}{nN_j}\delta\left(|q_3-q_4|\ll \left(\frac{C^2}{(N_jN)^{1/3}}+\frac{C^2N^{1/3}}{\sqrt{nt}N_j^{2/3}}\right)t^\varepsilon\right)+t^{-A},
\end{align*}
for any $A>0$.

Inserting (\ref{D30Def}) and the above bound for $I_3$ into (\ref{S2After3rdPoisson}), we get the contribution of $|q_2'|\ll Q_S$ is bounded by \begin{align*}
    \ll &\frac{t^{\varepsilon}}{C^{3/2}\sqrt{N_j}} \sup_{\|\beta\|_2=1} \sup_{|u_1|\ll1}\Bigg\{C\int_{|u_2|\ll1}\sum_{n_0}\sum_{q_0}\mathop{\sum\sum}_{q_{1,0},q_{2,0}|q_0^\infty}\sum_{q_0|n_1|q_0^\infty}\sum_r \nonumber\\
    &\times \mathop{\sum\sum\sum}_{\substack{n_2h_{1,0}h_{2,0} \square\text{-full}\\ h_{2,0}|(n_2h_{1,0})^\infty\\ (n_2,q_0)=1}}\mathop{\sum\sum\sum}_{\substack{N_S\ll n=n_1n_2n' \ll N_B\\ H'\leq h_1=h_{1,0}h_1', h_2=h_{2,0}h_2'<(1+10^{-10})H'\\ n'h_1' \square\text{-free}, (n_1'h_1',h_2')=1 \\(n'h_1'h_2',n_1h_{1,0}h_{2,0}q_0)=1}}\mathop{\sum\sum}_{\substack{C\leq q_3=n_0q_0q_{2,0}q_3'r<(1+10^{-10})C\\C\leq q_4=n_0q_0q_{2,0}q_4'r<(1+10^{-10})C\\ |q_3-q_4|\ll \left(\frac{C^2}{(N_jN)^{1/3}}+\frac{C^2N^{1/3}}{\sqrt{nt}N_j^{2/3}}\right)t^\varepsilon\\ q_3'\equiv q_4'\bmod{h_2'}}}\nonumber\\
    &\times \sum_{Q_1Q_S\ll |q_1'|\ll Q_1Q_B}\sum_{|q_2'|\ll Q_S}\frac{\sqrt{h_{1,0}}h_2|q_1'||\beta(q_3,h_2,u_2)\beta(q_4,h_2,u_2)|}{n_0^3n_1^2\sqrt{n_2}n'^2h_1'^2q_0q_{1,0}^2q_{2,0}r^3}\frac{C^2}{nN_j}|\mathcal{D}_3(q_3',q_4')|\mathcal{D}_{3,0}(q_3',q_4')du_2\Bigg\}^{1/4}.
\end{align*}
Recall that $Q_S, Q_B$ are defined in (\ref{QsDef}) and (\ref{Q1QBDef}) respectively. Applying the AM-GM inequality \begin{align*}
    |\beta(q_2,h_2,u_2)\beta(q_3,h_2,u_2)|\ll |\beta(q_2,h_2,u_2)|^2+|\beta(q_3,h_2,u_2)|^2,
\end{align*}
using (\ref{D3Bound}) and \eqref{D30Def} with the same counting argument as in \eqref{DetailedCountingComputation}, the above is bounded by \begin{align*}
    \ll &\frac{t^\varepsilon}{C^{3/2}\sqrt{N_j}}\Bigg\{\sum_{n\ll \frac{CN^{1/3}t^\varepsilon}{N_j^{2/3}}}\frac{C^3}{nN_j}\left(\frac{nN_j^{1/3}t^2}{N^{5/3}}\right)^2\left(1+\frac{C^3t^2}{n(N_jN)^{4/3}}\right)\frac{1}{n^2}\left(\frac{Ct}{N}\right)^{-1}\\
    &\times\left(\underbrace{\frac{C^2}{(N_jN)^{1/3}}+\frac{C^2N^{1/3}}{\sqrt{nt}N_j^{2/3}}}_{q_3'\neq q_4'}+\underbrace{\frac{Ct}{N}}_{q_3'=q_4'}\right)\left(\frac{nCt}{N}\right)^{3/2}\Bigg\}^{1/4}\\
    \ll&\frac{t^{9/8+\varepsilon}}{C^{1/4}(N_jN)^{5/6}} \left(1+\frac{\sqrt{Ct}}{N_j^{1/6}N^{5/12}}\right)\left(\frac{C^2}{(N_jN)^{1/3}}+\frac{Ct}{N}\right)^{1/4}.
\end{align*}

Indeed, with the same process we have the following technical Lemma that we will apply repeatedly later.

\begin{Lemma}\label{TechnicalLemmaForDiagonal}
    Let $q_3,q_4\sim C$. Let $B\ll\frac{C^2N^{2/3}t^\varepsilon}{n^2N_j^{4/3}}$ and let $\eta=0$ or $1$. Let $W$ be a $\left(\frac{CN^{1/3}}{nN_j^{2/3}}+\eta \frac{CN}{n^{3/2}N_j\sqrt{t}}\right)t^\varepsilon$-inert function and let \begin{align*}
        J=&\int_0^\infty\Upsilon(y)\int_0^\infty\Upsilon(y_1)\int_0^\infty\Upsilon(y_2)\Psi(Cy_1,Cy,q_3)\overline{\Psi(Cy_2,Cy,q_4)}W(y,y_1,y_2,q_3,q_4)\nonumber\\
        &\times e\left(G(y_1,y)-G\left(y_1,\frac{q_3}{C}\right)-G(y_2,y)+G\left(y_2,\frac{q_4}{C}\right)+(1-\eta)By\right)dy_2dy_1dy.
    \end{align*}
    Then there exists some $\left(\frac{CN^{1/3}}{nN_j^{2/3}}+\frac{CN}{n^{3/2}N_j\sqrt{t}}\right)t^\varepsilon$-inert function $\tilde{W}$ and some function $F$ such that \begin{align*}
        J=&\int_{1/2}^{5/2}\int_{|v|\ll \left((1-\eta)\frac{C^2}{nN_j}+\eta \left(\frac{C}{(N_jN)^{1/3}}+\frac{CN^{1/3}}{\sqrt{nt}N_j^{2/3}}\right)\right)t^\varepsilon} F(u,v)\int_0^\infty \Upsilon (y)\Upsilon(y+v)\Psi(Cy,Cu,q_3)\\
        &\times \overline{\Psi(C(y+v),Cu,q_4)}\tilde{W}(u,y,y+v,q_3,q_4)e\left(-G\left(y,\frac{q_3}{C}\right)+G\left(y,\frac{q_4}{C}\right)\right)dydvdu.
    \end{align*}
    Moreover, we have \begin{align*}
        J\ll& \left((1-\eta)\frac{C^2}{nN_j}+\eta\left(\frac{C}{(N_jN)^{1/3}}+ \frac{CN^{1/3}}{\sqrt{nt}N_j^{2/3}}\right)\right)t^\varepsilon\\
        &\times \delta\left(|q_3-q_4|\ll \left(\frac{C^2}{(N_jN)^{1/3}}+\frac{C^2N^{1/3}}{\sqrt{nt}N_j^{2/3}}\right)t^\varepsilon\right)+t^{-A}
    \end{align*}
    for any $A>0$.
\end{Lemma}

All together we have \begin{align*}
    \mathcal{R}_j^*(N)\ll &\frac{t^{\varepsilon}}{C^{3/2}\sqrt{N_j}} \sup_{\|\beta\|_2=1} \sup_{|u_1|\ll1}\Bigg\{C\int_{|u_2|\ll1}\sum_{n_0}\sum_{q_0}\mathop{\sum\sum}_{q_{1,0},q_{2,0}|q_0^\infty}\sum_{q_0|n_1|q_0^\infty}\sum_r \nonumber\\
    &\times \mathop{\sum\sum\sum}_{\substack{n_2h_{1,0}h_{2,0} \square\text{-full}\\ h_{2,0}|(n_2h_{1,0})^\infty\\ (n_2,q_0)=1}}\mathop{\sum\sum\sum}_{\substack{N_S\ll n=n_1n_2n' \ll N_B\\ H'\leq h_1=h_{1,0}h_1', h_2=h_{2,0}h_2'<(1+10^{-10})H'\\ n'h_1' \square\text{-free}, (n_1'h_1',h_2')=1 \\(n'h_1'h_2',n_1h_{1,0}h_{2,0}q_0)=1}}\mathop{\sum\sum}_{\substack{C\leq q_3=n_0q_0q_{2,0}q_3'r<(1+10^{-10})C\\C\leq q_4=n_0q_0q_{2,0}q_4'r<(1+10^{-10})C\\ q_3'\equiv q_4'\bmod{h_2'}}}\nonumber\\
    &\times \sum_{Q_1Q_S\ll |q_1'|\ll Q_1Q_B}\sum_{Q_S\ll |q_2'|\ll Q_B}\frac{h_2'|q_1'|\beta(q_3,h_2,u_2)\overline{\beta(q_4,h_2,u_2)}}{n_0^3n_1^3n_2^{3/2}n'^2\sqrt{h_{1,0}}h_1'^2q_0^4q_{1,0}^3q_{2,0}^2r^3}\nonumber\\
    &\times \mathcal{D}_3(q_3',q_4')\mathcal{D}_{3,0}(q_3',q_4')I_3(q_1',q_2',q_3,q_4)du_2\Bigg\}^{1/4}+\mathcal{B}_1+\mathcal{B}_2,
\end{align*}
where $\mathcal{B}_1$ is defined in (\ref{B1Def}) and \begin{align}\label{B2Def}
    \mathcal{B}_2:=&\frac{t^{9/8+\varepsilon}}{C^{1/4}(N_jN)^{5/6}} \left(1+\frac{\sqrt{Ct}}{N_j^{1/6}N^{5/12}}\right)\left(\frac{C^2}{(N_jN)^{1/3}}+\frac{Ct}{N}\right)^{1/4}.
\end{align}

\subsubsection{Iteration of the above Process}

Iterating the above process, we obtain the following lemma.

\begin{Lemma}\label{IteraionLemma}
    For any positive integer $\nu\geq2$, we have \begin{align*}
        \mathcal{R}_j^*(N)\ll |\mathcal{A}_\nu|+\sum_{j=1}^\nu\mathcal{B}_j,
    \end{align*}
    where \begin{align*}
    \mathcal{A}_\nu=&\frac{t^{\varepsilon}}{C^{3/2}\sqrt{N_j}} \sup_{\|\beta\|_2=1} \sup_{|u_1|\ll1}\Bigg\{\int_{|u_2|\ll1}\sum_{n_0}\sum_{q_0}\mathop{\sum\sum}_{q_{1,0},q_{2,0}|q_0^\infty}\sum_{q_0|n_1|q_0^\infty}\sum_r \sum_{Q_1Q_S\ll |q_1'|\ll Q_1Q_B}\mathop{\sum\cdots\sum}_{Q_S\ll |q_2'|,...,|q_\nu'|\ll Q_B}\nonumber\\
    &\times \mathop{\sum\sum\sum}_{\substack{n_2h_{1,0}h_{2,0} \square\text{-full}\\ h_{2,0}|(n_2h_{1,0})^\infty\\ (n_2,q_0)=1}}\mathop{\sum\sum\sum}_{\substack{N_S\ll n=n_1n_2n' \ll N_B\\ H'\leq h_1=h_{1,0}h_1', h_2=h_{2,0}h_2'<(1+10^{-10})H'\\ n'h_1' \square\text{-free}, (n_1'h_1',h_2')=1 \\(n'h_1'h_2',n_1h_{1,0}h_{2,0}q_0)=1}}\mathop{\sum\sum}_{\substack{C\leq q_{\nu+1}=n_0q_0q_{2,0}q_{\nu+1}'r<(1+10^{-10})C\\C\leq q_{\nu+2}=n_0q_0q_{2,0}q_{\nu+2}'r<(1+10^{-10})C\\ q_{\nu+1}'\equiv q_{\nu+2}'\bmod{h_2'}}}\nonumber\\
    &\times \left(\frac{\sqrt{h_{1,0}h_{2,0}}C}{n_0^2n_1n'h_1'q_0^2q_{1,0}q_{2,0}r^2}\right)^{2^{\nu-1}}\frac{n_0h_2'r}{n_1n_2^{3/2}h_{1,0}^{3/2}h_{2,0}q_{1,0}C}\prod_{j=1}^{\nu-1}|q_j|^{2^{\nu-j}-1}D(\nu)\nonumber\\
    &\times \beta'(q_{\nu+1},h_2,u_2)\overline{\beta'(q_{\nu+2},h_2,u_2)}\mathcal{D}_{\nu+1}(q_{\nu+1}',q_{\nu+2}')\mathcal{D}_0'(q_{\nu+1}',q_{\nu+2}')I_{\nu+1}(q_1',...,q_\nu',q_{\nu+1},q_{\nu+2})du_2\Bigg\}^{2^{-\nu}},
\end{align*}
    with the notations $\beta'=\beta$ or $\overline{\beta}$, $\mathcal{D}_0'=\mathcal{D}_0$ or $\overline{\mathcal{D}_0}$ depending on $\nu$, with $\mathcal{B}_1$, $\mathcal{D}_0$, $\mathcal{D}_3$, $I_3$ as defined in (\ref{B1Def}), (\ref{D0Def}), (\ref{D3Def}) and (\ref{I3Def}) respectively, \begin{align*}
        D(\nu)=\begin{cases}
            e\left(\frac{h_{1,0}q_0q_{1,0}^2q_2'\overline{n'(h_1q_0q_{1,0}^2\overline{q_3'}-\eta_jq_1')}}{h_2'}\right) & \nu=2\\
            e\left(\frac{q_{\nu}'q_{\nu+1}'\overline{n'h_1'}}{h_2'}\right) & \nu>2.
        \end{cases}
    \end{align*}
    The character sum and integral $\mathcal{D}_\nu(q_{\nu+1}',q_{\nu+2}')$ and $I_\nu(q_1',...,q_\nu',q_{\nu+1},q_{\nu+2})$ are defined recursively by \begin{align*}
        \mathcal{D}_{\nu+1}(q_{\nu+1}',q_{\nu+2}')=\sum_{\gamma\bmod{n'h_1'}}\mathcal{D}_\nu(\gamma,q_{\nu+1}')\overline{\mathcal{D}_\nu(\gamma,q_{\nu+2}')}e\left(\frac{\gamma q_\nu'\overline{h_2'}}{n'h_1'}\right)
    \end{align*}
    and
    \begin{align*}
        I_{\nu+1}(q_1',...,q_\nu',q_{\nu+1},q_{\nu+2})=&\int_0^\infty\Upsilon(y_{2^\nu-1})I_{\nu}(q_1',...,q_{\nu-1}',Cy_{2^\nu-1},q_{\nu+1})\\
        &\times\overline{I_{\nu}(q_1',...,q_{\nu-1}',Cy_{2^\nu-1},q_{\nu+2})}e\left(-\frac{q_{\nu}'Cy_{2^\nu-1}}{n_0n'h_1h_2q_0q_{2,0}r}\right)dy_{2^\nu-1}
    \end{align*}
    respectively, and for $\ell\geq2$, \begin{align*}
        \mathcal{B}_\ell=&\frac{\sqrt{C}t^{2+\varepsilon}}{N_jN^{3/2}} \Bigg\{\frac{N_j^{2/3}N^{8/3}}{C^3t^{7/2}}\left(1+\frac{C^2t^2}{N_j^{2/3}N^{5/3}}\right)\left(\frac{Ct}{N}+\frac{C^2}{(N_jN)^{1/3}}\right)\Bigg\}^{2^{-\ell}}.
    \end{align*}
    Moreover, we have the bound \begin{align*}
        \mathcal{D}_{\nu+1}(q_{\nu+1}',q_{\nu+2}')\ll (n'h_1')^{2^{\nu-1}-1/2}t^\varepsilon\left(\sqrt{n'h_1'}\delta(*_{\nu-2})+1\right),
    \end{align*}
    where $*_{\nu-2}$ is the condition \begin{align*}
        *_{\nu-2}=\Bigg\{& q_{\nu+1}'\equiv q_{\nu+2}'\bmod n'h_1'\nonumber\\
        &\text{or } q_1'\overline{q_0q_{1,0}}\pm h_1q_{1,0}(\overline{q_{\nu+1}'}-\overline{q_{\nu+2}'})\equiv 0 \bmod n'h_1'\nonumber\\
        &\text{or } q_1'\overline{q_0q_{1,0}}-\eta_jh_1q_{1,0}\overline{q_{\nu+1}'}\equiv 0\bmod n'h_1'\nonumber\\
        &\text{or } q_1'\overline{q_0q_{1,0}}-\eta_jh_1q_{1,0}\overline{q_{\nu+2}'}\equiv 0\bmod n'h_1'\Bigg\}.
    \end{align*}
\end{Lemma}

\begin{proof}
    Proof is presented in appendix \ref{sect.IterationLemmaProof}.
\end{proof}

With the above lemma and the bound for $\mathcal{D}_0$ in \eqref{D0Def}, we are left to bound $I_{\nu+1}(q_1',...,q_\nu',q_{\nu+1},q_{\nu+2})$. We do so with following lemma.

\begin{Lemma}\label{InuBound}
    Let $q_1',...,q_\nu',q_{\nu+1},q_{\nu+2}$ be integers such that $|q_1'|\gg Q_1Q_S$, $|q_2'|,...,|q_\nu'|\gg Q_S$ and $q_{\nu+1},q_{\nu+2}\sim C$. With $I_3$ defined in (\ref{I3Def}) and $I_\nu$ defined in Lemma \ref{IteraionLemma}, we have the following bound. \begin{align*}
        I_{\nu+1}(q_1',...,q_\nu',q_{\nu+1},q_{\nu+2})\ll_\nu&\prod_{j=2}^\nu\left(1+\frac{(N_jN)^{1/3}n_0n'h_1h_2q_0q_{2,0}r}{\sqrt{\min\left\{q_u'':1\leq u\leq j-1\right\}q_j''}C^2}\right)^{2^{\nu-j}}\\
        &\times \frac{(n_1n_2h_{1,0}h_{2,0}q_0q_{1,0}^2)^{2^{\nu-2}}(n_0n'h_1h_2q_0q_{2,0}r)^{2^{\nu-1}-1/2}}{q_1'^{2^{\nu-2}}q_2'^{2^{\nu-3}}\cdots q_\nu'^{2^{-1}}C^{2^{\nu-1}-1/2}}t^\varepsilon
    \end{align*}
    for $\nu\geq2$.
\end{Lemma}
\begin{proof}
    The proof is presented in appendix \ref{sect.InuBoundProof} via a careful iterative stationary phase analysis. Indeed, we have to use (3) instead of (2) in Lemma \ref{statlemma}, and the parameter $U$ for each stationary phase step has to be chosen carefully to yield a recursive structure.
\end{proof}

Inserting the above bound into the bound of $\mathcal{R}_j^*(N)$ in Lemma \ref{IteraionLemma}, together with $\mathcal{D}(k+1)\ll1$, $\mathcal{D}_0\ll n_1n_2h_{1,0}h_{2,0}q_0^3q_{1,0}q_{2,0}$ as in (\ref{D0Def}) and the last statement of Lemma \ref{IteraionLemma} on $\mathcal{D}_{\nu+1}(q_{\nu+1}',q_{\nu+2}')$, we get for any $\nu\geq2$,

\begin{align*}
    \mathcal{R}_j^*(N)\ll &\frac{t^{\varepsilon}}{C^{3/2}\sqrt{N_j}} \sup_{\|\beta\|_2=1} \sup_{|u_1|\ll1}\Bigg\{\int_{|u_2|\ll1}\sum_{n_0}\sum_{q_0}\mathop{\sum\sum}_{q_{1,0},q_{2,0}|q_0^\infty}\sum_{q_0|n_1|q_0^\infty}\sum_r \sum_{Q_1Q_S\ll |q_1'|\ll Q_1Q_B}\mathop{\sum\cdots\sum}_{Q_S\ll |q_2'|,...,|q_\nu'|\ll Q_B}\nonumber\\
    &\times \mathop{\sum\sum\sum}_{\substack{n_2h_{1,0}h_{2,0} \square\text{-full}\\ h_{2,0}|(n_2h_{1,0})^\infty\\ (n_2,q_0)=1}}\mathop{\sum\sum\sum}_{\substack{N_S\ll n=n_1n_2n' \ll N_B\\ H'\leq h_1=h_{1,0}h_1', h_2=h_{2,0}h_2'<(1+10^{-10})H'\\ n'h_1' \square\text{-free}, (n_1'h_1',h_2')=1 \\(n'h_1'h_2',n_1h_{1,0}h_{2,0}q_0)=1}}\mathop{\sum\sum}_{\substack{C\leq q_{\nu+1}=n_0q_0q_{2,0}q_{\nu+1}'r<(1+10^{-10})C\\C\leq q_{\nu+2}=n_0q_0q_{2,0}q_{\nu+2}'r<(1+10^{-10})C\\ q_{\nu+1}'\equiv q_{\nu+2}'\bmod{h_2'}}}\nonumber\\
    &\times \left(\frac{\sqrt{h_{1,0}h_{2,0}}C}{n_0^2n_1n'h_1'q_0^2q_{1,0}q_{2,0}r^2}\right)^{2^{\nu-1}}\frac{n_0h_2'r}{n_1n_2^{3/2}h_{1,0}^{3/2}h_{2,0}q_{1,0}C}\prod_{j=1}^{\nu-1}|q_j|^{2^{\nu-j}-1}|\beta(q_{\nu+1},h_2,u_2)|\nonumber\\
    &\times |\beta(q_{\nu+2},h_2,u_2)|du_2(n'h_1')^{2^{\nu-1}-1/2}\left(\sqrt{n'h_1'}\delta(*_{\nu-2})+1\right)n_1n_2h_{1,0}h_{2,0}q_0^3q_{1,0}q_{2,0}\nonumber\\
    &\times \prod_{j=2}^\nu\left(1+\frac{(N_jN)^{1/3}n_0n'h_1h_2q_0q_{2,0}r}{\sqrt{\min\left\{q_u'':1\leq u\leq j-1\right\}q_j''}C^2}\right)^{2^{\nu-j}}\\
    &\times \frac{(n_1n_2h_{1,0}h_{2,0}q_0q_{1,0}^2)^{2^{\nu-2}}(n_0n'h_1h_2q_0q_{2,0}r)^{2^{\nu-1}-1/2}}{q_1'^{2^{\nu-2}}q_2'^{2^{\nu-3}}\cdots q_\nu'^{2^{-1}}C^{2^{\nu-1}-1/2}}\Bigg\}^{2^{-\nu}}+\sum_{\ell=1}^\nu\mathcal{B}_\ell,
\end{align*}
with $(*_{\nu-2})$ as defined in Lemma \ref{IteraionLemma}. Recall the definition of $Q_B$ in (\ref{Q1QBDef}). Applying the AM-GM inequality \begin{align*}
    |\beta(q_{\nu+1},h_2,u_2)\beta(q_{\nu+2},h_2,u_2)|\ll |\beta(q_{\nu+1},h_2,u_2)|^2+|\beta(q_{\nu+2},h_2,u_2)|^2,
\end{align*}
the same counting argument as in \eqref{DetailedCountingComputation} implies that the above is bounded by \begin{align}\label{S2BignBound}
    \mathcal{R}_j^*(N)\ll &\frac{t^{\varepsilon}}{C^{3/2}\sqrt{N_j}}\Bigg\{\sum_{n\ll\frac{CN^{1/3}t^\varepsilon}{N_j^{2/3}}}\frac{n^{2^{\nu-1}-1/2}}{\sqrt{C}}\left(\frac{Ct}{N}\right)^{2^\nu}\left(\frac{nN_j^{1/3}t^2}{N^{5/3}}\right)^{2^{\nu-1}-3/2}\left(\frac{Ct}{N}+C\right)\Bigg\}^{2^{-\nu}}+\nu\sup_{1\leq \ell\leq \nu}\mathcal{B}_\ell\nonumber\\
    \ll& \frac{\sqrt{C}t^{2+\varepsilon+99/2^\nu}}{N_jN^{3/2}}+\nu\sup_{1\leq \ell\leq \nu}\mathcal{B}_\ell
\end{align}
for $N\gg t^{1-\varepsilon}$.

\begin{Remark}
    To get the above bound in the first line, first observe that the worst case scenario for $q_j'$ is when $q_2',...,q_\nu'$ are as big as possible, and $|q_1'|$ is either 1 or as big as possible.
\end{Remark}

\subsection{Bound for \texorpdfstring{$\mathcal{R}_j(N)$}{Rj(N)}}

Recall the definition of $\mathcal{R}_j(N)$ in Lemma \ref{LemmaAfterFirstStep}, the definition of $\mathcal{B_\ell}$ in (\ref{B1Def}) and Lemma \ref{IteraionLemma}. Taking $\nu=t^\varepsilon$ in (\ref{S2BignBound}) and combining it with the bounds (\ref{DiagonalBound}), (\ref{SmallnBound}) and (\ref{MidnBound}), we get for $N\gg t^{1-\varepsilon}$, \begin{align}\label{RjBound}
    \mathcal{R}_j(N)\ll& t^\varepsilon\Bigg\{\frac{1}{C^2H}+\frac{Ht}{CN_j^{1/3}N^{4/3}}+ \frac{t}{N_jN}\left(1+\frac{C^2}{(N_jN)^{1/3}}+\frac{CN}{Ht}\right)\delta\left(N_j\ll C^2t^\varepsilon\right)\nonumber\\
    &+\frac{\sqrt{t}}{\sqrt{CH}N_j^{5/6}N^{1/3}}\left(1+\frac{CN}{Ht}\right)\delta\left(N_j\ll\frac{C^{3/2}N^2}{(Ht)^{3/2}}t^\varepsilon\right)+\frac{\sqrt{C}t^{2}}{N_jN^{3/2}}+\frac{t}{(N_jN)^{5/6}}+\frac{t^{1/4}}{N_j^{5/6}N^{1/3}}\nonumber\\
    &+ \frac{t^{5/4}}{N_jN}+ \frac{C^{3/4}t^{5/4}}{(N_jN)^{13/12}}+\sup_{2\leq \ell\leq t^\varepsilon}\frac{\sqrt{C}t^{2}}{N_jN^{3/2}}\Bigg\{\frac{N_j^{2/3}N^{8/3}}{C^3t^{7/2}}\left(1+\frac{C^2t^2}{N_j^{2/3}N^{5/3}}\right)\left(\frac{Ct}{N}+\frac{C^2}{(N_jN)^{1/3}}\right)\Bigg\}^{2^{-\ell}}.
\end{align}

\section{Final Bound}

We choose $Q$ such that $\frac{N}{t^{1-\varepsilon}}<Q<\frac{N}{\sqrt{Ht}}t^\varepsilon$. Inserting (\ref{RjBound}) into Lemma \ref{LemmaAfterFirstStep}, the conditions $\frac{N}{t^{1+\varepsilon}}\ll C\ll Qt^\varepsilon$, $t^{1/2+\varepsilon}\ll H\ll\frac{Nt^\varepsilon}{M}$, $t^{1/2+\varepsilon}\ll M\ll t^{1-\varepsilon}$, $N\ll t^{3/2+\varepsilon}$, $N_1\sim\frac{N^2X^3}{Q^3}$ and $N_2\ll \left(\frac{N^2}{Q^3}+\frac{(CHt)^3}{N^4}\right)t^\varepsilon\ll\frac{N^2t^\varepsilon}{Q^3}$ (by the choice of $Q$) implies \begin{align*}
    &\left(S_{1,X}(N)S_2(N)\right)^{1/2}\\
    \ll&\frac{CM(N_1N_2)^{2/3}N^{4/3}}{t^{1-\varepsilon}}\Bigg\{\left(\frac{Ct}{N_1N}\delta\left(N_1\ll\frac{C^3t^\varepsilon}{N}\right)+\mathcal{R}_1(N)\right)\mathcal{R}_2(N)\Bigg\}^{1/2}\\
    \ll&t^\varepsilon\Bigg\{Q^3M+\frac{MHN^2}{Q^3}+Q^{5/3}MN^{1/3}\left(1+\frac{Q^{4/3}}{N^{1/3}}+\frac{QN}{Ht}\right)+\frac{Q^{5/4}MN^2}{(Ht)^{5/4}}\left(1+\frac{QN}{Ht}\right)\\
    &+\frac{MN^{3/2}}{\sqrt{Q}}+QMN^{3/4}t^{1/4}+\sqrt{QN}Mt\left(1+\frac{N}{Q^{5/4}t^{7/8}}\left(1+\frac{Q\sqrt{t}}{N^{3/4}}\right)\left(\left(\frac{Qt}{N}\right)^{1/4}+\frac{Q^{3/4}}{N^{1/4}}\right)\right)\Bigg\},
\end{align*}
including the case $X=0$.

We choose $Q=\frac{M^{5/6}H^{1/3}N^{1/3}}{t^{3/4}}$, then $Q$ satisfies $\frac{N}{t^{1-\varepsilon}}<Q<\frac{N}{\sqrt{Ht}}t^\varepsilon$ when $H\gg \frac{N^2}{M^{5/2}t^{3/4-\varepsilon}}$. Inserting the bound into Theorem \ref{S1S2Thm} and then (\ref{qDyadicSub}), we get for $t^{1/2+\varepsilon}\ll M\ll t^{1-\varepsilon}$, $N\ll t^{3/2+\varepsilon}$, $\sqrt{t}+\frac{N^2}{M^{5/2}t^{3/4-\varepsilon}}\ll H\ll \frac{Nt^\varepsilon}{M}$, \begin{align*}
    \frac{S_{M,H}^\pm(N)}{N}\ll& t^\varepsilon\Bigg\{\frac{M^{5/2}}{t^{3/4}}+\frac{t^{9/4}}{M^{3/2}}+\frac{M^{29/9}}{H^{1/9}t^{7/6}}+\frac{M^{23/8}t^{3/16}}{H^{3/2}}+\frac{M^{7/12}t^{7/8}}{H^{1/6}}+M^{3/2}t^{1/8}+\frac{M^{5/4}t^{5/8}}{N^{1/6}}\Bigg\}.
\end{align*}

\begin{Remark}
    The optimal choice of $Q$ presented above is determined by minimizing $\frac{MHN^2}{Q^3}+\sqrt{QN}Mt$ restricted to $Q<\frac{Nt^\varepsilon}{\sqrt{Ht}}$.
\end{Remark}

Inserting the above bound together with the trivial bound (\ref{trivialBound}) into Lemma \ref{SMHLemma}, we get \begin{align*}
    \left|L\left(\frac{1}{2}+it,\pi\right)\right|^2\ll&\log t\left(1+\int_\BR U\left(\frac{v}{M}\right)\left|L\left(\frac{1}{2}+it+iv,\pi\right)\right|^2dv\right)\nonumber\\
    \ll& t^\varepsilon\Bigg(\frac{M^{5/2}}{t^{3/4}}+\frac{t^{9/4}}{M^{3/2}}+\sup_{H}\min\left\{\frac{M^{29/9}}{H^{1/9}t^{7/6}}+\frac{M^{23/8}t^{3/16}}{H^{3/2}}+\frac{M^{7/12}t^{7/8}}{H^{1/6}}, MH\right\}\nonumber\\
    &+M^{3/2}t^{1/8}+\sup_{N}\min\left\{\frac{M^{5/4}t^{5/8}}{N^{1/6}},N\right\}\Bigg)\nonumber\\
    \ll & t^\varepsilon\left(\frac{M^{5/2}}{t^{3/4}}+\frac{t^{9/4}}{M^{3/2}}+\frac{M^3}{t^{21/20}}+M^{7/4}t^{3/40}+M^{9/14}t^{3/4}+M^{3/2}t^{1/8}+M^{15/14}t^{15/28}\right)\nonumber\\
    \ll &t^\varepsilon\left(\frac{t^{9/4}}{M^{3/2}}+\frac{M^3}{t^{21/20}}+M^{7/4}t^{3/40}+M^{15/14}t^{15/28}\right).
\end{align*}
The optimal choice of $M$ is $M=t^{2/3}$, which yields 
\begin{align*}
    \left|L\left(\frac{1}{2}+it,\pi\right)\right|^2\ll\log t\left(1+\int_\BR U\left(\frac{v}{t^{2/3}}\right)\left|L\left(\frac{1}{2}+it+iv,\pi\right)\right|^2dv\right)\ll t^{5/4+\varepsilon}.
\end{align*}
This completes the proof of Theorem \ref{SecondMomentThm} and Theorem \ref{main theorem gl3}.

\appendix

\section{Integral Analysis}

\subsection{Stationary Phase Analysis of \texorpdfstring{$\mathcal{J}_0(m,m_0,q_1,q_2,x_1,x_2)$}{J0}}\label{sect.J0IntegralAnalysis}

Let $x_1,x_2$ be such that $\varphi_X(x_j)\neq0$ for $j=1,2$, with $\varphi_X$ as defined in (\ref{PhiXDef}). Recall from (\ref{J11Def}) that \begin{align*}
    &\mathcal{J}_0(m,m_0,q_1,q_2,x_1,x_2)\nonumber\\
    =&\int_0^\infty \varphi(z)U_{\eta_1} \left(\frac{N_1z}{(m_0q_1)^3},m_0q_1,x_1\right)\overline{U_{\eta_1} \left(\frac{N_1z}{(m_0q_2)^3},m_0q_2,x_2\right)}e\left(-\frac{mN_1z}{m_0^2q_1q_2}\right)dz.
\end{align*}
and \begin{align*}
    U_{\eta_1}\left(a,b,c\right)=\int_0^\infty U\left(w\right)\overline{U}_0\left(a,Nw\right)w^{-\frac{1}{3}}e\left(-\frac{cNw}{bQ}+\eta_13(aNw)^{1/3}\right)dw
\end{align*}
for $j=1,2$ with $U_0$ being a fixed flat function.

The oscillatory integral in $U_{\eta_1}$ has the phase function \begin{align*}
    f_1(w)=-\frac{Nwx_j}{m_0q_jQ}+\eta_13\left(\frac{N_1Nwz}{(m_0q_j)^3}\right)^\frac{1}{3}.
\end{align*}
Differentiating gives us \begin{align*}
    f_1'(w)=&-\frac{Nx_j}{m_0q_jQ}+\eta_1\left(\frac{N_1Nz}{(m_0q_j)^3w^2}\right)^\frac{1}{3}\\
    f_1''(w)=&-\eta_1\frac{2}{3}\left(\frac{N_1Nz}{(m_0q_j)^3w^5}\right)^\frac{1}{3}\\
    f_1^{(j)}(w)\ll&_j \frac{(N_1N)^\frac{1}{3}}{C}
\end{align*}
for any $j\geq2$.

For the case $X=0$, the support of $\varphi_0(x_j)$ gives us $\frac{Nx_j}{CQ}\ll t^\varepsilon$, and hence (1) in Lemma \ref{statlemma} gives us arbitrary savings unless \begin{align*}
    \left(\frac{N_1N}{C^3}\right)^\frac{1}{3}\ll t^\varepsilon \Longleftrightarrow N_1\ll \frac{C^3t^\varepsilon}{N}.
\end{align*}
In such a case, $U_{\eta_1}$ is $t^\varepsilon$-inert. Now by repeated integration by parts on the $z$-integral in $\mathcal{J}_0$, we get arbitrary savings unless \begin{align*}
    |m|\ll\frac{C^2t^\varepsilon}{N_1}.
\end{align*}
We have $\mathcal{J}_0(m,m_0,q_1,q_2,x_1,x_2)$ is $t^\varepsilon$-inert when $|m|\ll\frac{C^2t^\varepsilon}{N_1}$.

For the case $X\neq0$, the support of $\varphi_X(x_j)$ gives us $\frac{N|x_j|}{CQ}\sim\frac{N|X|}{CQ}\gg t^\varepsilon$, (1) in Lemma \ref{statlemma} gives us arbitrary savings unless \begin{align*}
    \eta_1X>0 \text{ and } \frac{(N_1N)^{1/3}}{C}\sim \frac{N|X|}{CQ}\gg t^\varepsilon.
\end{align*}
In such a case, we have $f_1''(w)\gg t^\varepsilon$. Computing the stationary phase point $f_1'(w_0)=0$, we get \begin{align*}
    w_0=\sqrt{\frac{N_1Q^3z}{N^2|x_j|^3}} \quad \text{ and } \quad f_1(w_0)=\frac{2Nx_jw_0}{m_0q_jQ}=2\sqrt{\frac{N_1Qz}{(m_0q_j)^2|x_j|}}.
\end{align*}
Hence for any $A>0$, (2) in Lemma \ref{statlemma} yields \begin{align*}
    U_{\eta_1}\left(\frac{N_1z}{(m_0q_j)^3},m_0q_j,x_j\right)=&\frac{1}{\sqrt{f_1''(w_0)}}\delta\left(\eta_1X>0\right)U_2\left(\frac{N_1Q^3z}{N^2|x_j|^3}\right)e\left(2\sqrt{\frac{N_1Qz}{(m_0q_j)^2|x_j|}}\right)+O\left(t^{-A}\right)
\end{align*}
for some $t^\varepsilon$-inert function $U_2\in C_c^\infty([1/4,25/4])$.

Writing \begin{align*}
    U_3\left(\frac{N_1Q^3z}{N^2|x_j|^3},m_0q_j,x_j\right)=\frac{(N_1N)^{1/6}}{\sqrt{C}}\frac{1}{\sqrt{f_1''(w_0)}}U_2\left(\frac{N_1Q^3z}{N^2|x_j|^3}\right),
\end{align*}
then $U_3$ is a $t^\varepsilon$-inert function supported on $[1/4,25/4]\times\R\times\R$ such that \begin{align*}
    U_{\eta_1}\left(\frac{N_1z}{(m_0q_j)^3},m_0q_j,x_j\right)=&\frac{\sqrt{C}}{(N_1N)^{1/6}}\delta\left(\eta_1X>0\right)U_3\left(\frac{N_1Q^3z}{N^2|x_j|^3},m_0q_j,x_j\right)\\
    &\times e\left(2\sqrt{\frac{N_1Qz}{(m_0q_j)^2|x_j|}}\right)+O\left(t^{-A}\right).
\end{align*}
Inserting this back into (\ref{J11Def}) yields \begin{align*}
    &\mathcal{J}_0(m,m_0,q_1,q_2,x_1,x_2)\\
    =&\frac{C}{(N_1N)^{1/3}}\delta\left(\eta_1X>0\right)\int_0^\infty \varphi(z)U_3 \left(\frac{N_1Q^3z}{N^2|x_1|^3},m_0q_1,x_1\right)\overline{U_3 \left(\frac{N_1Q^3z}{N^2|x_2|^3},m_0q_2,x_2\right)}\\
    &\times e\left(2\sqrt{\frac{N_1Qz}{(m_0q_1)^2|x_1|}}-2\sqrt{\frac{N_1Qz}{(m_0q_2)^2|x_2|}}-\frac{mN_1z}{m_0^2q_1q_2}\right)dz.
\end{align*}
Repeated integration by parts gives us arbitrary savings unless \begin{align*}
    \frac{|m|N_1}{m_0^2q_1q_2}\ll\sqrt{\frac{N_1Q}{C^2|X|}}\Longleftrightarrow |m|\ll\frac{C\sqrt{Q}}{\sqrt{N_1|X|}}t^\varepsilon\sim\frac{CQ^2}{NX^2}t^\varepsilon.
\end{align*}

\subsection{Stationary Phase Analysis of \texorpdfstring{$\mathcal{J}(q_1,q_2,\vec{v}\,)$}{J(q1,q2,v)}}\label{sect.JIntegralAnalysis}

Recall from (\ref{JDef}) that \begin{align*}
    \mathcal{J}(q_1,q_2,\vec{v}\,)=&\int_\BR \varphi_X(x_1)g(m_0q_1,x_1)\phi_{1}\left(u_1,m_0q_1,h_1,x_1\right)\nonumber\\
    &\times \int_\BR \varphi_X(x_2)\overline{g(m_0q_2,x_2)\phi_{1}\left(u_2,m_0q_2,h_2,x_2\right)}\nonumber\\
    &\times \int_0^\infty \varphi(z)U_3 \left(\frac{N_1Q^3z}{N^2|x_1|^3},m_0q_1,x_1\right)\overline{U_3 \left(\frac{N_1Q^3z}{N^2|x_2|^3},m_0q_2,x_2\right)}\nonumber\\
    &\times e\left(2\sqrt{\frac{N_1Qz}{m_0^2q_1^2|x_1|}}-2\sqrt{\frac{N_1Qz}{m_0^2q_2^2|x_2|}}-\frac{mN_1z}{m_0^2q_1q_2}\right)dz\nonumber\\
    &\times e\left(\frac{N^2x_1u_1}{m_0q_1QHt^{1-\varepsilon}}-\frac{N^2x_2u_2}{m_0q_2QHt^{1-\varepsilon}}+\frac{tx_1}{2\pi h_1Q}-\frac{tx_2}{2\pi h_2Q}\right)dx_2dx_1,
\end{align*}
and $U_3$ is some $t^\varepsilon$-inert function supported on $[1/4,25/4]\times\R\times\R$.

Consider the $x_1$-integral, \begin{align*}
    &\int_\BR\varphi_X(x_1)g(m_0q_1,x_1)\phi_{1}\left(u_1,m_0q_1,h_1,x_1\right)U_3 \left(\frac{N_1Q^3z}{N^2|x_1|^3},m_0q_1,x_1\right)\\
    &\times e\left(2\sqrt{\frac{N_1Qz}{m_0^2q_1^2|x_1|}}+\frac{N^2x_1u_1}{m_0q_1QHt^{1-\varepsilon}}+\frac{tx_1}{2\pi h_1Q}\right)dx_1\\
    =&X\int_\BR\varphi(x_1)g(m_0q_1,Xx_1)\phi_{1}\left(u_1,m_0q_1,h_1,Xx_1\right)U_3 \left(\frac{N_1Q^3z}{N^2X^3x_1^3},m_0q_1,Xx_1\right)e\left(F_1(x_1)\right)dx_1.
\end{align*}
The phase function is \begin{align*}
    F_1(x_1)=2\sqrt{\frac{N_1Qz}{m_0^2q_1^2|X|x_1}}+\frac{N^2Xx_1u_1}{m_0q_1QHt^{1-\varepsilon}}+\frac{tXx_1}{2\pi h_1Q}.
\end{align*}
Recall we have $N_1\sim\frac{N^2X^3}{Q^3}$ by the support of $U_3$. Differentiating, we get \begin{align*}
    F_1'(x_1)=&-\sqrt{\frac{N_1Qz}{m_0^2q_1^2|X|x_1^3}}+\frac{N^2Xu_1}{m_0q_1QHt^{1-\varepsilon}}+\frac{tX}{2\pi h_1Q}\\
    F_1''(x_1)=&\frac{3}{2}\sqrt{\frac{N_1Qz}{m_0^2q_1^2|X|x_1^5}}\sim \frac{\sqrt{N_1Q}}{C\sqrt{X}}\\
    F_1^{(j)}(x_1)\ll&_j \frac{\sqrt{N_1Q}}{C\sqrt{X}}
\end{align*}
for any $j\geq2$. The stationary point $x_{1,0}$ satisfying $F_1'(x_{1,0})=0$ is \begin{align*}
    x_{1,0}=\left(\frac{2\pi h_1}{m_0q_1t}\right)^{2/3}\frac{Q(N_1z)^{1/3}}{|X|}\left(1+\frac{2\pi h_1N^2u_1}{m_0q_1Ht^{2-\varepsilon}}\right)^{-2/3}
\end{align*}
and \begin{align*}
    F_1(x_{1,0})=3\sqrt{\frac{N_1Qz}{m_0^2q_1^2|X|x_{1,0}}}=3\left(\frac{N_1tz}{2\pi h_1(m_0q_1)^2}\left(1+\frac{2\pi h_1N^2u_1}{m_0q_1Ht^{2-\varepsilon}}\right)\right)^{1/3}.
\end{align*}
Since $X\gg\frac{CQt^\varepsilon}{N}$, we have $F_1''(x_1)\sim \frac{\sqrt{N_1Q}}{C\sqrt{|X|}}\sim\frac{N|X|}{CQ}\gg t^\varepsilon$. Together with (\ref{gDerProp}), (2) in Lemma \ref{statlemma} implies that there exists a $t^\varepsilon$-inert function $\tilde{\Phi}$ supported on $[1/4,25/4]\times\BR^2$ such that the $x_1$-integral is equal to
\begin{align*}
    &X\int_\BR\varphi(x_1)g(m_0q_1,Xx_1)\phi_{1}\left(u_1,m_0q_1,h_1,Xx_1\right)U_3 \left(\frac{N_1Q^3z}{N^2X^3x_1^3},m_0q_1,Xx_1\right)e\left(F_1(x_1)\right)dx_1\nonumber\\
    =&X\sqrt{\frac{C\sqrt{X}}{\sqrt{N_1Q}}}t^\varepsilon\tilde{\Phi}\left(\frac{N_1Q^3z}{N^2X^3},m_0q_1,u_1\right)g(m_0q_1,Xx_{1,0})e\left(3\left(\frac{N_1tz}{2\pi h_1(m_0q_1)^2}\left(1+\frac{2\pi h_1N^2u_1}{m_0q_1Ht^{2-\varepsilon}}\right)\right)^{1/3}\right)\nonumber\\
    &+O\left(t^{-A}\right)
\end{align*}
for any $A>0$. Applying the same treatment to the $x_2$-integral, there exists a $t^\varepsilon$-inert function $\Phi$ supported on $[1/4,25/4]\times\BR^4$ such that \begin{align*}
    \mathcal{J}(q_1,q_2,\vec{v}\,)=&\frac{CX^{5/2}t^\varepsilon}{\sqrt{N_1Q}}\int_0^\infty \varphi(z)\Phi \left(\frac{N_1Q^3z}{N^2X^3},m_0q_1,m_0q_2,u_1.u_2\right)g(m_0q_1,Xx_{1,0})g(m_0q_2,Xx_{2,0})\nonumber\\
    &\times e\left(3\left(\frac{N_1tz}{2\pi h_1(m_0q_1)^2}\left(1+\frac{2\pi h_1N^2u_1}{m_0q_1Ht^{2-\varepsilon}}\right)\right)^{1/3}\right)\nonumber\\
    &\times e\left(-3\left(\frac{N_1tz}{2\pi h_2(m_0q_2)^2}\left(1+\frac{2\pi h_2N^2u_2}{m_0q_2Ht^{2-\varepsilon}}\right)\right)^{1/3}-\frac{mN_1z}{m_0^2q_1q_2}\right)dz+O\left(t^{-A}\right)
\end{align*}
for any $A>0$, and \begin{align*}
    x_{j,0}=\left(\frac{2\pi h_j}{m_0q_jt}\right)^{2/3}\frac{Q(N_1z)^{1/3}}{X}\left(1+\frac{2\pi h_jN^2u}{m_0q_jHt^{2-\varepsilon}}\right)^{-2/3}
\end{align*}
for $j=1,2$. Note that $x_{j,0}$ is flat with respect to $z$.

\subsection{Stationary Phase Analysis of \texorpdfstring{$\mathcal{J}_j(q_1,q_2,\vec{v}\,)$}{Jj(q1,q2,v)}}\label{sect.JjIntegralAnalysis}

Let $nN_j\gg C^2t^\varepsilon$. Recall the definition of $\mathcal{J}_j(q_1,q_2,\vec{v}\,)$ in Lemma \ref{LemmaAfterFirstStep} that \begin{align*}
        &\mathcal{J}_j(q_1,q_2,\vec{v}\,)\\
        =&\int_0^\infty \varphi(z)\Phi_{j,\eta_j}\left(n_0q_1,n_0q_2,u_1,u_2\right)(g(n_0q_1,Xx_{1,0})g(n_0q_2,Xx_{2,0}))^{2-j}\nonumber\\
        &\times e\left(-\eta_j'3\left(\frac{N_jNz}{(n_0q_1)^3}\left(\frac{n_0q_1t}{2\pi h_1N}+\frac{Nu_1}{Ht^{1-\varepsilon}}\right)\right)^\frac{1}{3}+\eta_j'3\left(\frac{N_jNz}{(n_0q_2)^3}\left(\frac{n_0q_2t}{2\pi h_2N}+\frac{Nu_2}{Ht^{1-\varepsilon}}\right)\right)^\frac{1}{3}-\frac{nN_jz}{n_0^2q_1q_2}\right)dz.
    \end{align*}
    The phase function is \begin{align*}
    F_3(z):=-\eta_j'3\left(\frac{N_jNz}{(n_0q_1)^3}\left(\frac{n_0q_1t}{2\pi h_1N}+\frac{Nu_1}{Ht^{1-\varepsilon}}\right)\right)^\frac{1}{3}+\eta_j'3\left(\frac{N_jNz}{(n_0q_2)^3}\left(\frac{n_0q_2t}{2\pi h_2N}+\frac{Nu_2}{Ht^{1-\varepsilon}}\right)\right)^\frac{1}{3}-\frac{nN_jz}{n_0^2q_1q_2}.
\end{align*}
Differentiating, we get \begin{align*}
    F'_3(z)=&-\eta_j'\left(\frac{N_jN}{(n_0q_1)^3z^2}\left(\frac{n_0q_1t}{2\pi h_1N}+\frac{Nu_1}{Ht^{1-\varepsilon}}\right)\right)^\frac{1}{3}+\eta_j'\left(\frac{N_jN}{(n_0q_2)^3z^2}\left(\frac{n_0q_2t}{2\pi h_2N}+\frac{Nu_2}{Ht^{1-\varepsilon}}\right)\right)^\frac{1}{3}-\frac{nN_j}{n_0^2q_1q_2}\\
    F''_3(z)=&\eta_j'\frac{2}{3}\left(\frac{N_jN}{(n_0q_1)^3z^5}\left(\frac{n_0q_1t}{2\pi h_1N}+\frac{Nu_1}{Ht^{1-\varepsilon}}\right)\right)^\frac{1}{3}-\eta_j'\frac{2}{3}\left(\frac{N_jN}{(n_0q_2)^3z^5}\left(\frac{n_0q_2t}{2\pi h_2N}+\frac{Nu_2}{Ht^{1-\varepsilon}}\right)\right)^\frac{1}{3}\\
    F_3^{(j)}(z)\ll&_j (N_jN)^{1/3}\left(\frac{1}{n_0q_1}\left(\frac{n_0q_1t}{2\pi h_1N}+\frac{Nu_1}{Ht^{1-\varepsilon}}\right)^\frac{1}{3}-\frac{1}{n_0q_2}\left(\frac{n_0q_2t}{2\pi h_2N}+\frac{Nu_2}{Ht^{1-\varepsilon}}\right)^\frac{1}{3}\right)
\end{align*}
for any $j\geq2$. Computing the stationary phase point $F_3'(z_0(q_1,q_2))=0$, we get \begin{align*}
    z_0(n_0q_1,n_0q_2)=\frac{\sqrt{N}}{n^{3/2}N_j}\left(-\eta_j' n_0q_2\left(\frac{n_0q_1t}{2\pi h_1N}+\frac{Nu_1}{Ht^{1-\varepsilon}}\right)^{1/3}+\eta_j'n_0q_1\left(\frac{n_0q_2t}{2\pi h_2N}+\frac{Nu_2}{Ht^{1-\varepsilon}}\right)^{1/3}\right)^{3/2}
\end{align*}
and \begin{align*}
    &F_3(z_0(n_0q_1,n_0q_2))=\frac{2nN_jz_0}{n_0^2q_1q_2}\nonumber\\
    =&2\sqrt{\frac{N}{n}}\left(-\eta_j'\left(\frac{q_2t}{2\pi h_1q_1N}\left(1+\frac{2\pi h_1N^2u_1}{n_0q_1Ht^{2-\varepsilon}}\right)\right)^{1/3}+\eta_j'\left(\frac{q_1t}{2\pi q_2h_2N}\left(1+\frac{2\pi h_2N^2u_2}{n_0q_2Ht^{2-\varepsilon}}\right)\right)^{1/3}\right)^{3/2}.
\end{align*}
Using (\ref{gDerProp}) and $x_{j,0}$ is flat with respect to $z$ and the fact that $nN_j\gg C^2t^\varepsilon$, (1) in Lemma \ref{statlemma} gives us arbitrary savings unless \begin{align*}
    F_3''(z)\sim \frac{nN_j}{C^2}\gg t^\varepsilon.
\end{align*}
Hence (2) in Lemma \ref{statlemma} gives us 
\begin{align*}
    &\mathcal{J}_j(q_1,q_2,\vec{v}\,)\nonumber\\
    =&\frac{C}{\sqrt{nN_j}}(g(n_0q_1,Xx_1)g(n_0q_2,Xx_2))^{2-j} \Phi_2(z_0(n_0q_1,n_0q_2),q_1,q_2; \nu)e\left(F_3(z_0(n_0q_1,n_0q_2))\right)+O\left(t^{-A}\right),
\end{align*}
for any $A>0$, with some flat function $\Phi_2$ supported on $[1/2,5/2]\times\BR^7$, and 
\begin{align*}
    x_j:=x_j(z_0(n_0q_1,n_0q_2))=\left(\frac{2\pi h_j}{n_0q_jt}\right)^{2/3}\frac{Q(N_1z_0(n_0q_1,n_0q_2))^{1/3}}{X}\left(1+\frac{2\pi h_jN^2u}{n_0q_jHt^{2-\varepsilon}}\right)^{-2/3},
\end{align*}
for $j=1,2$.

\subsection{Stationary Phase Analysis in (\ref{S2After2ndPoisson})}\label{sect.yIntegralAnalysis}

Here we analyse the $y$-integral \begin{align*}
    \int_0^\infty\Upsilon(y)\Psi\left(Cy,q_2,q_3\right)e\left(F_3(z_0(Cy,q_2))-F_3(z_0(Cy,q_3))-\frac{q_1'Cy}{n_0nh_1h_2q_0^2q_{1,0}^2q_{2,0}r}\right)dy
\end{align*}
in (\ref{S2After2ndPoisson}), when $|q_1'|\ll Q_1Q_S$ with  
\begin{align*}
    Q_S=\frac{n_0q_0q_{2,0}rC^3t^{2+\varepsilon}}{n_1n(N_jN)^{4/3}}
\end{align*} as defined in (\ref{QsDef}). Recall that \begin{align*}
    G\left(y,z\right)=&G_{h_1,h_2}(y,z)=F_3(z_0(Cy,Cz))\nonumber\\
    =&\sqrt{\frac{2t}{\pi n}}\left(-\eta_j' \left(\frac{z}{h_1y}+\frac{2\pi N^2u_1z}{CHt^{2-\varepsilon}y^2}\right)^{1/3}+\eta_j'\left(\frac{y}{h_2z}+\frac{2\pi N^2u_2y}{CHt^{2-\varepsilon}z^2}\right)^{1/3}\right)^{3/2}.
\end{align*}
Lemma \ref{OGfDer} gives us \begin{align*}
    \frac{d}{dy}G\left(y,\frac{q}{C}\right)=&\eta_j'\sqrt{\frac{t}{2\pi n}}\left(-\eta_j' \left(\frac{q}{h_1Cy}+\frac{2\pi qN^2u_1}{C^2Ht^{2-\varepsilon}y^2}\right)^{1/3}+\eta_j'\left(\frac{Cy}{h_2q}+\frac{2\pi CN^2u_2y}{q^2Ht^{2-\varepsilon}}\right)^{1/3}\right)^{1/2}\\
    &\times \Bigg\{\left(\frac{q}{h_1Cy^4}+\frac{2\pi qN^2u_1}{C^2Ht^{2-\varepsilon}y^5}\right)^{1/3}+\left(\frac{C}{h_2qy^2}+\frac{2\pi CN^2u_2}{q^2Ht^{2-\varepsilon}y^2}\right)^{1/3}\\
    &+\frac{2\pi qN^2u_1}{C^2Ht^{2-\varepsilon}y^3}\left(\frac{q}{h_1Cy}+\frac{2\pi qN^2u_1}{C^2Ht^{2-\varepsilon}y^2}\right)^{-2/3}\Bigg\}.
\end{align*}
Using (\ref{z_0DiffBound}) to bound the derivatives of $\Psi$ (see (\ref{PsiDef}) for the definition of $\Psi$), applying (1) in Lemma \ref{statlemma} on the $y$-integral gives us arbitrary savings unless there exists $y'\in[1-10^{-9},1+10^{-9}]$ such that $z_0(Cy',q_j)\sim1$ for $j=1,2$ and \begin{align*}
    \frac{d}{dy}\Bigg|_{y=y'}\left(G\left(y,\frac{q_2}{C}\right)-G\left(y,\frac{q_3}{C}\right)\right)\ll \frac{C^2N^{2/3}t^\varepsilon}{n^2N_j^{4/3}}.
\end{align*}
Note that for any $y$ such that $z_0(Cy,q)\sim1$, together with the condition $n\gg N_S\gg \frac{CN^{4/3}}{N_j^{2/3}Ht^{1-\varepsilon}}$, we have \begin{align}\label{z0condition}
    &\left|-\eta_j' q\left(\frac{Cty}{2\pi h_1N}+\frac{Nu_1}{Ht^{1-\varepsilon}}\right)^{1/3}+\eta_j'Cy\left(\frac{qt}{2\pi h_2N}+\frac{Nu_2}{Ht^{1-\varepsilon}}\right)^{1/3}\right|\sim\frac{nN_j^{2/3}}{N^{1/3}}\nonumber\\
    \Longrightarrow & \left|\left(\frac{q}{h_1Cy}+\frac{2\pi qN^2u_1}{C^2Ht^{2-\varepsilon}y^2}\right)^{1/3}-\left(\frac{Cy}{h_2q}+\frac{2\pi CN^2u_2y}{q^2Ht^{2-\varepsilon}}\right)^{1/3}\right|\sim \frac{nN_j^{2/3}}{C^{4/3}t^{1/3}}.
\end{align}
Hence in the case where such $y'$ exists, the bound for the derivative above implies \begin{align*}
    &\left|\left(\left(\frac{q_2}{h_1Cy'}+\frac{2\pi q_2N^2u_1}{C^2Ht^{2-\varepsilon}y'^2}\right)^{1/3}-\left(\frac{Cy'}{h_2q_2}+\frac{2\pi CN^2u_2y'}{q_2^2Ht^{2-\varepsilon}}\right)^{1/3}\right)\left(\left(\frac{q_2}{h_1Cy'}+\frac{2\pi q_2N^2u_1}{C^2Ht^{2-\varepsilon}y'^2}\right)^{1/3}\right.\right.\nonumber\\
    &\left.+\left(\frac{Cy'}{h_2q_2}+\frac{2\pi CN^2u_2y'}{q_2^2Ht^{2-\varepsilon}}\right)^{1/3}+\frac{2\pi q_2N^2u_1}{C^2Ht^{2-\varepsilon}y'^2}\left(\frac{q_2}{h_1Cy'}+\frac{2\pi q_2N^2u_1}{C^2Ht^{2-\varepsilon}y'^2}\right)^{-2/3}\right)^2\\
    &\left.-\left(\left(\frac{q_3}{h_1Cy'}+\frac{2\pi q_3N^2u_1}{C^2Ht^{2-\varepsilon}y'^2}\right)^{1/3}-\left(\frac{Cy'}{h_2q_3}+\frac{2\pi CN^2u_2y'}{q_3^2Ht^{2-\varepsilon}}\right)^{1/3}\right)\left(\left(\frac{q_3}{h_1Cy'}+\frac{2\pi q_3N^2u_1}{C^2Ht^{2-\varepsilon}y'^2}\right)^{1/3}\right.\right.\nonumber\\
    &\left.\left.+\left(\frac{Cy'}{h_2q_3}+\frac{2\pi CN^2u_2y'}{q_3^2Ht^{2-\varepsilon}}\right)^{1/3}+\frac{2\pi q_3N^2u_1}{C^2Ht^{2-\varepsilon}y'^2}\left(\frac{q_3}{h_1Cy'}+\frac{2\pi q_3N^2u_1}{C^2Ht^{2-\varepsilon}y'^2}\right)^{-2/3}\right)^2\right|\\
    &\ \ll \frac{C^2N^{2/3}}{n^2N_j^{4/3}}\frac{n(N_jN)^{1/3}}{Ct^{1-\varepsilon}}=\frac{CN}{nN_jt^{1-\varepsilon}}.
\end{align*}
To simplify this condition, observe that the expression inside the absolute value can be expressed as $\mathcal{T}_2-\mathcal{T}_3$, where for some suitable $A, B, C, D$, \begin{align*}
    \mathcal{T}_j:=&\left(q_j^{1/3}A^{1/3}-q_j^{-1/3}\left(B+\frac{C}{q_j}\right)^{1/3}\right)\left(q_j^{1/3}A^{1/3}+q_j^{-1/3}\left(B+\frac{C}{q_j}\right)^{1/3}+q_j^{1/3}D\right)^2\\
    =&q_j(A+2A^{2/3}D+A^{1/3}D^2)+q_j^{1/3}(A^{2/3}-D^2)\left(B+\frac{C}{q_j}\right)^{1/3}\\
    &-\frac{1}{q_j^{1/3}}(A^{1/3}+2D)\left(B+\frac{C}{q_j}\right)^{2/3}-\frac{1}{q_j}\left(B+\frac{C}{q_j}\right).
\end{align*}
Applying the Mean Value theorem, we can factor out $q_2-q_3$ from each pair of the difference from each term in $\mathcal{T}_2-\mathcal{T}_3$, and simplifying gives \begin{align}\label{q2q30close}
    |q_2-q_3|\ll \frac{C^3t^\varepsilon}{nN_j}.
\end{align}

Moreover, (\ref{z0condition}) gives us \begin{align*}
    \left|y^{2/3}-\left(\frac{h_2q_2^2}{h_1C^2}\right)^{1/3}\right|\sim\frac{nN_j^{2/3}}{CN^{1/3}}+O\left(\frac{N}{Ht^{1-\varepsilon}}\right)\Longrightarrow\left|y-\frac{\sqrt{h_2}q_2}{\sqrt{h_1}C}\right|\sim \frac{nN_j^{2/3}}{CN^{1/3}}.
\end{align*}
Here we used $|n|\gg\frac{CN^{4/3}}{N_j^{2/3}Ht^{1-\varepsilon}}$. Hence the $y$-integral is bounded by \begin{align*}
    \ll \frac{nN_j^{2/3}t^\varepsilon}{CN^{1/3}}\delta\left(|q_2-q_3|\ll \frac{C^3t^\varepsilon}{nN_j}\right)+t^{-A}
\end{align*}
for any $A>0$.

\subsection{Integral Analysis involving \texorpdfstring{$\mathrm{GL}(3)$}{GL(3)} Bessel Functions}

\begin{Lemma}\label{OGfDer}
    Let $\eta=\pm1$ and let $h_1,h_2,A_1,A_2,x,y$ be such that $\eta\left(\frac{x}{h_2y}\right)^{1/3}\geq \eta\left(\frac{y}{h_1x}\right)^{1/3}$. Let $g$ be \begin{align*}
        g\left(x,y,h_1,h_2,A_1,A_2\right)=&\left(-\eta \left(\frac{y}{h_1x}+\frac{A_1y}{x^2}\right)^{1/3}+\eta\left(\frac{x}{h_2y}+\frac{A_2x}{y^2}\right)^{1/3}\right)^{3/2},
    \end{align*}
    we have \begin{align*}
        g(x,y,h_1,h_2,A_1,A_2)=(-1)^{3/2}g(y,x,h_2,h_1,A_2,A_1).
    \end{align*}
    Differentiating yields \begin{align*}
        &g_x(x,y,h_1,h_2,A_1,A_2)=(-1)^{3/2}g_y(y,x,h_2,h_1,A_2,A_1)\\
        =&\frac{\eta}{2x}\left(-\eta \left(\frac{y}{h_1x}+\frac{A_1y}{x^2}\right)^{1/3}+\eta\left(\frac{x}{h_2y}+\frac{A_2x}{y^2}\right)^{1/3}\right)^{1/2}\\
        &\times \left( \left(\frac{y}{h_1x}+\frac{A_1y}{x^2}\right)^{1/3}+\left(\frac{x}{h_2y}+\frac{A_2x}{y^2}\right)^{1/3}+\frac{A_1y}{x^2}\left(\frac{y}{h_1x}+\frac{A_1y}{x^2}\right)^{-2/3}\right).
    \end{align*}
\end{Lemma}

\begin{Lemma}\label{fDer}
    Let $\eta=\pm1$ and let $h_1,h_2,x,y$ be such that $\eta\left(\frac{x}{h_2y}\right)^{1/3}\geq \eta\left(\frac{y}{h_1x}\right)^{1/3}$. Let $f$ be \begin{align*}
        f\left(x,y,h_1,h_2\right)=&\left(-\eta \left(\frac{y}{h_1x}\right)^{1/3}+\eta\left(\frac{x}{h_2y}\right)^{1/3}\right)^{3/2},
    \end{align*}
    we have \begin{align*}
        f(x,y,h_1,h_2)=(-1)^{3/2}f(y,x,h_2,h_1).
    \end{align*}
    Differentiating yields \begin{align*}
        f_x(x,y,h_1,h_2)=&(-1)^{3/2}f_y(y,x,h_2,h_1)\\
        =&\frac{\eta}{2}\left(-\eta \left(\frac{y}{h_1x}\right)^{1/3}+\eta\left(\frac{x}{h_2y}\right)^{1/3}\right)^{1/2}\left( \left(\frac{y}{h_1x^4}\right)^{1/3}+\frac{1}{(h_2x^2y)^{1/3}}\right)
    \end{align*}
    \begin{align*}
        f_{xx}(x,y,h_1,h_2)=&(-1)^{3/2}f_{yy}(y,x,h_2,h_1)\\
        =&\frac{1}{12}\left(-\eta\left(\frac{y}{h_1x}\right)^{1/3}+\eta\left(\frac{x}{h_2y}\right)^{1/3}\right)^{-1/2}\left(\frac{9y^{2/3}}{h_1^{2/3}x^{8/3}}-\frac{2}{(h_1h_2)^{1/3}x^2}-\frac{3}{h_2^{2/3}x^{4/3}y^{2/3}}\right),
    \end{align*}
    \begin{align*}
        f_{xy}(x,y,h_1,h_2)=-\frac{1}{12}\left(-\eta\left(\frac{y}{h_1x}\right)^{1/3}+\eta\left(\frac{x}{h_2y}\right)^{1/3}\right)^{-1/2}\left(\frac{3}{(h_1^2x^5y)^{1/3}}+\frac{3}{(h_2^2xy^5)^{1/3}}-\frac{2}{(h_1h_2)^{1/3}xy}\right)
    \end{align*}
    and \begin{align*}
        f_{xxy}(x,y,h_1,h_2)=&\frac{\eta}{72}\left(-\eta\left(\frac{y}{h_1x}\right)^{1/3}+\eta\left(\frac{x}{h_2y}\right)^{1/3}\right)^{-3/2}\\
        &\times\left(\frac{9}{h_2xy^2}+\frac{43}{h_1^{2/3}h_2^{1/3}x^{7/3}y^{2/3}}-\frac{17}{h_1^{1/3}h_2^{2/3}x^{5/3}y^{4/3}}-\frac{27}{h_1x^3}\right).
    \end{align*}
\end{Lemma}

Define \begin{align*}
    G(x,y)=G_{h_1,h_2}(x,y)=\sqrt{\frac{2t}{\pi n}}f(x,y,h_1,h_2)
\end{align*}
as in the main sections. Then we have the following integral analysis lemmas.

\begin{Lemma}\label{S2IntegralLemma}
    Let $1\leq a,b,\frac{h_1}{H'}, \frac{h_2}{H'}<1+10^{-10}$. Let $\Upsilon\in C_c^\infty([1-10^{-9},1+10^{-9}])$ be fixed and let $W$ be a $X$-inert function such that $W(y,a,b)$ contains the restriction \begin{align}\label{abRestriction}
        \left|\left(\frac{a}{h_1y}\right)^{1/3}-\left(\frac{y}{h_2a}\right)^{1/3}\right|\sim \left|\left(\frac{b}{h_1y}\right)^{1/3}-\left(\frac{y}{h_2b}\right)^{1/3}\right|\sim \frac{nN_j^{2/3}}{C^{4/3}t^{1/3}}.
    \end{align}
    Let $Z>0$ and let $|q'|\gg \frac{X^2}{Z}t^\varepsilon$. Let \begin{align*}
        J=\int_0^\infty\Upsilon(y)W(y,a,b)e\left(G(y,a)-G(y,b)-q'Zy\right)dy
    \end{align*}
    and let $y_0$ be such that \begin{align*}
        G_x(y_0,a)-G_x(y_0,b)=q'Z.
    \end{align*}
    Let $U\gg \sqrt{\frac{nN_j^{2/3}}{CN^{1/3}q'Z}}t^\varepsilon$ and let $W_0\in C_c^\infty([-2,2])$ such that $W_0(x)=1$ for $|x|\leq1$, then we have \begin{align*}
        J=&\int_0^\infty\Upsilon(y)W(y,a,b)W_0\left(\frac{y-y_0}{U}\right)e\left(G(y,a)-G(y,b)-q'Zy\right)dy+O\left(t^{-A}\right)
    \end{align*}
    for any $A>0$. Moreover, we have \begin{align*}
        \frac{d^jy_0}{da^j}, \frac{d^jy_0}{db^j}\ll_j \left(\frac{nN_j^{2/3}}{CN^{1/3}}+\frac{(N_jN)^{1/3}}{Cq'Z}\right)^j.
    \end{align*}
    for any $j\geq1$.
\end{Lemma}
\begin{proof}
    Since $W$ is $X$-inert and $|q'|\gg \frac{X^2}{Z}t^\varepsilon$, (1) in Lemma \ref{statlemma} gives us arbitrary savings unless there exists $1\leq y'<1+10^{-10}$ such that \begin{align*}
        \frac{d}{dy}\Bigg|_{y=y'}\left(G(y,a)-G(y,b)\right)=G_x(y',a)-G_x(y',b)\sim q'Z.
    \end{align*}
    Applying the same analysis that concludes (\ref{y_1y_20close}) yields \begin{align}\label{abclose}
        |a-b|\sim\frac{nN_j^{1/3}}{N^{2/3}} q'Z.
    \end{align}
    Differentiating once more with Lemma \ref{fDer} yields \begin{align*}
        \frac{d^2}{dy^2}\left(G(y,a)-G(y,b)\right)=\sqrt{\frac{t}{72\pi n}}(F(a)-F(b)),
    \end{align*}
    where \begin{align*}
        F(x)=G_{xx}(y,x)=\left(-\eta\left(\frac{x}{h_1y}\right)^{1/3}+\eta\left(\frac{y}{h_2x}\right)^{1/3}\right)^{-1/2}\left(\frac{9x^{2/3}}{h_1^{2/3}y^{8/3}}-\frac{2}{(h_1h_2)^{1/3}y^2}-\frac{3}{h_2^{2/3}x^{2/3}y^{4/3}}\right).
    \end{align*}
    We want to compute the bounds for the derivatives of $G$ in order to apply stationary phase analysis in Lemma \ref{statlemma}. Notice that \begin{align*}
    F'(x)=&\frac{\eta}{6}\left(-\eta\left(\frac{x}{h_1y}\right)^{1/3}+\eta\left(\frac{y}{h_2x}\right)^{1/3}\right)^{-3/2}\\
    &\times \left(\frac{43}{(h_1^2h_2x^2y^7)^{1/3}}+\frac{9}{h_2x^2y}-\frac{27}{h_1y^3}-\frac{17}{(h_1h_2^2x^4y^5)^{1/3}}\right).
\end{align*}
Moreover, $\left(\frac{x}{h_1y}\right)^{1/3}-\left(\frac{y}{h_2x}\right)^{1/3}$ is increasing w.r.t. $x$ for $x,y>0$ together with (\ref{abRestriction}) gives us \begin{align*}
    \min_{x=a,b}\left\{\left|\left(\frac{x}{h_1y}\right)^{1/3}-\left(\frac{y}{h_2x}\right)^{1/3}\right|\right\}\leq\left|\left(\frac{x}{h_1y}\right)^{1/3}-\left(\frac{y}{h_2x}\right)^{1/3}\right|\leq \max_{x=a,b}\left\{\left|\left(\frac{x}{h_1y}\right)^{1/3}-\left(\frac{y}{h_2x}\right)^{1/3}\right|\right\}
\end{align*}
\begin{align}\label{temp2}
    \Longrightarrow \left|\left(\frac{x}{h_1y}\right)^{1/3}-\left(\frac{y}{h_2x}\right)^{1/3}\right|\sim \frac{nN_j^{2/3}}{C^{4/3}t^{1/3}}
\end{align}
for any $x$ in between $a,b$. Together with $1\leq a,b,\frac{h_1}{H'}, \frac{h_2}{H'}<1+10^{-10}, 1-10^{-9}\leq y\leq 1+10^{-9}$ by the assumption and the support of $\Upsilon$, we get \begin{align*}
    F'(x)\gg \left(\frac{nN_j^{2/3}}{C^{4/3}t^{1/3}}\right)^{-3/2}\left(\frac{Ct}{N}\right)^{-1}=\frac{CN}{n^{3/2}N_j\sqrt{t}}
\end{align*}
for $x$ in between $a,b$. Hence the mean value theorem yields \begin{align*}
    \frac{d^2}{dy^2}\left(G(y,a)-G(y,b)\right)=&\sqrt{\frac{t}{72\pi n}}\left(F(a)-F(b)\right)\gg\min_{\min\{a,b\}\leq x\leq \max\{a,b\}}\sqrt{\frac{t}{72\pi n}}|F'(x)||a-b|\\
    \gg& \frac{CN}{n^2N_j}|a-b|\sim \frac{CN^{1/3}}{nN_j^{2/3}}q'Z.
\end{align*}
With similar analysis we obtain \begin{align*}
    \frac{d^j}{dy^j}\left(G(y,a)-G(y,b)\right)\ll_j\left(\frac{CN^{1/3}}{nN_j^{2/3}}\right)^{j-1}q'Z
\end{align*}
for any $j\geq2$. Let $y_0$ be as in the statement, i.e. \begin{align}\label{temp}
    \frac{d}{dy}\Bigg|_{y=y_0}\left(G(y,a)-G(y,b)\right)=G_x(y_0,a)-G_x(y_0,b)=q'Z.
\end{align}
Then (3) in Lemma \ref{statlemma} yields the first statement.

For the second statement about the derivatives of $y_0$, notice that differentiating (\ref{temp}) yields, \begin{align*}
    \frac{dy_0}{da}=&\frac{G_{xy}(y_0,a)}{G_{xx}(y_0,b)-G_{xx}(y_0,a)}.
\end{align*}
Mean value theorem then implies that there exists some $x'$ in between $a$ and $b$ such that \begin{align*}
    G_{xx}(y_0,b)-G_{xx}(y_0,a)=G_{xxy}(y_0,x')(b-a),
\end{align*}
with $G_{xxy}(y_0,x')$ computed with Lemma \ref{fDer} as $F$ above. Then the restriction (\ref{abclose}) and (\ref{temp2}) together with $1\leq x',a,b,\frac{h_1}{H'},\frac{h_2}{H'}<1+10^{-10}, 1-10^{-9}\leq y_0\leq 1+10^{-9}$ gives us \begin{align*}
    \frac{dy_0}{da}=&\frac{G_{xy}(y_0,a)}{(b-a)G_{xxy}(y_0,x')}\\
    =&\frac{6\eta}{a-b}\left(-\eta\left(\frac{a}{h_1y_0}\right)^{1/3}+\eta\left(\frac{y_0}{h_2a}\right)^{1/3}\right)^{-1/2}\left(\frac{3}{(h_1^2ay_0^5)^{1/3}}+\frac{3}{(h_2^2a^5y_0)^{1/3}}-\frac{2}{(h_1h_2)^{1/3}ay_0}\right)\\
    &\times \left(-\eta\left(\frac{x'}{h_1y_0}\right)^{1/3}+\eta\left(\frac{y_0}{h_2x'}\right)^{1/3}\right)^{3/2}\left(\frac{43}{(h_1^2h_2x'^2y^7)^{1/3}}+\frac{9}{h_2x'^2y_0}-\frac{27}{h_1y_0^3}-\frac{17}{(h_1h_2^2x'^4y_0^5)^{1/3}}\right)^{-1}\\
    \sim &\frac{N^{2/3}}{nN_j^{1/3}q'Z} \frac{nN_j^{2/3}}{C^{4/3}t^{1/3}}\left(\frac{Ct}{N}\right)^{1/3}=\frac{(N_jN)^{1/3}}{Cq'Z}.
\end{align*}
Performing similar analysis yields \begin{align*}
    \frac{d^jy_0}{da^j}, \frac{d^jy_0}{db^j}\ll_j \left(\frac{nN_j^{2/3}}{CN^{1/3}}\right)^j\left(1+\frac{N^{2/3}}{nN_j^{1/3}q'Z}\right)^j=\left(\frac{nN_j^{2/3}}{CN^{1/3}}+\frac{(N_jN)^{1/3}}{Cq'Z}\right)^j.
\end{align*}
\end{proof}

\section{Infinite Cauchy-Schwarz}

\subsection{Character Sum Analysis}

Let $a,b_1,b_2,b_3,c,q_1,q_2$ be integers such that $(q_1q_2,c)=1$ and $c\nmid b_2, b_3$ with $c$ being square-free. Let \begin{align*}
    \mathcal{C}(a,b_1,b_2,b_3,c,q_1,q_2)=\sumstar_{\gamma\bmod c}e\left(\frac{a\gamma}{c}\right)S(b_1-\overline{q_1}b_2+\overline{\gamma}b_2,b_3(\gamma-q_1);c)S(b_1-\overline{q_2}b_2+\overline{\gamma}b_2,b_3(\gamma-q_2);c).
\end{align*}

\begin{Lemma}\label{CharSumLemmaFirstCase}
    If $q_1\equiv q_2\bmod c$ or $b_1\pm b_2(\overline{q_1}-\overline{q_2})\equiv 0 \bmod c$ or $b_1-\overline{q_1}b_2\equiv 0\bmod c$ or $b_1-\overline{q_2}b_2\equiv 0\bmod c$, then we have \begin{align*}
        \mathcal{C}(a,b_1,b_2,b_3,c,q_1,q_2)\ll c^{2+\varepsilon}.
    \end{align*}
    Otherwise, we have \begin{align*}
        \mathcal{C}(a,b_1,b_2,b_3,c,q_1,q_2)\ll c^{3/2+\varepsilon}.
    \end{align*}
\end{Lemma}
\begin{proof}
    Write $d_j(\gamma)=gcd(b_1-\overline{q_j}b_2+\overline{\gamma}b_2,b_3(\gamma-q_j),c)$ for $j=3,4$. For the first statement, applying Weil bound on the two Kloosterman sums yields \begin{align*}
        \mathcal{C}(a,b_1,b_2,b_3,c,q_1,q_2)\ll c^{1+\varepsilon}\sumstar_{\gamma\bmod c}\sqrt{(d_3(\gamma)d_4(\gamma))}\ll c^{2+\varepsilon}.
    \end{align*}
    
    For the second statement, we apply the method in \cite{Adolphson-Sperber}. Consider the Newton polyhedron $\Delta(f)$ of \begin{align*}
        f(x,y,z)=az+x(b_1-\overline{q_1}b_2)+xz^{-1}b_2+x^{-1}zb_3-x^{-1}b_3q_1+y(b_1-\overline{q_2}b_2)+yz^{-1}b_2+y^{-1}zb_3-y^{-1}b_3q_2.
    \end{align*}
    We have to separate into two cases.
    \begin{enumerate}
        \item $a\equiv 0\bmod c$, other coefficients of $f$ are not $0\bmod c$
        \item No coefficients of $f$ is $0\bmod c$.
    \end{enumerate}
    The Newton polyhedron $\Delta(f)$ for case (2) is illustrated in Figure \ref{fig1}. One can check that the locus of $\partial f_{\sigma} / \partial x=\partial f_{\sigma} / \partial y=\partial f_{\sigma} / \partial z=0$ is empty in $\left(\left(\mathbb{Z}/c\mathbb{Z}\right)^{\times}\right)^{3}$ for any face $\sigma$ in the Newton polyhedron $\Delta(f)$ that corresponds to any of the three cases unless $q_1\equiv q_2\bmod c$ or $b_1\pm b_2(\overline{q_1}-\overline{q_2})\equiv 0 \bmod c$.
\end{proof}

In general, we can extend the above lemma to the following setting. Let $\nu\geq0$ be an integer and let $I_\nu\subset \mathcal{P}(\{1,2,...,2^\nu\})$ be the set defined by \begin{align*}
    I_\nu=&\left\{\{2j+1,2j+2\}: 0\leq j\leq 2^{\nu-1}-1\right\}\\
    &\bigcup \left\{\{2^k(1+4j), 2^k(3+4j):j, k\geq 0, 2^k(3+4j)\leq 2^\nu\right\}\backslash\left\{\{2^{\nu-1},2^\nu\}\right\}.
\end{align*}
Let $J_\nu=\{(j,S):1\leq j\leq 2^{\nu}, S\in I_\nu \text{ and } j\in S\}$. Let $a_1,...,a_{2^{\nu}},a_S$ $(S\in I_\nu)$, $b_1, b_2, b_3, c, q_1, q_2$ be integers. Write $a_{I_\nu}=\{a_1,...,a_{2^{\nu}},a_S:S\in I_\nu\}$ and define \begin{align*}
    &\mathcal{C}_{\nu}(a_{I_\nu}, b_1, b_2, b_3, c, q_1, q_2)\\
    =&\mathop{\sumx\cdots\sumx}_{\{\alpha_{j,S}: (j,S)\in J_\nu\}, \beta_1,\beta_2,\gamma_1,...,\gamma_{2^{\nu}},\{\gamma_S: S\in I_\nu\}\bmod c}e\left(\frac{f_\nu(\alpha_{J_\nu},\beta_1,\beta_2,\gamma_{I_\nu})}{c}\right)
\end{align*}
with \begin{align*}
    f_\nu(\alpha_{J_\nu},\beta_1,\beta_2,\gamma_{I_\nu})=&\sum_{j=1}^{2^{\nu}}a_j\gamma_j +\sum_{S\in I_\nu}a_S\gamma_S+\sum_{(j,S)\in J_\nu}\left(\alpha_{j,S}(b_1-\overline{\gamma_S}b_2+\overline{\gamma_j}b_2)+\overline{\alpha_{j,S}}(\gamma_j-\gamma_S)b_3)\right)\\
    &+\beta_1(b_1-\overline{q_1}b_2+\overline{\gamma_{2^{\nu-1}}}b_2)+\overline{\beta_1}(\gamma_{2^{\nu-1}}-q_1)b_3)\\
    &+\beta_2(b_1-\overline{q_2}b_2+\overline{\gamma_{2^{\nu}}}b_2)+\overline{\beta_2}(\gamma_{2^{\nu}}-q_2)b_3),
\end{align*}
and the notation $\alpha_{J_\nu}=\{\alpha_{j,S}:(j,S)\in J_\nu\}$ and $\gamma_{I_\nu} = \{\gamma_1,..,\gamma_{2^{\nu}},\gamma_S:S\in I_\nu\}$.

\begin{Lemma}\label{CharSumLemmaGeneral}
    Notations as above. Let $a_1,...,a_{2^{\nu}},a_S$ $(S\in I_\nu)$, $b_1, b_2, b_3, c, q_1, q_2$ be integers such that $(q_1q_2,c)=1$ and $c\nmid b_2,b_3$. If $q_1\equiv q_2\bmod c$ or $b_1\pm b_2(\overline{q_1}-\overline{q_2})\equiv 0 \bmod c$ or $b_1-\overline{q_1}b_2\equiv 0\bmod c$ or $b_1-\overline{q_2}b_2\equiv 0\bmod c$, then we have \begin{align*}
        \mathcal{C}_\nu(a_{I_\nu}, b_1, b_2, b_3, c, q_1, q_2)\ll c^{2^{\nu+1}+\varepsilon}.
    \end{align*}
    Otherwise, we have \begin{align*}
        \mathcal{C}_\nu(a_{I_\nu}, b_1, b_2, b_3, c, q_1, q_2)\ll c^{2^{\nu+1}-1/2+\varepsilon}.
    \end{align*}
\end{Lemma}
\begin{proof}
    We apply the same method as the proof of Lemma \ref{CharSumLemmaFirstCase}.
    
    For the second statement, we apply the method in \cite{Adolphson-Sperber}. Studying all the cases whether each $a_j,a_S$ for $1\leq j\leq 2^\nu, S\in I_\nu$ can be $0\bmod c$ or not, one can check that the locus of $\partial {f_\nu}_{\sigma} / \partial x_j=0$ for all $1\leq j\leq 2^{\nu+2}-1$ is empty in $\left(\left(\mathbb{Z}/c\mathbb{Z}\right)^{\times}\right)^{2^{\nu+2}-1}$ for any face $\sigma$ in the Newton polyhedron $\Delta(f_\nu)$ that corresponds to any of the three cases unless $q_1\equiv q_2\bmod c$ or $b_1\pm b_2(\overline{q_1}-\overline{q_2})\equiv 0 \bmod c$.
    
    For the first statement, we prove them by induction. Note that $\nu=0$ is precisely the statement of Lemma \ref{CharSumLemmaFirstCase}. Notice that \begin{align*}
        \mathcal{C}_\nu(a_{I_\nu}, b_1, b_2, b_3, c, q_1, q_2)=&\sumx_{\gamma_{2^{\nu-2},3(2^{\nu-2})}}e\left(\frac{a_{\{2^{\nu-2},3(2^{\nu-2})\}}\gamma_{\{2^{\nu-2},3(2^{\nu-2})\}}}{c}\right)\\
        &\times \mathcal{C}_{\nu-1}(a_{I_{\nu-1}}, b_1, b_2, b_3, c, \gamma_{2^{\nu-2},3(2^{\nu-2})}, q_1)\\
        &\times \mathcal{C}_{\nu-1}(a_{I_{\nu-1}}', b_1, b_2, b_3, c, \gamma_{2^{\nu-2},3(2^{\nu-2})}, q_2),
    \end{align*}
    with appropriate choices of $a_{I_{\nu-1}}$ and $a_{I_{\nu-1}}'$ given $a_{I_\nu}$. Hence we have \begin{align*}
        \mathcal{C}_\nu(a_{I_\nu}, b_1, b_2, b_3, c, q_1, q_2)\ll&\sumx_{\gamma_{2^{\nu-2},3(2^{\nu-2})}\bmod c}\left|\mathcal{C}_{\nu-1}(a_{I_{\nu-1}}, b_1, b_2, b_3, c, \gamma_{2^{\nu-2},3(2^{\nu-2})}, q_1)\right|\\
        &\times \left|\mathcal{C}_{\nu-1}(a_{I_{\nu-1}}', b_1, b_2, b_3, c, \gamma_{2^{\nu-2},3(2^{\nu-2})}, q_2)\right|.
    \end{align*}
    For $j=1,2$, write $(*_j)$ as the condition $\Big\{\gamma_{2^{\nu-2},3(2^{\nu-2})}\equiv q_j\bmod c$ or $b_1\pm b_2(\overline{\gamma_{2^{\nu-2},3(2^{\nu-2})}}-\overline{q_j})\equiv 0 \bmod c$ or $b_1-\overline{\gamma_{2^{\nu-2},3(2^{\nu-2})}}b_2\equiv 0\bmod c$ or $b_1-\overline{q_j}b_2\equiv 0\bmod c\Big\}$.
    By induction hypothesis, we have \begin{align*}
        &\left|\mathcal{C}_{\nu-1}(a_{I_{\nu-1}}, b_1, b_2, b_3, c, \gamma_{2^{\nu-2},3(2^{\nu-2})}, q_j)\right|\\
        \ll & c^{2^\nu-1/2+\varepsilon}\left(\sqrt{c}\delta(*_j)+1\right).
    \end{align*}
    Hence we have \begin{align*}
        \mathcal{C}_\nu(a_{I_\nu}, b_1, b_2, b_3, c, q_1, q_2)\ll\sumx_{\gamma_{2^{\nu-2},3(2^{\nu-2})}\bmod c}c^{2^{\nu+1}-1+\varepsilon}\prod_{j=1,2}\left(\sqrt{c}\delta(*_j)+1\right),
    \end{align*}
    which yields the desired result.
\end{proof}

\subsection{Proof of Lemma \ref{IteraionLemma}}\label{sect.IterationLemmaProof}

We prove this by induction. Note that $\mathcal{B}_2$ coincides with the definition in (\ref{B2Def}) and the statement is proved in the previous subsection for $\nu=2$. Assume the statement is true for $\nu=k$, then we have \begin{align*}
    S_{2}(N)^*\ll |\mathcal{A}_k|+\sum_{j=1}^k\mathcal{B}_j,
\end{align*}
with $\mathcal{A}_k$, $\mathcal{B}_1$, ..., $\mathcal{B}_k$ defined as in Lemma \ref{IteraionLemma}.
Notice that by the same counting argument as in \eqref{DetailedCountingComputation},
\begin{align*}
    &\sup_{\|\beta\|_2=1}\sum_{q_{k+1}\ll C}\mathop{\sum\sum\sum\sum}_{\substack{n_0q_0q_{2,0}q_{k+1}'r=q_{k+1}\\q_{2,0}|q_0^\infty}}\sum_{\substack{q_{1,0}\ll C\\q_{1,0}|q_0^\infty}}\sum_{q_0|n_1|q_0^\infty}\mathop{\sum\sum\sum}_{\substack{n_2h_{1,0}h_{2,0} \square\text{-full}\\ h_{2,0}|(n_2h_{1,0})^\infty}}\sum_{n'\ll N_B}\mathop{\sum\sum}_{h_1',h_2'\ll H'}\sum_{0<|q_1'|\ll Q_1Q_B}\mathop{\sum\cdots\sum}_{0<|q_2'|,...,|q_k'|\ll Q_B}\\
    &\times \frac{1}{n_1n_2^{3/2}n'h_{1,0}^{3/2}h_{2,0}h_1'q_0q_{1,0}q_{2,0}|q_1'\cdots q_k'|}\int_{|u_2|\ll1}|\beta(q_{k+1},h_2,u_2)|^2du_2\sum_{\gamma\bmod{n_1n_2h_{1,0}h_{2,0}q_0q_{1,0}q_{2,0}}}\\
    &\times \left|D(k)\sumstar_{\beta_1\bmod{q_0q_{1,0}}}\sumstar_{\beta_2\bmod{q_0q_{2,0}}}\delta\left(-\eta_j\frac{n}{q_0}\equiv \beta_1h_1q_{2,0}\gamma q_{k+1}'+\beta_2h_2q_{1,0}\gamma q_{k+1}'\bmod{q_0q_{1,0}q_{2,0}}\right)\right|^2\ll t^\varepsilon.
\end{align*}
Recall the definition of $\mathcal{D}_0$ in (\ref{D0Def}), applying Cauchy-Schwarz inequality to take out $n,n_0,q_0,q_{1,0},q_{2,0},q_1',q_2',h_1,h_2$-sums, $\gamma$-sum in $\mathcal{D}_0$ and the $u_2$-integral, we get for the same fixed $\Upsilon\in C_c^\infty([1-10^{-9},1+10^{-9}])$ as before,
\begin{align*}
    \mathcal{A}_k\ll&\frac{t^{\varepsilon}}{C^{3/2}\sqrt{N_j}} \sup_{\|\beta\|_2=1} \sup_{|u_1|\ll1}\Bigg\{\int_{|u_2|\ll1}\sum_{n_0}\sum_{q_0}\mathop{\sum\sum}_{q_{1,0},q_{2,0}|q_0^\infty}\sum_{q_0|n_1|q_0^\infty}\sum_r \sum_{Q_1Q_S\ll |q_1'|\ll Q_1Q_B}\mathop{\sum\cdots\sum}_{Q_S\ll |q_2'|,...,|q_\nu'|\ll Q_B}\nonumber\\
    &\times \mathop{\sum\sum\sum}_{\substack{n_2h_{1,0}h_{2,0} \square\text{-full}\\ h_{2,0}|(n_2h_{1,0})^\infty\\ (n_2,q_0)=1}}\mathop{\sum\sum\sum}_{\substack{N_S\ll n=n_1n_2n' \ll N_B\\ H'\leq h_1=h_{1,0}h_1', h_2=h_{2,0}h_2'<(1+10^{-10})H'\\ n'h_1' \square\text{-free}, (n_1'h_1',h_2')=1 \\(n'h_1'h_2',n_1h_{1,0}h_{2,0}q_0)=1}}\sum_{q_{k+1}=n_0q_0q_{2,0}q_{k+1}'r}\Psi\left(\frac{q_{k+1}}{C}\right)\nonumber\\
    &\times \mathop{\sum\sum}_{\substack{C\leq q_{k+2}=n_0q_0q_{2,0}q_{k+2}'r<(1+10^{-10})C\\C\leq q_{k+3}=n_0q_0q_{2,0}q_{k+3}'r<(1+10^{-10})C\\ q_{k+1}'\equiv q_{k+2}'\equiv q_{k+1}'\bmod{h_2'}}}\left(\frac{\sqrt{h_{1,0}h_{2,0}}C}{n_0^2n_1n'h_1'q_0^2q_{1,0}q_{2,0}r^2}\right)^{2^{k}}\left(\frac{n_0h_2'r\prod_{j=1}^{k-1}|q_j|^{2^{k-j}-1}}{n_1n_2^{3/2}h_{1,0}^{3/2}h_{2,0}q_{1,0}C}\right)^2\nonumber\\
    &\times n_1n_2^{3/2}n'h_{1,0}^{3/2}h_{2,0}h_1'q_0q_{1,0}q_{2,0}|q_1'\cdots q_k'|\beta'(q_{k+2},h_2,u_2)\overline{\beta'(q_{k+3},h_2,u_2)}\mathcal{D}_{k+1}(q_{k+1}',q_{k+2}')\nonumber\\
    &\times \overline{\mathcal{D}_{k+1}(q_{k+1}',q_{k+3}')}\overline{\mathcal{D}_0'(q_{k+2}',q_{k+3}')}I_{k+1}(q_1',...,q_k',q_{k+1},q_{k+2})\overline{I_{k+1}(q_1',...,q_k',q_{k+1},q_{k+3})}du_2\Bigg\}^{2^{-k-1}}.
\end{align*}

Note that $\mathcal{D}_k$ is defined $\bmod{n'h_1'}$. Applying Poisson summation on the $q_{k+1}'$-sum, we get \begin{align*}
    &\sum_{q_{k+1}'}\Upsilon\left(\frac{n_0q_0q_{2,0}q_{k+1}'r}{C}\right)\mathcal{D}_{k+1}(q_{k+1}',q_{k+2}')\overline{\mathcal{D}_{k+1}(q_{k+1}',q_{k+3}')}\delta(q_{k+1}'\equiv q_{k+2}'\equiv q_{k+3}'\bmod{h_2'})\\
    &\times I_{k+1}(q_1',...,q_k',q_{k+1},q_{k+2})\overline{I_{k+1}(q_1',...,q_k',q_{k+1},q_{k+3})}\\
    =&\frac{C}{n_0n'h_1'h_2'q_0q_{2,0}r}\sum_{q_{k+1}'}\sum_{\gamma \bmod{n'h_1'h_2'}}\mathcal{D}_{k+1}(\gamma,q_{k+2}')\overline{\mathcal{D}_{k+1}(\gamma,q_{k+3}')}\delta(\gamma\equiv q_{k+2}'\equiv q_{k+3}'\bmod{h_2'})e\left(\frac{\gamma q_{k+1}'}{n'h_1'h_2'}\right)\\
    &\times \int_0^\infty \Upsilon(y_{2^{k+1}-1})I_{k+1}(q_1',...,q_k',Cy_{2^{k+1}-1},q_{k+2})\\
    &\times \overline{I_{k+1}(q_1',...,q_k',Cy_{2^{k+1}-1},q_{k+3})}e\left(-\frac{q_{k+1}'Cy_{2^{k+1}-1}}{n_0n'h_1'h_2'q_0q_{2,0}r}\right)dy_{2^{k+1}-1}.
\end{align*}
For the character sum, the Chinese remainder theorem yields \begin{align*}
    &\sum_{\gamma \bmod{n'h_1'h_2'}}\mathcal{D}_{k+1}(\gamma,q_{k+2}')\overline{\mathcal{D}_{k+1}(\gamma,q_{k+3}')}\delta(\gamma\equiv q_{k+2}'\equiv q_{k+3}'\bmod{h_2'})e\left(\frac{\gamma q_{k+1}'}{n'h_1'h_2'}\right)\\
    =&e\left(\frac{q_{k+1'}q_{k+2}'\overline{n'h_1'}}{h_2'}\right)\delta(q_{k+2}'\equiv q_{k+3}'\bmod{h_2'})\sum_{\gamma\bmod{n'h_1'}}\mathcal{D}_{k+1}(\gamma,q_{k+2}')\overline{\mathcal{D}_{k+1}(\gamma,q_{k+3}')}e\left(\frac{\gamma q_{k+1}'\overline{h_2'}}{n'h_1'}\right)\\
    =&D(k+1)\mathcal{D}_{k+2}(q_{k+2}',q_{k+3}')\delta(q_{k+2}'\equiv q_{k+3}'\bmod{h_2'}).
\end{align*}
For the integral, note that the phase function in the $y_{2^{k+1}-1}$-integral is \begin{align*}
    (-1)^{k+1}\left(G(y_{2^{k-1}-k+1},y_{2^{k+1}-1})-G(y_{2^k+2^{k-1}-k},y_{2^{k+1}-1})\right)
    -\frac{q_{k+1}'Cy_{2^{k+1}-1}}{n_0n'h_1'h_2'q_0q_{2,0}r}.
\end{align*}
Hence repeated integration by parts gives us arbitrary savings unless $|q_{k+1}'|\ll Q_B$.

Let $\mathcal{B}_{k+1}$ be the contribution when $|q_{k+1}'|\ll Q_S$ including the case $q_{k+1}'=0$. Then the above analysis yields \begin{align*}
    S_{2}(N)^*\ll |\mathcal{A}_{k+1}|+\sum_{j=1}^{k}\mathcal{B}_j+\tilde{\mathcal{B}}_{k+1},
\end{align*}
with $\mathcal{A}_{k+1}, \mathcal{B}_1,...,\mathcal{B}_k$ as defined in the statement and \begin{align*}
    \tilde{\mathcal{B}}_{k+1}:=&\frac{t^{\varepsilon}}{C^{3/2}\sqrt{N_j}} \sup_{\|\beta\|_2=1} \sup_{|u_1|\ll1}\Bigg\{\int_{|u_2|\ll1}\sum_{n_0}\sum_{q_0}\mathop{\sum\sum}_{q_{1,0},q_{2,0}|q_0^\infty}\sum_{q_0|n_1|q_0^\infty}\sum_r \sum_{Q_1Q_S\ll |q_1'|\ll Q_1Q_B}\mathop{\sum\cdots\sum}_{Q_S\ll |q_2'|,...,|q_k'|\ll Q_B}\nonumber\\
    &\times \sum_{|q_{k+1}'|\ll Q_S}\mathop{\sum\sum\sum}_{\substack{n_2h_{1,0}h_{2,0} \square\text{-full}\\ h_{2,0}|(n_2h_{1,0})^\infty\\ (n_2,q_0)=1}}\mathop{\sum\sum\sum}_{\substack{N_S\ll n=n_1n_2n' \ll N_B\\ H'\leq h_1=h_{1,0}h_1', h_2=h_{2,0}h_2'<(1+10^{-10})H'\\ n'h_1' \square\text{-free}, (n_1'h_1',h_2')=1 \\(n'h_1'h_2',n_1h_{1,0}h_{2,0}q_0)=1}}\mathop{\sum\sum}_{\substack{C\leq q_{k+2}=n_0q_0q_{2,0}q_{k+2}'r<(1+10^{-10})C\\C\leq q_{k+3}=n_0q_0q_{2,0}q_{k+3}'r<(1+10^{-10})C\\ q_{k+2}'\equiv q_{k+3}'\bmod{h_2'}}}\nonumber\\
    &\times \left(\frac{\sqrt{h_{1,0}h_{2,0}}C}{n_0^2n_1n'h_1'q_0^2q_{1,0}q_{2,0}r^2}\right)^{2^{k}}\frac{n_0h_2'r}{n_1n_2^{3/2}h_{1,0}^{3/2}h_{2,0}q_{1,0}C}\prod_{j=1}^{k}|q_j|^{2^{k+1-j}-1}D(k+1)\beta'(q_{k+2},h_2,u_2)\nonumber\\
    &\times \overline{\beta'(q_{k+3},h_2,u_2)}\mathcal{D}_{k+2}(q_{k+2}',q_{k+3}')\mathcal{D}_0'(q_{k+2}',q_{k+3}')I_{k+2}(q_1',...,q_{k+1}',q_{k+2},q_{k+3})du_2\Bigg\}^{2^{-k-1}}.
\end{align*}
Applying Lemma \ref{TechnicalLemmaForDiagonal} with $\eta=0$ on the $y_{2^{k+1}-1}$-integral, then $\eta=1$ for $(2^k-2)$-times on the remaining integrals, followed by the final statement on the single integral left after the process, we get for $|q_{k+1}'|\ll Q_S$, \begin{align*}
    &I_{k+2}(q_1',...,q_{k+1}',q_{k+2},q_{k+3})\\
    \ll& \frac{C^2t^\varepsilon}{nN_j}\left(\frac{C}{(N_jN)^{1/3}}+ \frac{CN^{1/3}}{\sqrt{nt}N_j^{2/3}}\right)^{2^k-2} \delta\left(|q_{k+2}-q_{k+3}|\ll \left(\frac{C^2}{(N_jN)^{1/3}}+\frac{C^2N^{1/3}}{\sqrt{nt}N_j^{2/3}}\right)t^\varepsilon\right)+t^{-A}
\end{align*}
for any $A>0$. The fact that every use of Lemma \ref{TechnicalLemmaForDiagonal} will truncate the difference between two variables of integration and finally truncate the difference between $q_{k+2}', q_{k+3}'$ can be seen by noticing that every integral $y_i$ appears in some $G(\cdot,\cdot)$ in the phase function of the $(2^{k+1}-1)$-fold integral twice and no two of them appears in the same $G(\cdot,\cdot)$. Also, the way $q_{k+2},q_{k+3}$ is constructed through iterations of Cauchy-Schwarz inequality shows that the truncation between them occurs at last.

On the other hand, notice that the character sum $\mathcal{D}_{k+2}$ is exactly of the form defined before Lemma \ref{CharSumLemmaGeneral}. Indeed, we have \begin{align*}
    &\mathcal{D}_{k+2}(q_{k+2}',q_{k+3}')\\
    =\, &\mathcal{C}_{k-1}(a_{I_{k-1}},q_1'\overline{n_1n_2h_{1,0}h_2q_0q_{1,0}q_{2,0}},\eta_jq_{1,0}\overline{n_1n_2h_2q_{2,0}},\eta_jq_{2,0}\overline{n_1n_2h_{1,0}q_{1,0}},n'h_1',q_{k+2}',q_{k+3}'),
\end{align*}
with $a_{I_{k-1}}$ being an appropriate sequence of numbers consisting of $q_j'\overline{h_2'}$.
Hence Lemma \ref{CharSumLemmaGeneral} gives us \begin{align*}
    \mathcal{D}_{k+2}(q_{k+2}',q_{k+3}')\ll (n'h_1')^{2^k-1/2}t^\varepsilon\left(\sqrt{n'h_1'}\delta(*_{k-1})+1\right),
\end{align*}
where $*_{k-1}$ is the condition \begin{align*}
    *_{k-1}=\Bigg\{& q_{k+2}'\equiv q_{k+3}'\bmod n'h_1'\nonumber\\
    &\text{or } q_1'\overline{q_0q_{1,0}}\pm h_1q_{1,0}(\overline{q_{k+2}'}-\overline{q_{k+3}'})\equiv 0 \bmod n'h_1'\nonumber\\
    &\text{or } q_1'\overline{q_0q_{1,0}}-\eta_jh_1q_{1,0}\overline{q_{k+2}'}\equiv 0\bmod n'h_1'\nonumber\\
    &\text{or } q_1'\overline{q_0q_{1,0}}-\eta_jh_1q_{1,0}\overline{q_{k+3}'}\equiv 0\bmod n'h_1'\Bigg\}.
\end{align*}
This proves the last statement of the Lemma.

Together with $\mathcal{D}(k+1)\ll1$ and $\mathcal{D}_0\ll n_1n_2h_{1,0}h_{2,0}q_0^3q_{1,0}q_{2,0}$ as in (\ref{D0Def}), this yields \begin{align*}
    \tilde{\mathcal{B}}_{k+1}\ll &\frac{t^{\varepsilon}}{C^{3/2}\sqrt{N_j}} \sup_{\|\beta\|_2=1} \sup_{|u_1|\ll1}\Bigg\{\int_{|u_2|\ll1}\sum_{n_0}\sum_{q_0}\mathop{\sum\sum}_{q_{1,0},q_{2,0}|q_0^\infty}\sum_{q_0|n_1|q_0^\infty}\sum_r \sum_{Q_1Q_S\ll |q_1'|\ll Q_1Q_B}\mathop{\sum\cdots\sum}_{Q_S\ll |q_2'|,...,|q_k'|\ll Q_B}\nonumber\\
    &\times \sum_{|q_{k+1}'|\ll Q_S}\mathop{\sum\sum\sum}_{\substack{n_2h_{1,0}h_{2,0} \square\text{-full}\\ h_{2,0}|(n_2h_{1,0})^\infty\\ (n_2,q_0)=1}}\mathop{\sum\sum\sum}_{\substack{N_S\ll n=n_1n_2n' \ll N_B\\ H'\leq h_1=h_{1,0}h_1', h_2=h_{2,0}h_2'<(1+10^{-10})H'\\ n'h_1' \square\text{-free}, (n_1'h_1',h_2')=1 \\(n'h_1'h_2',n_1h_{1,0}h_{2,0}q_0)=1}}\mathop{\sum\sum}_{\substack{C\leq q_{k+2}=n_0q_0q_{2,0}q_{k+2}'r<(1+10^{-10})C\\C\leq q_{k+3}=n_0q_0q_{2,0}q_{k+3}'r<(1+10^{-10})C\\ q_{k+2}'\equiv q_{k+3}'\bmod{h_2'}\\|q_{k+2}-q_{k+3}|\ll\left(\frac{C^2}{(N_jN)^{1/3}}+\frac{C^2N^{1/3}}{\sqrt{nt}N_j^{2/3}}\right)t^\varepsilon}}\nonumber\\
    &\times \left(\frac{\sqrt{h_{1,0}h_{2,0}}C}{n_0^2n_1n'h_1'q_0^2q_{1,0}q_{2,0}r^2}\right)^{2^{k}}\frac{n_0h_2'r}{n_1n_2^{3/2}h_{1,0}^{3/2}h_{2,0}q_{1,0}C}\prod_{j=1}^{k}|q_j|^{2^{k+1-j}-1}|\beta(q_{k+2},h_2,u_2)|\nonumber\\
    &\times |\beta'(q_{k+3},h_2,u_2)|n_1n_2h_{1,0}h_{2,0}q_0^3q_{1,0}q_{2,0}(n'h_1')^{2^k-1/2}t^\varepsilon\left(\sqrt{n'h_1'}\delta(*_{k-1})+1\right)\nonumber\\
    &\times \frac{C^2}{nN_j}\left(\frac{C}{(N_jN)^{1/3}}+ \frac{CN^{1/3}}{\sqrt{nt}N_j^{2/3}}\right)^{2^k-2}du_2\Bigg\}^{2^{-k-1}}.
\end{align*}
Recall the definition of $Q_S, Q_B$ in (\ref{QsDef}) and (\ref{Q1QBDef}). Applying the AM-GM inequality \begin{align*}
    |\beta(q_{k+2},h_2,u_2)\beta(q_{k+3},h_2,u_2)|\ll |\beta(q_{k+2},h_2,u_2)|^2+|\beta(q_{k+3},h_2,u_2)|^2,
\end{align*}
the same counting argument as in \eqref{DetailedCountingComputation} yields \begin{align*}
    \tilde{\mathcal{B}}_{k+1}\ll& \frac{t^{\varepsilon}}{C^{3/2}\sqrt{N_j}}\Bigg\{ \frac{C^{2^k+1}}{N_j}\sum_{n\ll \frac{CN^{1/3}t^\varepsilon}{N_j^{2/3}}} n^{-3/2}\left(\frac{nN_j^{1/3}t^2}{N^{5/3}}\right)^{2^{k+1}-2} \left(1+\frac{C^3t^2}{n(N_jN)^{4/3}}\right)\sqrt{\frac{Ct}{N}}\\
    &\times \left(\underbrace{\frac{Ct}{N}}_{q_{k+2}=q_{k+3}}+\underbrace{\frac{C^2}{(N_jN)^{1/3}}+\frac{C^2N^{1/3}}{\sqrt{nt}N_j^{2/3}}}_{q_{k+2}\neq q_{k+3}}\right)\left(\frac{C}{(N_jN)^{1/3}}+ \frac{CN^{1/3}}{\sqrt{nt}N_j^{2/3}}\right)^{2^k-2}\Bigg\}^{2^{-k-1}}\\
    \ll&\frac{\sqrt{C}t^{2+\varepsilon}}{N_jN^{3/2}} \Bigg\{\frac{N_j^{2/3}N^{8/3}}{C^3t^{7/2}}\left(1+\frac{C^2t^2}{N_j^{2/3}N^{5/3}}\right)\left(\frac{Ct}{N}+\frac{C^2}{(N_jN)^{1/3}}\right)\Bigg\}^{2^{-k-1}}.
\end{align*}

\subsection{Proof of Lemma \ref{InuBound}}\label{sect.InuBoundProof}

Recall that such an integral is defined recursively by \begin{align*}
    I_{\nu+1}(q_1',...,q_\nu',q_{\nu+1},q_{\nu+2})=&\int_0^\infty\Upsilon(y_{2^\nu-1})I_{\nu}(q_1',...,q_{\nu-1}',Cy_{2^\nu-1},q_{\nu+1})\\
    &\times\overline{I_{\nu}(q_1',...,q_{\nu-1}',Cy_{2^\nu-1},q_{\nu+2})}e\left(-\frac{q_{\nu}'Cy_{2^\nu-1}}{n_0n'h_1h_2q_0q_{2,0}r}\right)dy_{2^\nu-1}
\end{align*}
with \begin{align*}
    &I_3(q_1',q_2',q_3,q_4)=\int_0^\infty \Upsilon(y_3)I_2\left(q_1',Cy_3,q_3,q_4\right)e\left(-\frac{q_2'Cy_3}{n_0n'h_1h_2q_0q_{2,0}r}\right)dy_3\\
    =&\int_0^\infty \Upsilon(y_3)\int_0^\infty\Upsilon(y_1)\Psi(Cy_1,Cy_3,q_3)e\left(G(y_1,y_3)-G\left(y_1,\frac{q_3}{C}\right)-\frac{q_1'Cy_1}{n_0nh_1h_2q_0^2q_{1,0}^2q_{2,0}r}\right)dy_1\nonumber\\
    &\times \int_0^\infty\Upsilon(y_2)\overline{\Psi(Cy_2,Cy_3,q_4)}e\left(-G(y_2,y_3)+G\left(y_2,\frac{q_4}{C}\right)+\frac{q_1'Cy_2}{n_0nh_1h_2q_0^2q_{1,0}^2q_{2,0}r}\right)dy_2\nonumber\\
    &\times e\left(-\frac{q_2'Cy_3}{n_0n'h_1h_2q_0q_{2,0}r}\right)dy_3.
\end{align*}
Recall the definition of $G$ in (\ref{GDef}). Applying Taylor expansion, we get 
\begin{align*}
    G(y,z)=\tilde{G}(y,z)+\theta(y,z),
\end{align*}
where \begin{align*}
    \tilde{G}(y,z)=\sqrt{\frac{2t}{\pi n}}\left(-\eta_j' \left(\frac{z}{h_1y}\right)^{1/3}+\eta_j'\left(\frac{y}{h_2z}\right)^{1/3}\right)^{3/2}.
\end{align*}
and
\begin{align*}
        \theta(y,z)=&\sqrt{\frac{2t}{\pi n}}\sum_{j=1}^\infty a_j \left(-\eta_j' \left(\frac{z}{h_1y}\right)^{1/3}+\eta_j'\left(\frac{y}{h_2z}\right)^{1/3}\right)^{3/2-j}\\
        &\times \left(-\eta_j' \left(\frac{z}{h_1y}\right)^{1/3}\sum_{\ell=1}^\infty b_{\ell}\left(\frac{2\pi h_1N^2u_1}{CHt^{2-\varepsilon}y}\right)^\ell+\eta_j'\left(\frac{y}{h_2z}\right)^{1/3}\sum_{\ell=1}^\infty b_\ell\left(\frac{2\pi h_2N^2u_2}{CHt^{2+\varepsilon}z}\right)^\ell\right)^j
    \end{align*}
for some constants $a_j,b_\ell\in\R$. Then $e\left(\theta(y,z)\right)$ is $\frac{N^{3/2}}{\sqrt{Cn}Ht^{1-\varepsilon}}$-inert and $I_3$ is simplified to \begin{align*}
    &I_3(q_1',q_2',q_3,q_4)\\
    =&\int_0^\infty \Upsilon(y_3)\int_0^\infty\Upsilon(y_1)\tilde{\Psi}(Cy_1,Cy_3,q_3)e\left(\tilde{G}(y_1,y_3)-\tilde{G}\left(y_1,\frac{q_3}{C}\right)-\frac{q_1'Cy_1}{n_0nh_1h_2q_0^2q_{1,0}^2q_{2,0}r}\right)dy_1\nonumber\\
    &\times \int_0^\infty\Upsilon(y_2)\overline{\tilde{\Psi}(Cy_2,Cy_3,q_4)}e\left(-\tilde{G}(y_2,y_3)+\tilde{G}\left(y_2,\frac{q_4}{C}\right)+\frac{q_1'Cy_2}{n_0nh_1h_2q_0^2q_{1,0}^2q_{2,0}r}\right)dy_2\\
    &\times e\left(-\frac{q_2'Cy_3}{n_0n'h_1h_2q_0q_{2,0}r}\right)dy_3,
\end{align*}
where \begin{align*}
    \tilde{\Psi}\left(x,y,z\right)=\Psi\left(x,y,z\right)e\left(\theta(x,y)-\theta(x,z)\right)
\end{align*}
is $\left(1+\frac{CN^{1/3}}{nN_j^{2/3}}+\frac{N^{3/2}}{\sqrt{Cn}Ht}\right)t^\varepsilon$-inert by (\ref{z_0DiffBound}) and above analysis on $\theta$.

Write $q_1''=\frac{q_1'}{Q_1}=\frac{q_1'}{n_1n_2h_{1,0}h_{2,0}q_0q_{1,0}^2}$ and $q_j''=q_j'$ for all $2\leq j\leq \nu$. Note that $n\gg N_S$ and $|q_j''|\gg Q_S$, this allows us to apply Lemma \ref{S2IntegralLemma} ($2^\nu-1$)-times to bound the integral. (The lower bounds for $q_j'$ and $n$ are needed so that the oscillation of $\tilde{\Psi}$ satisfies the condition in Lemma \ref{S2IntegralLemma}.) For $1\leq j\leq \nu$, we choose \begin{align*}
    U_j=&\left(1+\delta(j>1)\left(1+\frac{(N_jN)^{1/3}n_0n'h_1h_2q_0q_{2,0}r}{\sqrt{\min\left\{q_u'':1\leq u\leq j-1\right\}q_j''}C^2}\right)\right)\sqrt{\frac{n_0n'h_1h_2q_0q_{2,0}r}{q_j''C}}
\end{align*}
as the parameter $U$ in Lemma \ref{S2IntegralLemma} on the $2^{\nu-j}$ integrals with $q_j'$. We first apply $1$ time on the integral with $q_\nu'$ , then $2$-times on the $2$ integrals with $q_2'$ and go on until we apply $2^{\nu-1}$-times on the $2^{\nu-1}$ integrals with $q_1'$. Such choices of $U$ are made so that the conditions of Lemma \ref{S2IntegralLemma} is fulfilled in the next iteration.

\begin{Remark}
    Explanation on the choice of $U_j$: Suppose we have applied Lemma \ref{S2IntegralLemma} on all the integrals involving $q_\nu'',...,q_k'$. Then a number of functions of the form $W_0\left(\frac{y_u-y_0}{U_j}\right)$ is created, and they are $\frac{t^\varepsilon}{U_j}\left(1+\frac{(N_jN)^{1/3}}{Cq_j''Z}\right)$-inert by Lemma \ref{S2IntegralLemma}, where $Z=\frac{C}{n_0n'h_1h_2q_0q_{2,0}r}$. To apply Lemma \ref{S2IntegralLemma} on the integrals involving $q_{k-1}'$, we need to make sure $U_\nu,...,U_k$ are chosen such that \begin{align*}
        q_{k-1}''Z\gg \frac{t^\varepsilon}{U_j^2}\left(1+\frac{(N_jN)^{1/3}}{Cq_j''Z}\right)^2
    \end{align*}
    for all $k\leq j\leq \nu$.
\end{Remark}

Applying the above procedure and bounding the resulting integral trivially yields \begin{align*}
    I_{\nu+1}(q_1',...,q_\nu',q_{\nu+1},q_{\nu+2})\ll_\nu \,& U_1^{2^{\nu-1}}U_2^{2^{\nu-2}}\cdots U_\nu t^\varepsilon\\
    \ll_\nu \, &\prod_{j=2}^\nu\left(1+\frac{(N_jN)^{1/3}n_0n'h_1h_2q_0q_{2,0}r}{\sqrt{\min\left\{q_u'':1\leq u\leq j-1\right\}q_j''}C^2}\right)^{2^{\nu-j}}\\
    &\times \frac{(n_1q_0q_{1,0}^2)^{2^{\nu-2}}(n_0n'h_1h_2q_0q_{2,0}r)^{2^{\nu-1}-1/2}}{q_1'^{2^{\nu-2}}q_2'^{2^{\nu-3}}\cdots q_\nu'^{2^{-1}}C^{2^{\nu-1}-1/2}}t^\varepsilon
\end{align*}
for any $\nu\geq2$.

\bibliographystyle{abbrv}%{amsplain}
\bibliography{ref}

\begin{thebibliography}{10}

\bibitem{Adolphson-Sperber}
A.~Adolphson and S.~Sperber.
\newblock Twisted exponential sums and {N}ewton polyhedra.
\newblock {\em J. Reine Angew. Math.}, 443:151--177, 1993.

\bibitem{Agg2018published}
K.~Aggarwal.
\newblock Weyl bound for {GL}(2) in {$t$}-aspect via a simple delta method.
\newblock {\em J. Number Theory}, 208:72--100, 2020.

\bibitem{Agg-IJNT}
K.~Aggarwal.
\newblock A new subconvex bound for {$\mathrm{GL}(3)$} {$L$}-functions in the
  {$t$}-aspect.
\newblock {\em Int. J. Number Theory}, 17(5):1111--1138, 2021.

\bibitem{Bernstein-Reznikov}
J.~Bernstein and A.~Reznikov.
\newblock Subconvexity bounds for triple {$L$}-functions and representation
  theory.
\newblock {\em Ann. of Math. (2)}, 172(3):1679--1781, 2010.

\bibitem{Blo12}
V.~Blomer.
\newblock Subconvexity for twisted {$L$}-functions on {${\mathrm{GL}}(3)$}.
\newblock {\em Amer. J. Math.}, 134(5):1385--1421, 2012.

\bibitem{Blomer-Jana-Nelson}
V.~Blomer, S.~Jana, and P.~D. Nelson.
\newblock The weyl bound for triple product {$L$}-functions.
\newblock {\em Duke Mathematical Journal}, 172(6):1173--1234, 2023.

\bibitem{BKY}
V.~Blomer, R.~Khan, and M.~Young.
\newblock Distribution of mass of holomorphic cusp forms.
\newblock {\em Duke Math. J.}, 162(14):2609--2644, 2013.

\bibitem{Bo-Mi-Ng}
A.~R. Booker, M.~B. Milinovich, and N.~Ng.
\newblock Subconvexity for modular form {$L$}-functions in the {$t$} aspect.
\newblock {\em Adv. Math.}, 341:299--335, 2019.

\bibitem{bourgain}
J.~Bourgain.
\newblock Decoupling, exponential sums and the {R}iemann zeta function.
\newblock {\em J. Amer. Math. Soc.}, 30(1):205--224, 2017.

\bibitem{Conrey-Iwaniec}
J.~B. Conrey and H.~Iwaniec.
\newblock The cubic moment of central values of automorphic {$L$}-functions.
\newblock {\em Ann. of Math. (2)}, 151(3):1175--1216, 2000.

\bibitem{goldfeldbook}
D.~Goldfeld.
\newblock {\em Automorphic forms and {$L$}-functions for the group
  {${\mathrm{GL}}(n,\textbf{R})$}}, volume~99 of {\em Cambridge Studies in
  Advanced Mathematics}.
\newblock Cambridge University Press, Cambridge, 2006.
\newblock With an appendix by Kevin A. Broughan.

\bibitem{Good}
A.~Good.
\newblock The square mean of {D}irichlet series associated with cusp forms.
\newblock {\em Mathematika}, 29(2):278--295 (1983), 1982.

\bibitem{graham_kolesnik_1991}
S.~W. Graham and G.~Kolesnik.
\newblock {\em Van der Corput's Method of Exponential Sums}.
\newblock London Mathematical Society Lecture Note Series. Cambridge University
  Press, 1991.

\bibitem{HB-twelfth-moment}
D.~R. Heath-Brown.
\newblock {The twelfth power moment of the Riemann-function}.
\newblock {\em The Quarterly Journal of Mathematics}, 29(4):443--462, 12 1978.

\bibitem{Ho-Mu-Qi-2021arxiv}
R.~Holowinsky, R.~Munshi, and Z.~Qi.
\newblock Beyond the weyl barrier for {$\mathrm{GL}(2)$} exponential sums.
\newblock {\em Advances in Mathematics}, 426:109099, 2023.

\bibitem{Huang1}
B.~Huang.
\newblock Hybrid subconvexity bounds for twisted {$L$}-functions on
  {${\mathrm{GL}}(3)$}.
\newblock {\em Science China Mathematics}, 64:443--478, 2021.

\bibitem{Ivic2015}
A.~Ivi\'{c}.
\newblock Hybrid moments of the {R}iemann zeta-function.
\newblock {\em J. Numbers}, pages Art. ID 892324, 14, 2015.

\bibitem{Iwaniec1979}
H.~Iwaniec.
\newblock Fourier coefficients of cusp forms and the riemann zeta-function.
\newblock {\em Seminaire de Théorie des Nombres de Bordeaux}, 9:1--36,
  1979-1980.

\bibitem{Iw-Ko}
H.~Iwaniec and E.~Kowalski.
\newblock {\em Analytic number theory}, volume~53 of {\em American Mathematical
  Society Colloquium Publications}.
\newblock American Mathematical Society, Providence, RI, 2004.

\bibitem{Jut87}
M.~Jutila.
\newblock {\em Lectures on a method in the theory of exponential sums},
  volume~80 of {\em Tata Institute of Fundamental Research Lectures on
  Mathematics and Physics}.
\newblock Published for the Tata Institute of Fundamental Research, Bombay; by
  Springer-Verlag, Berlin, 1987.

\bibitem{khan_young_2021}
R.~Khan and M.~P. Young.
\newblock Moments and hybrid subconvexity for symmetric-square l-functions.
\newblock {\em Journal of the Institute of Mathematics of Jussieu}, page
  1–45, 2021.

\bibitem{Kiral-Petrow-Young}
E.~M. Kiral, I.~Petrow, and M.~P. Young.
\newblock Oscillatory integrals with uniformity in parameters.
\newblock {\em J. Th\'{e}or. Nombres Bordeaux}, 31(1):145--159, 2019.

\bibitem{landau}
E.~Landau.
\newblock \"{U}ber die {W}urzeln der {Z}etafunktion.
\newblock {\em Math. Z.}, 20(1):98--104, 1924.

\bibitem{Li1}
X.~Li.
\newblock Bounds for {${\mathrm{GL}}(3)\times {\mathrm{GL}}(2)$}
  {$L$}-functions and {${\mathrm{GL}}(3)$} {$L$}-functions.
\newblock {\em Ann. of Math. (2)}, 173(1):301--336, 2011.

\bibitem{Lin-Nunes-Qi-arxiv}
Y.~{Lin}, R.~{Nunes}, and Z.~{Qi}.
\newblock {Strong subconvexity for self-dual $\mathrm{GL} (3)$ $L$-functions}.
\newblock {\em International Mathematics Research Notices},
  2023(13):11453--11470, 2023.

\bibitem{McKee-Sun-Ye}
M.~McKee, H.~Sun, and Y.~Ye.
\newblock Improved subconvexity bounds for {$GL(2)\times GL(3)$} and {$GL(3)$}
  {$L$}-functions by weighted stationary phase.
\newblock {\em Trans. Amer. Math. Soc.}, 370(5):3745--3769, 2018.

\bibitem{Meurman}
T.~Meurman.
\newblock On exponential sums involving the {F}ourier coefficients of {M}aass
  wave forms.
\newblock {\em J. Reine Angew. Math.}, 384:192--207, 1988.

\bibitem{Michel-Venkatesh}
P.~Michel and A.~Venkatesh.
\newblock The subconvexity problem for {${\mathrm{GL}}_2$}.
\newblock {\em Publ. Math. Inst. Hautes \'Etudes Sci.}, (111):171--271, 2010.

\bibitem{Miller-Schmid}
S.~D. Miller and W.~Schmid.
\newblock Automorphic distributions, {$L$}-functions, and {V}oronoi summation
  for {${\mathrm{GL}}(3)$}.
\newblock {\em Ann. of Math. (2)}, 164(2):423--488, 2006.

\bibitem{Mun3}
R.~Munshi.
\newblock The circle method and bounds for {$L$}-functions---{III}:
  {$t$}-aspect subconvexity for {$GL(3)$} {$L$}-functions.
\newblock {\em J. Amer. Math. Soc.}, 28(4):913--938, 2015.

\bibitem{Mun-JEMS2022}
R.~Munshi.
\newblock Subconvexity for {$\mathrm{GL}(3)\times GL(2)$} {$L$}-functions in
  {$t$}-aspect.
\newblock {\em J. Eur. Math. Soc. (JEMS)}, 24(5):1543--1566, 2022.

\bibitem{Nelson21}
P.~D. {Nelson}.
\newblock {Bounds for standard $L$-functions}.
\newblock {\em preprint, arXiv:2109.15230}, 2021.

\bibitem{Nunes2017}
R.~M. {Nunes}.
\newblock {Subconvexity for $\mathrm{GL}(3)$ L-functions}.
\newblock {\em preprint, arXiv 1703.04424}, 2017.

\bibitem{Qi20}
Z.~{Qi}.
\newblock {Subconvexity for $L$-Functions on $\mathrm{GL}_3$ over Number
  Fields}.
\newblock {\em J. Eur. Math. Soc. (JEMS)}, 26(3):1113--1192, 2024.

\bibitem{Young-second-moment}
M.~P. Young.
\newblock The second moment of gl(3)×gl(2) l-functions, integrated.
\newblock {\em Advances in Mathematics}, 226(4):3550--3578, 2011.

\end{thebibliography}

\end{document}